\documentclass[fontsize=11pt]{amsart} 

\makeatletter 
\def\blfootnote{\gdef\@thefnmark{}\@footnotetext}
\makeatother

\usepackage[T1]{fontenc} 
\usepackage[english]{babel} 
\usepackage{amsmath,amsfonts,amsthm, verbatim, textcomp, amssymb,mathtools, empheq, color, multicol,units, xfrac,euscript,tikz,graphicx} 
\usepackage{enumitem}
\usepackage{hyperref}
\usepackage{tikz-cd}
\usepackage{subfiles}

\input xy
\usepackage[all]{xy}
\xyoption{line}
\xyoption{arrow}
\xyoption{color}
\SelectTips{cm}{}

\tikzset{
    partial ellipse/.style args={#1:#2:#3}{
        insert path={+ (#1:#3) arc (#1:#2:#3)}
    }
}

\usetikzlibrary{backgrounds}

\numberwithin{equation}{section} 
\numberwithin{figure}{section} 
\numberwithin{table}{section} 

\newcommand{\into}{\hookrightarrow}

\newcommand{\Hom}{\text{Hom}}

\newcommand{\Ker}{\text{Ker}}

\newcommand{\Tr}{\operatorname{Tr}}

\newcommand{\Sym}{\text{Sym}}

\DeclareMathOperator{\ob}{Ob}

\newcommand{\mf}[1]{\mathfrak{#1}}
\newcommand{\mc}[1]{\mathcal{#1}}

\newcommand{\ZZ}{\mathbb{Z}}
\newcommand{\RR}{\mathbb{R}}

\newcommand{\NN}{\mathbb{N}}

\newcommand{\KK}{\mathbb{K}}
\newcommand{\QQ}{\mathbb{Q}}

\newcommand{\WW}{\mathbb{W}}

\newcommand{\Strat}{\mathbb{S}}
\newcommand{\kk}{\Bbbk}

\newcommand{\lam}[1]{\lambda^{(#1)}}
\newcommand{\lamtild}{\tilde{\lambda}}
\newcommand{\mutild}{\tilde{\mu}}

\newcommand{\g}{\mathfrak{g}}
\newcommand{\h}{\mathfrak{h}}
\newcommand{\gtild}{\tilde{\mathfrak{g}}}

\newcommand{\UEAdotZg}{\dot{U}^{\ZZ}(\mf{g})}
\newcommand{\UEAdotZgtild}{\dot{U}^{\ZZ}(\gtild)}

\newcommand{\QGdotZ}{\dot{U}_q^{\ZZ}}

\newcommand{\VZq}[1]{V_q^{\ZZ}(#1)}
\newcommand{\VZqulam}{V_q^{\ZZ}(\ulambda)}

\newcommand{\VZ}[1]{V^{\ZZ}(#1)}
\newcommand{\VZulam}{V^{\ZZ}(\ulambda)}
\newcommand{\VZlamtild}{V^{\ZZ}(\lamtild)}

\newcommand{\curr}{U(\mf{g}[t])}
\newcommand{\currdot}{\dot{U}(\mf{g}[t])}
\newcommand{\currk}{U_{\kk}(\g[t])}
\newcommand{\currdotk}{\dot{U}_{\kk}(\g[t])}
\newcommand{\currK}{U_{\KK}(\g[t])}
\newcommand{\currdotK}{\dot{U}_{\KK}(\g[t])}
\newcommand{\currtild}{U_{\KK}(\gtild[t])}
\newcommand{\currdottild}{\dot{U}_{\KK}(\gtild[t])}

\newcommand{\WZtild}{W_{\KK}^{Z}(\lamtild)}

\newcommand{\Wulam}{W(\ulambda)}
\newcommand{\Wulamk}{W_{\kk}(\ulambda)}

\newcommand{\WWulam}{\WW(\ulambda)}
\newcommand{\WWulamk}{\WW_{\kk}(\ulambda)}

\newcommand{\bA}{\mathbb{A}}
\newcommand{\ATP}{\mathbb{A}_{\ulambda}}

\newcommand{\Rtild}{\tilde{R}}
\newcommand{\Rhat}{\hat{R}}
\newcommand{\Rtildhat}{\hat{\tilde{R}}}

\newcommand{\varep}{\varepsilon}

\newcommand{\catQG}{\mc{U}}
\newcommand{\catQGdot}{\dot{\mc{U}}}
\newcommand{\catQGk}{\dot{\mc{U}}_{\kk}(\g)}

\newcommand{\catQGstar}{\mc{U}^*}
\newcommand{\catQGdotstar}{\dot{\mc{U}}^*}
\newcommand{\catQGstark}{\dot{\mc{U}}_{\kk}^*(\g)}
\newcommand{\catQGstarK}{\dot{\mc{U}}_{\KK}^*(\g)}
\newcommand{\catQGstarKtild}{\dot{\mc{U}}_{\KK}^*(\gtild)}

\newcommand{\E}{\mc{E}}

\newcommand{\cyclo}[1]{\mc{U}^{#1}}
\newcommand{\cyclolam}{\mc{U}^{\lambda}}
\newcommand{\cycloulam}{\mc{U}^{\ulambda}}
\newcommand{\cyclolamk}{\mc{U}_{\kk}^{\lambda}(\g)}

\newcommand{\defcyclolam}{\check{\mc{U}}^{\lambda}}
\newcommand{\defcyclolamk}{\check{\mc{U}}_{\kk}^{\lambda}(\g)}

\newcommand{\cyclostar}[1]{\mc{U}^{#1,*}}
\newcommand{\cyclolamstar}{\mc{U}^{\lambda,*}}
\newcommand{\cycloulamstar}{\mc{U}^{\ulambda,*}}
\newcommand{\cyclolamstark}{\mc{U}_{\kk}^{\lambda,*}(\g)}

\newcommand{\defcyclolamstar}{\check{\mc{U}}^{\lambda,*}}

\newcommand{\cyclostartild}{\mc{U}_{\KK}^{\lamtild,*}(\gtild)}

\newcommand{\TPaux}{\widetilde{\mc{X}}^{\ulambda}}
\newcommand{\TPauxstar}{\widetilde{\mc{X}}^{\ulambda,*}}

\newcommand{\TP}{\mc{X}^{\ulambda}}
\newcommand{\TPstar}{\mc{X}^{\ulambda,*}}

\newcommand{\defTP}{\check{\mc{X}}^{\ulambda}}
\newcommand{\defTPstar}{\check{\mc{X}}^{\ulambda,*}}

\newcommand{\defTPlam}{\check{\mc{X}}^{\lambda}}
\newcommand{\defTPlamstar}{\check{\mc{X}}^{\lambda,*}}

\newcommand{\genTP}{\mc{G}^{\ulambda}}

\newcommand{\TrTP}{\Tr(\mc{X}^{\ulambda})}
\newcommand{\TrTPstar}{\Tr(\mc{X}^{\ulambda,*})}

\newcommand{\TrdefTP}{\Tr(\check{\mc{X}}^{\ulambda})}
\newcommand{\TrdefTPstar}{\Tr(\check{\mc{X}}^{\ulambda,*})}

\newcommand{\TrgenTP}{\Tr(\mc{G}^{\ulambda})}
\newcommand{\TrZgenTP}{\Tr^{Z}(\mc{G}^{\ulambda})}

\newcommand{\Itild}{\tilde{I}}
\newcommand{\Xtild}{\tilde{X}}
\newcommand{\eptild}{\tilde{\varepsilon}}
\newcommand{\ytild}{\tilde{y}}
\newcommand{\psitild}{\tilde{\psi}}
\newcommand{\vtild}{\tilde{v}}

\newcommand{\EndQ}{\TPstar(Q_m,Q_m)}
\newcommand{\EndP}{\cycloulamstar(P_{\nu},P_{\nu})}

\newcommand{\ulambda}{\underline{\lambda}}
\newcommand{\uz}{\underline{z}}
\newcommand{\ui}{\underline{i}}
\newcommand{\uj}{\underline{j}}

\theoremstyle{definition}

\theoremstyle{remark}
\newtheorem{remark}{Remark}[section]

\theoremstyle{plain}
\newtheorem{definition}[remark]{Definition}
\newtheorem{corollary}[remark]{Corollary}
\newtheorem{proposition}[remark]{Proposition}
\newtheorem{lemma}[remark]{Lemma}
\newtheorem{theorem}[remark]{Theorem}
\newtheorem{introtheorem}{Theorem}

\newtheorem{introprop}[introtheorem]{Proposition}

 \usetikzlibrary{decorations.pathreplacing,backgrounds,decorations.markings}
\tikzset{wei/.style={draw=red,double=red!40!white,double distance=1.5pt,thin}}
\tikzset{bdot/.style={fill,circle,color=blue,inner sep=3pt,outer sep=0}}
\tikzset{dir/.style={postaction={decorate,decoration={markings,
    mark=at position .8 with {\arrow[scale=1.3] {}}}}}}
\tikzset{rdir/.style={postaction={decorate,decoration={markings,
    mark=at position .8 with {\arrow[scale=1.3]{>}}}}}}
\tikzset{edir/.style={postaction={decorate,decoration={markings,
    mark=at position .2 with {\arrow[scale=1.3] {}}}}}}

\begin{document}

\blfootnote{2010 \textit{Mathematics Subject Classification}: 17B10.

\textit{Key words and phrases}: current algebra, categorifed quantum group, tensor product algebra}

\title{Trace decategorification of tensor product algebras}
\author{Christopher Leonard and Michael Reeks}
\maketitle

\begin{abstract}
    We show that in ADE type the trace of Webster's categorification of a tensor product of irreducibles for the quantum group is isomorphic to a tensor product of Weyl modules for the current algebra $\dot{U}(\mf{g}[t])$. This extends a result of Beliakova, Habiro, Lauda, and Webster who showed that the trace of the categorified quantum group $\dot{\mc{U}}^*(\mf{g})$ is isomorphic to $\dot{U}(\mf{g}[t])$, and the trace of a cyclotomic quotient of $\dot{\mc{U}}^*(\mf{g})$, which categorifies a single irreducible for the quantum group, is isomorphic to a Weyl module for $\dot{U}(\mf{g}[t])$. We use a deformation argument based on Webster's technique of unfurling 2-representations.
\end{abstract}

\tableofcontents

\section{Introduction}

Categorification is the process of enriching algebraic objects by increasing their categorical dimension by one, for example by passing from a set to a category. Many categories have been constructed whose split Grothendieck groups are, by design, isomorphic to important objects in Lie theory. Recently  following \cite{BGHL14}, there has been an effort to determine the trace decategorifications of these categories. For example, the traces of the Heisenberg category, diagrammatic Hecke category, and categorified quantum group have been identified with the W-algebra $W_{1+\infty}$ (\cite{CLLS18}), the semidirect product of the Weyl group and a polynomial algebra (\cite{EL16}), and the corresponding current algebra (\cite{BHLZ14} and \cite{BHLW17}), respectively.

The \emph{trace} of a $\kk$-linear category $\mc{C}$, denoted $\Tr(\mc{C})$, is the $\kk$-vector space
\begin{equation*}
	\Tr(\mathcal{C})= \bigg( \bigoplus_{x\in \ob(\mathcal{C})} \mc{C}(x,x) \bigg) \big/\operatorname{span}_{\Bbbk}\{fg-gf\}, 
\end{equation*}

\noindent with the span taken over all $f:x\rightarrow y$ and $g: y\rightarrow x$ for $x,y\in \ob(\mc{C})$. The trace of a category is invariant under passage to the Karoubi envelope, which makes it particularly convenient for diagrammatic categories.

The trace and split Grothendieck group are related by the \emph{Chern character map} $h_{\mc{C}}: K^{\kk}_0(\mc{C}) \rightarrow \Tr(\mc{C})$ which sends the isomorphism class of an object to the trace class of the identity morphism on that object. It is often injective but not surjective. In particular, if $\mc{C}$ is a graded category and $\mc{C}^*$ is the category obtained from $\mc{C}$ by enlarging morphism spaces to include morphisms of non-zero degree, then $\Tr(\mc{C}^*)$ is a graded vector space and the image of $h_{\mc{C}^*}$ is concentrated in degree zero, so the trace of $\mc{C}^*$ is considerably richer than its Grothendieck group.

Higher representation theory is concerned with categorifying Lie algebras (and related algebras) and their representation theory. For a given Cartan datum, Rouquier \cite{Rou08} and Khovanov-Lauda \cite{KL10} independently constructed a graded 2-category, the \emph{categorified quantum group} $\catQGdot=\catQGdot(\g)$, whose split Grothendieck group $K_0(\catQGdot)$ is isomorphic to the corresponding quantum group $\QGdotZ=\QGdotZ(\g)$. Khovanov and Lauda's presentation of $\catQGdot$ is diagrammatic; the 2-morphism spaces are spanned by oriented string diagrams.

The deformed and undeformed \emph{cyclotomic quotients}, denoted $\defcyclolam$ and $\cyclolam$ respectively, are graded 2-representations of $\catQGdot$ associated to a given dominant weight $\lambda$. Kang-Kashiwara \cite{KK12} and Webster \cite{Web17} independently showed that $K_0(\defcyclolam)$ and $K_0(\cyclolam)$ are both isomorphic to the highest weight module for $\QGdotZ$ of weight $\lambda$. In addition, Webster constructed a graded 2-representation $\TP$ of $\catQGdot$ for any sequence $\ulambda=(\lam{1},\ldots,\lam{n})$ of dominant weights whose split Grothendieck group is isomorphic to the tensor product of highest weight modules with these weights (actually in \cite{Web17} Webster worked with the \emph{tensor product algebra} $T^{\ulambda}$, but it will be convenient for us to work with the category $\TP\cong T^{\ulambda}~\text{-pmod}$ defined in \cite{Web16}). Morphisms spaces in $\TP$ are spanned by string diagrams containing red strings that separate the tensor factors.

The trace of $\catQGdotstar$ is a graded algebra and graded 2-representations of $\catQGdotstar$ become graded representations of $\Tr(\catQGdotstar)$. In \cite{BHLZ14} the authors showed that $\Tr(\catQGdotstar(\mf{sl}_2))$ is isomorphic to (the idempotent form of) the \emph{current algebra} $\dot{U}(\mf{sl}_2[t])$ with the indeterminate $t$ in degree 2. The analogous statement for $\mf{sl}_3$ was proved in \cite{Ziv14} and this was extended to any $\g$ of ADE type in \cite{BHLW17}, where the authors also identified $\Tr(\defcyclolamstar)$ and $\Tr(\cyclolamstar)$ with global $\WW(\lambda)$ and local $W(\lambda)$ Weyl modules for $\currdot$ respectively.

In this paper we extend the results of \cite{BHLW17} to the starred tensor product categorification $\TPstar$ and its deformed counterpart $\defTPstar$ defined in \cite{Web12}. Our main theorem is:

\begin{introtheorem}\label{thm:introtheorem}
    Take $\g$ of ADE type and a sequence $\ulambda=(\lam{1},\ldots,\lam{n})$ of dominant weights. There are isomorphisms of graded $\currdot$-algebras
    \begin{align}\label{eq:introtheorem}
        \begin{split}
            \WWulam:=\WW(\lam{1})\otimes\cdots\otimes\WW(\lam{n})&\longrightarrow \Tr(\defTPstar) \\
            \Wulam:=W(\lam{1})\otimes\cdots\otimes W(\lam{n})&\longrightarrow \Tr(\TPstar).
        \end{split}
    \end{align}
\end{introtheorem}

\noindent These maps are uniquely characterized by the fact that they commute with the action of $\Tr(\catQGdotstar)\cong\currdot$, and that adding an additional tensor factor in $\WWulam$ or $\Wulam$ and taking its cyclic vector corresponds to drawing an additional red string on the left of a diagram.

The algebra $\Sym$ of symmetric functions is central to how we relate $\currdot$ and $\catQGdotstar$ and their (2-)representations. Let $\Pi=\bigotimes_{i\in I}\Sym$, where $I$ is the indexing set for simple roots of $\g$. This is a graded Hopf algebra with $p_{i,r}$, the $r$th power-sum symmetric function in the $i^{\text{th}}$ copy of $\Sym$, in degree $2r$ and coproduct $\delta$ given by
\begin{equation}
    \delta(p_{i,r})=p_{i,r}\otimes 1+1\otimes p_{i,r}.
\end{equation}

For $\mu\in X$, there is a commutative diagram of graded algebras:
\begin{equation}
    \begin{tikzcd}
          & \Pi \ar[ld] \ar[rd, "b_{\mu}"] &  \\ 
        U(\h[t]) \ar[rr] &  & \catQGdotstar(1_{\mu},1_{\mu})
    \end{tikzcd}
\end{equation}

\noindent The left diagonal map (a map of Hopf algebras) sends $p_{i,r}\mapsto \xi_i\otimes t^r$ and the map $b_{\mu}$ sends complete homogeneous (resp.\ elementary) symmetric functions to clockwise (resp.\ counter-clockwise) bubbles. The horizontal map agrees with the isomorphism $\currdot\cong \Tr(\catQGstar)$ after passing to the trace.

In \cite{CFK10} the authors showed that each global Weyl module $\WW(\lam{k})$ carries a graded right action of $\Pi$ and this action factors through the surjection ${a_k:\Pi\to \bA_{\lam{k}}=\bigotimes_{i\in I}\Sym_{\langle i,\lam{k}\rangle}}$, where $\Sym_m$ denotes the algebra of symmetric polynomials in $m$ variables. So the tensor product $\WWulam$ is naturally a graded $(\currdot,\ATP)$-bimodule, where $\ATP=\bigotimes_{k=1}^n\bA_{\lam{k}}$. If we consider $\ATP$ as a subalgebra of a polynomial algebra $\kk[Z]$ in a set of indeterminates $Z$ then $\Wulam$ is obtained from $\WWulam$ by tensoring with the left $\ATP$-module $\kk$ on which all $z\in Z$ act as zero.

The deformed tensor product categorification $\defTPstar$ is obtained by deforming the defining relations in $\TPstar$ over $\ATP$ so that a dot on a string, rather than being nilpotent, has spectrum contained in the set of indeterminates $Z$. It is enriched over graded right $\ATP$-modules and the undeformed category $\TPstar$ is obtained from $\defTPstar$ by tensoring morphism spaces with the left $\ATP$-module $\kk$ as above. The spectrum of a dot is determined by the action of bubbles, so it is important for us to understand how these interact with red strings in $\defTPstar$. In particular we show the following:
\begin{introprop}\label{prop:introprop}
    Take $\mu\in X$ and $1\leq k\leq n$. Recall the coproduct $\delta$ on $\Pi$, the map $b_{\mu}$ sending elements of $\Pi$ to bubbles, and the projection $a_k:\Pi\to \bA_{\lam{k}}\leq \ATP$. If $f\in \Pi$ and $\delta(f)=\sum_s g_s\otimes g_s'$ then in $\defTPstar$:
        \begin{equation}
        \begin{tikzpicture}[very thick, baseline=(current bounding box.center)]
            \draw[rounded corners] (-2.25,.3) rectangle (-.3,-.3) node[midway]{$b_{\mu+\lambda^{(k)}}(f)$};
            \draw[wei] (0,.5)--(0,-.5);
            \node at (0,-.7){$(k)$};
            \node[scale=1.3] at (.5,.5){$\mu$};
        \end{tikzpicture}
        \ \ = \ 
        \ 
        \begin{tikzpicture}[very thick, baseline=(current bounding box.center)]
            \draw[wei] (-.5,.5)--(-.5,-.5);
            \node[scale=1.4] at (-3,-.3){$\displaystyle \sum_s $};
            \draw[rounded corners] (-2.1,.3) rectangle (-.65,-.3) node[midway]{$a_k(g_s')$};
            \node at (-.5,-.7){$(k)$};
            \draw[rounded corners] (0,.3) rectangle (1.05,-.3) node[midway]{$b_\mu(g_s)$};
            \node[scale=1.3] at (1.4,.5){$\mu$};
        \end{tikzpicture} 
    \end{equation}
\end{introprop}

\noindent The label $(k)$ on the red string indicates that this corresponds to the weight $\lam{k}$.

In the rest of the introduction we outline the proof of Theorem~\ref{thm:introtheorem}. We show that $\TrdefTPstar$ and $\TrTPstar$ are spanned by the classes of diagrams with no crossings between red and black strings, so the maps in \eqref{eq:introtheorem} are surjective if they are well-defined. Combining this with a filtration of $\TrTPstar$ coming from the standardly stratified structure on $\TPstar$, the results of \cite{BHLW17} allow us to derive an upper bound for the dimension of the trace:
\begin{equation}\label{eq:introdims}
    \dim_{\kk} \TrTPstar\leq \dim_{\kk}\Wulam.
\end{equation}

\noindent It remains to show that these dimensions are equal and the maps are well-defined. We do this by showing the maps are well-defined \emph{at the generic point} using Webster's ``unfurling'' (\cite{Web15}), and apply the upper semi-continuity of dimension under deformation.

Let $\KK=\overline{\kk(Z)}$, the algebraic closure of the field of rational polynomials in $Z$. In \cite{Web15} Webster showed that the idempotent completion of the extension of scalars $\defTPstar\otimes_{\ATP}\KK$ carries a 2-representation of the categorified quantum group $\catQGdotstar(\gtild)$ for a larger Lie algebra $\gtild=\bigoplus_{z\in Z}\g$, called an \emph{unfurling} of $\g$, and moreover it is equivalent to a cyclotomic quotient of $\catQGdotstar(\gtild)$. The corresponding statement for Weyl modules was proved in \cite{CFK10}: $\WWulam\otimes_{\ATP}\KK$ is isomorphic to a tensor product of local Weyl modules for $\currdot$ indexed by $Z$; that is, to a single local Weyl module for $\dot{U}(\gtild[t])$ (actually the module structures on these local Weyl modules are ``twisted'' according to the parameter $z$).

Combining these two pictures with \cite{BHLW17} allows us to construct an isomorphism
\begin{equation}\label{eq:introgeneric}
    \WWulam\otimes_{\ATP}\KK\longrightarrow \TrdefTPstar\otimes_{\ATP}\KK.
\end{equation}

\noindent In particular their dimensions are equal, so the upper bound \eqref{eq:introdims} and the upper semi-continuity of dimension under deformation implies that $\TrdefTPstar$ is flat over $\ATP$. Moreover, Proposition~\ref{prop:introprop} implies that the right actions of $\ATP$ on $\WWulam$ and $\TrdefTPstar$ are compatible and so flatness allows us to lift \eqref{eq:introgeneric} to a bimodule isomorphism $\WWulam\to\TrdefTPstar$. Tensoring over $\ATP$ with $\kk$ gives the isomorphism ${\Wulam\to\TrTPstar}$.

The structure of the paper is as follows. In Section 2 we recall preliminaries on quantum groups, current algebras, and Weyl modules. In Section 3 we recall the categorified quantum group $\catQGdot$, 2-representations, and cyclotomic quotients. In Section 4 we recall the undeformed and deformed tensor product algebras and prove Proposition~\ref{prop:introprop}. In Section 5 we adapt Webster's theory of unfurling to our situation to determine the structure of $\defTPstar$ at the generic point. In Section 6 we use the results of the previous section to show that morphism spaces in $\defTPstar$ are free over $\ATP$. In Section 7 we recall the process of trace decategorification and start investigating the structure of $\TrTPstar$. Finally in Section 8 we prove our main result, Theorem~\ref{thm:introtheorem}.

\subsubsection*{Acknowledgements} 

The authors are incredibly grateful to Ben Webster who outlined how the main result could be proved and who was generous with his time and knowledge throughout the project. The authors also thank Weiqiang Wang for initially suggesting the project. 

\subsubsection*{Conventions}

Throughout the article $\kk$ will denote a fixed field of characteristic zero. For an integer $m$, $[1,m]$ denotes the set of integers $k$ such that $1\leq k\leq m$.

Let $\mc{C}$ be a small $\kk$-linear category (we will generally just say ``category'' and ``functor'' and the reader can assume that everything is small and linear). We write $\ob(\mc{C})$ for the set of objects in $\mc{C}$, $\mc{C}(x,y)$ for the $\kk$-vector space of morphisms from $x$ to $y$, and $1_x\in\mc{C}(x,x)$ for the identity morphism on $x$.

If $\mc{C}$ is a graded category with grading shift $\langle 1\rangle$ we let $\mc{C}^*$ denote the category with the same objects as $\mc{C}$ and morphism spaces given by
\begin{equation}
    \mc{C}^*(x,y)=\bigoplus_{t\in \ZZ}\mc{C}(x,y\langle t\rangle),
\end{equation}

\noindent Morphism spaces in $\mc{C}^*$ are $\ZZ$-graded with $\mc{C}(x,y\langle t\rangle)$ in degree $t$.

Let $K_0(\mc{C})$ denote the split Grothendick group of $\mc{C}$ and write $[x]\in K_0(\mc{C})$ for the class of $x\in\ob(\mc{C})$. Write $K_0^{\kk}(\mc{C})=K_0(\mc{C})\otimes_{\ZZ}\kk$. If $\mc{C}$ is graded then grading shift induces a $\ZZ[q^{\pm 1}]$ action on $K_0(\mc{C})$ and $K_0(\mc{C}^*)\cong K_0(\mc{C})\otimes_{\ZZ[q^{\pm 1}]}\ZZ$ where $q$ acts on $\ZZ$ as 1.

\section{Quantum groups and current algebras}\label{Sec:Current}

We fix notation and recall some standard results about current algebras and their modules. The main reference for current algebras is \cite{CFK10}.

\subsubsection{Cartan datum}\label{section:CartanDatum}

Fix a symmetric simply-laced Cartan datum consisting of:
\begin{itemize}
    \item a free $\ZZ$-module $X$, the weight lattice;
    \item a finite indexing set $I$, simple roots $\alpha_i\in X$ and fundamental weights $\Lambda_i\in X$ for $i\in I$;
    \item simple coroots $\alpha_i^{\vee}\in X^{\vee}:=\Hom_{\ZZ}(X,\ZZ)$ for $i\in I$;
    \item a symmetric bilinear form $(-,-)$ on $X$.
\end{itemize}

\noindent Write $\langle -,-\rangle:X^{\vee}\times X\to \ZZ$ for the canonical pairing. Assume that
\begin{itemize}
    \item $(\alpha_i,\alpha_i)=2$ for all $i\in I$;
    \item $\langle i,\lambda \rangle:=\langle \alpha_i^{\vee},\lambda \rangle=(\alpha_i,\lambda)$ for $i\in I$ and $\lambda\in X$;
    \item $(\alpha_i,\alpha_j)\in\{ 0,-1\}$ for $i,j\in I$ with $i\neq j$;
    \item $\langle i,\Lambda_j\rangle=\delta_{ij}$ for $i,j\in I$.
\end{itemize}

\noindent For $i,j\in I$ we will write $\langle j,i\rangle:=\langle \alpha^{\vee}_j,\alpha_i\rangle$. Let $X^+=\bigoplus_{i\in I} \ZZ_{\geq 0}\Lambda_i$ be the set of dominant weights in $X$. We sometimes write $+I$ for $I$ and define a set of formal negatives $-I=\{ -i \mid i\in I\}$. Define the ``absolute value'' $\lvert \pm i\rvert=i$. For $i,j\in I$ let $\alpha_{-i}=-\alpha_i\in X$ and $\langle j,-i\rangle=-\langle j,i\rangle $.

We assume that the Cartan datum is of finite type, so the Cartan matrix $(\langle i,j\rangle)_{i,j\in I}$ is invertible. Let $\Gamma$ denote the corresponding graph without loops or multiple edges. It has vertex set $I$ and an edge between $i$ and $j$ if any only if $\langle i,j\rangle=-1$. For the rest of the article we fix an orientation on $\Gamma$.

\subsubsection{The quantum group}\label{subsec:QG}

Let $q$ be an indeterminate and let $U_q=U_q(\mf{g})$ denote the quantum group associated to the Cartan datum above. This is the $\QQ(q)$-algebra generated by $E_i$, $F_i$, and $K_{\mu}$ for $i\in I$ and $\mu\in X^{\vee}$ and subject to familiar relations (we use the same conventions as in \cite[\S3.1]{BHLW17}.

The quantum group is a Hopf algebra with coproduct $\Delta_q$ given by
\begin{equation*}
    \Delta_q(E_i)=E_i\otimes 1+K_i\otimes E_i,\quad \Delta_q(K_{\mu})=K_{\mu}\otimes K_{\mu},\quad \Delta_q(F_i)=F_i\otimes K_{-i}+1\otimes F_i
\end{equation*}

\noindent for $i\in I$ and $\mu\in X^{\vee}$. Let $U_q^{\ZZ}$ denote the integral form of $U_q$; the $\ZZ[q^{\pm 1}]$-subalgebra of $U_q$ generated by the $K_{\mu}$ for $\mu\in X^{\vee}$ and the divided powers $E_i^{(a)}$ and $F_i^{(a)}$ for $i\in I$, $a\in \NN$.

Let $\dot{U}_q=\dot{U}_q(\mf{g})$ denote the idempotent form of $U_q$; the locally unital $\QQ(q)$-algebra obtained from $U_q$ by adjoining mutually orthogonal idempotents $1_{\lambda}$ for $\lambda\in X$ satisfying
\begin{equation}
    E_i1_{\lambda}=1_{\lambda+\alpha_i}E_i,\quad K_{\mu}1_{\lambda}=1_{\lambda}K_{\mu}=q^{\langle \mu,\lambda\rangle}1_{\lambda},\quad F_i1_{\lambda}=1_{\lambda-\alpha_i}F_i.
\end{equation}

\noindent The algebra $\dot{U}_q$ decomposes as a direct sum
\begin{equation}
    \dot{U}_q=\bigoplus_{\lambda,\mu\in X} 1_{\mu}U_q1_{\lambda}.
\end{equation}

\noindent Let ${\dot{U}_q^{\ZZ}=\bigoplus 1_{\mu}U_q^{\ZZ}1_{\lambda}}$ denote the integral form of $\dot{U}_q$.

For $\lambda\in X^+$, let $V_q(\lambda)$ denote the irreducible $\dot{U}_q$-module generated by a highest weight vector $v_{\lambda}$ of weight $\lambda$ and let $V_q^{\ZZ}(\lambda)=\dot{U}_q^{\ZZ}v_{\lambda}$ be its integral form. For a sequence $\ulambda=(\lam{1},\ldots,\lam{n})$ of dominant weights, let
\begin{equation}
    V_q^{\ZZ}(\ulambda)=V_q^{\ZZ}(\lam{1})\otimes_{\ZZ[q^{\pm 1}]}\cdots\otimes_{\ZZ[q^{\pm 1}]} V_q^{\ZZ}(\lam{n}),
\end{equation}

\noindent regarded as a $\dot{U}_q^{\ZZ}$-module via the coproduct $\Delta_q$.

\subsubsection{The current algebra}\label{Subsec:CurrentDef}

Let $\g$ be the semisimple Lie algebra over $\QQ$ determined by the Cartan matrix. For $i\in I$, we write $e_i$ and $f_i$ for the root vectors of weights $\alpha_i$ and $-\alpha_i$ respectively, and write $\xi_i=[e_i,f_i]$. We will sometimes also write $e_{-i}$ for $f_i$. Let $\mf{h}=\text{span}_{\QQ} \{\xi_i\mid i\in I\}$ be the canonical Cartan subalgebra of $\g$ and $\mf{n}=\langle e_i\mid i\in I\rangle$ the sum of the positive root spaces.

Let $t$ be an indeterminate and set $\mf{g}[t]:=\mf{g}\otimes_{\QQ}\QQ[t]$ with Lie bracket given by
\begin{equation}
    [x\otimes t^r,y\otimes t^s]=[x,y]\otimes t^{r+s}
\end{equation}

\noindent for $x,y\in\mf{g}$ and $r,s\in\NN$.

\begin{definition}
    The \emph{current algebra} $U(\mf{g}[t])=U_{\kk}(\mf{g}[t])$ of $\mf{g}$ over $\kk$ is the universal enveloping algebra of $\mf{g}[t]$ over the field $\kk$. It is $\ZZ$-graded with $x\otimes t^r$ in degree $2r$ for $x\in \mf{g}$ and $r\in \NN$.
\end{definition}

\noindent This is a Hopf algebra with coproduct sending $x$ to $x\otimes 1+1\otimes x$ for all $x\in \mf{g}[t]$.

The current algebra has an idempotent form $\dot{U}(\mf{g}[t])$; the locally unital $\kk$-algebra obtained from $U(\mf{g}[t])$ by adjoining mutually orthogonal idempotents $1_{\lambda}$ for $\lambda\in X$ satisfying
\begin{equation*}
    (e_i\otimes t^r)1_{\lambda}=1_{\lambda+\alpha_i}(e_i\otimes t^r),\quad(\xi_i\otimes t^r)1_{\lambda}=1_{\lambda}(\xi_i\otimes t^r),\quad
    (f_i\otimes t^r)1_{\lambda}=1_{\lambda-\alpha_i}(f_i\otimes t^r)
\end{equation*}

\noindent and
\begin{equation}
    (\xi_i\otimes 1)1_{\lambda}=1_{\lambda}(\xi_i\otimes 1)=\langle i,\lambda\rangle1_{\lambda}
\end{equation}

\noindent for $i\in I$ and $r\in \NN$. The idempotent form is also $\ZZ$-graded. We will pass freely between weight modules for $U(\mf{g}[t])$ with integral weights and $\dot{U}(\mf{g}[t])$-modules.

\subsubsection{Symmetric functions}\label{subsec:SymmetricFunc}

Let $\Sym$ denote the ring of symmetric functions over $\kk$ and define $\Pi:=\bigotimes_{i\in I}\Sym$. We denote power sum, elementary, and complete homogeneous symmetric functions in the $i^{\text{th}}$ component of $\Pi$ by $\{p_{i,r}\}_{r\geq 1}$, $\{e_{i,r}\}_{r\geq 1}$, and $\{h_{i,r}\}_{r\geq 1}$ respectively. We consider $\Pi$ as a $\ZZ$-graded algebra with $p_{i,r}$, $e_{i,r}$, and $h_{i,r}$ in degree $2r$.

It will be useful for us to consider the generating functions
\begin{equation}
    p_i(x)=\sum_{r=1}^{\infty}p_{i,r}x^r,\quad e_i(x)=\sum_{r=0}^{\infty} e_{i,r}x^r,\quad h_i(x)=\sum_{r=0}^{\infty} h_{i,r}x^r,
\end{equation}

\noindent where $x$ is a formal indeterminate (by convention $h_{i,0}=e_{i,0}=1$). If we regard the $i$th copy of $\Sym$ in $\Pi$ as consisting of symmetric functions in countably many variables $y_1,y_2,\ldots$ then
\begin{equation}\label{eq:GenFuncVariables}
    p_i(x)=\sum_{j=1}^{\infty}\frac{y_jx}{1-y_jx},\quad e_i(x)=\prod_{j=1}^{\infty} (1+y_jx),\quad h_i(x)=\prod_{j=1}^{\infty}(1-y_jx)^{-1}.
\end{equation}

For $n\in \NN$, let $\Sym_n$ denote the polynomial algebra in $n$ variables over $\kk$, also $\ZZ$-graded with each variable having degree 2. For $m,n\in\NN$, there is an inclusion
\begin{equation}
    \Sym_{m+n}\longrightarrow \Sym_m\otimes \Sym_n.
\end{equation}

\noindent Taking the direct limit over $m$ and $n$ this gives a map from $\Sym$ to $\Sym\otimes\Sym$, and tensoring over $I$ copies yields a coproduct $\delta:\Pi\to \Pi\otimes \Pi$. In terms of the generating functions:
\begin{equation*}
    \delta(p_i(x))=p_i(x)\otimes 1+1\otimes p_i(x),\quad \delta(e_i(x))=e_i(x)\otimes e_i(x),\quad \delta(h_i(x))=h_i(x)\otimes h_i(x).
\end{equation*}

For any $\lambda\in X$, there is an isomorphism of graded Hopf algebras
\begin{align}\label{eq:CartanAndSym}
    \begin{split}
        \Pi  & \longrightarrow 1_{\lambda}U(\mf{h}[t])1_{\lambda} \\
        p_{i,r} & \longmapsto 1_{\lambda}(\xi_i\otimes t^r)1_{\lambda}
    \end{split}
\end{align}

\noindent for $i\in I$ and $r\geq 1$. Note in particular that it intertwines $\delta$ with the coproduct on $U(\mf{h}[t])$.


%

%

For $\lambda\in X^+$ define a $\ZZ$-graded algebra
\begin{equation}
    \bA_{\lambda}:=\bigotimes_{i\in I}\Sym_{\langle i,\lambda\rangle}.
\end{equation}


\noindent There is a surjective homomorphism of graded algebras $\Pi\to \bA_{\lambda}$ sending $p_{i,r}$, $e_{i,r}$, and $h_{i,r}$ to the corresponding symmetric polynomials. 


\subsubsection{Weyl modules}\label{subsec:Verma}

Let $\mf{p}=\mf{h}\oplus \mf{n}[t]\subseteq \mf{g}[t]$. For $\lambda\in X^+$, let $M(\lambda)$ denote the Verma-like module
\begin{equation}
    M(\lambda)=U(\mf{g}[t])\otimes_{U(\mf{p})}\kk_{\lambda},
\end{equation}

\noindent where $\kk_{\lambda}$ is the one-dimensional $\mf{h}$-module on which each $\xi_i$ acts by $\langle i,\lambda\rangle$, induced up to $\mf{p}$. Let $m_{\lambda}:=1\otimes 1\in M(\lambda)$. Then $M(\lambda)$ is generated by $m_{\lambda}$ subject to the following relations:
\begin{equation}\label{eq:VermaPresentation}
    \mf{n}[t]\cdot m_{\lambda}=0, \quad (\xi_i\otimes 1)\cdot m_{\lambda}=\langle i,\lambda\rangle m_{\lambda}
\end{equation}

\noindent for $i\in I$. It is a graded $U(\mf{g}[t])$-module with $m_{\lambda}$ in degree 0. There is a right action of $\Pi$ on $M(\lambda)$ given by
\begin{equation}
    (um_{\lambda})\cdot p_{i,r}=u(\xi_i\otimes t^r)m_{\lambda}
\end{equation}

\noindent for $i\in I$, $r\geq 1$, and $u\in U(\mf{g}[t])$. This implicitly uses the identification in \eqref{eq:CartanAndSym} of $\Pi$ with $1_{\lambda}U(\mf{h}[t])1_{\lambda}$, and makes $M(\lambda)$ a graded $(U(\mf{g}[t]),\Pi)$-bimodule.

The \emph{global Weyl module} $\WW(\lambda)$ is the $U(\mf{g}[t])$-module quotient of $M(\lambda)$ by the relation
\begin{equation}
    (f_i\otimes 1)^{\langle i,\lambda\rangle+1}\cdot m_{\lambda}=0
\end{equation}

\noindent for $i\in I$. It is also $\ZZ$-graded. We write $w_{\lambda}$ for the image of $m_{\lambda}$ in $\WW(\lambda)$. The right action of $\Pi$ on $M(\lambda)$ descends to an action on $\WW(\lambda)$. In fact we have the following:

\begin{theorem}\cite{CFK10}\label{thm:WeylFree}
    The action of $\Pi$ on $\WW(\lambda)$ factors through a faithful action of $\bA_{\lambda}$. Moreover, $M(\lambda)$ and $\WW(\lambda)$ are free 
    right $\Pi$- and $\bA_{\lambda}$-modules respectively.
\end{theorem}

\begin{proof}
    The first claim is \cite[Theorem~4]{CFK10}. The fact that $\WW(\lambda)$ is a free $\bA_{\lambda}$-module follows from the work of \cite{CP01}, \cite{CL06}, \cite{FL07}, and \cite{BN04}. See \cite[\S7.2]{CFK10} for a more detailed discussion. Finally $M(\lambda)$ is free by the PBW theorem.
\end{proof}

\noindent In particular, $\WW(\lambda)$ is a graded $(U(\mf{g}[t]),\bA_{\lambda})$-bimodule.

The \emph{local Weyl module} $W(\lambda)$ associated to $\lambda\in X^+$ is 
\begin{equation}
    W(\lambda):=\WW(\lambda)\otimes_{\bA_{\lambda}}\kk,
\end{equation}

\noindent where $\kk$ is the unique simple graded $\bA_{\lambda}$-module. Equivalently, $W(\lambda)$ is the $U(\mf{g}[t])$-module quotient of $\WW(\lambda)$ by the relation
\begin{equation}
    \mf{h}t[t]\cdot w_{\lambda}=0.
\end{equation}

\noindent It is a graded $U(\mf{g}[t])$-module.

\subsubsection{Notation for tensor products}\label{subsec:TPWeylMods}

For the rest of the paper we fix $n\in \NN$ and a sequence $\ulambda=(\lam{1},\ldots,\lam{n})$ of dominant weights. For $k\in [1,n]$ and $i\in I$ let $\lam{k}_i:=\langle i,\lam{k}\rangle$. Define
\begin{equation}
    M(\ulambda)=\bigotimes_{k=1}^n M(\lam{k}),\quad \WWulam=\bigotimes_{k=1}^n \WW(\lam{k}), \quad \Wulam=\bigotimes_{k=1}^n W(\lam{k}).
\end{equation}

\noindent Each of these is a graded $\curr$-module via the coproduct.

Define also
\begin{equation}
    \ATP=\bigotimes_{k=1}^n\bA_{\lam{k}}=\bigotimes_{i\in I}\bigotimes_{k=1}^n\Sym_{\lam{k}_i}
\end{equation}

\noindent and write $a_k$ for the projection $\Pi\to\bA_{\lam{k}}\subseteq\ATP$. We regard elements of $\Sym_{\lam{k}_i}$ as symmetric polynomials in a set $Z_i^{(k)}$ of $\lam{k}_i$ indeterminates. Define disjoint unions
\begin{equation}
    Z^{(k)}:=\coprod_{i\in I} Z^{(k)}_i,\qquad Z_i:=\coprod_{k\in [1,n]} Z^{(k)}_i\qquad Z:=\coprod_{k\in[1,n]}Z^{(k)}.
\end{equation}

The algebra $\ATP$ is graded local ring. Let $\kk$ denote the unique simple graded $\ATP$-module (on which all $z\in Z$ act as zero) and let $\KK:=\overline{\kk(Z)}$ (the algebraic closure of the field of rational functions in the indeterminates $Z$). We consider $\ATP$ as embedded in $\KK$.

There is a graded $(\curr,\Pi^{\otimes n})$-bimodule structure on $M(\ulambda)$ where $\Pi^{\otimes n}$ acts component-wise (note that the right actions are defined component-wise, not using the coproduct $\delta$). This induces a graded $(\curr,\ATP)$-bimodule structure on $\WWulam$ from which we obtain $\Wulam=\WWulam\otimes_{\ATP}\kk$.

\section{Categorified quantum groups}

In this section we recall the definition of the categorified quantum group ${\catQGdot=\catQGk}$. We state an isomorphism between the tensor product $\Pi=\bigotimes_{i\in I}\Sym$ of symmetric functions and bubbles in $\catQGdot$ and formulate bubble slides in $\catQGdot$. Finally we recall the notion of a (graded) 2-representation of $\catQGdot$ and the definition of the deformed $\defcyclolam=\defcyclolamk$ and undeformed $\cyclolam=\cyclolamk$ cyclotomic quotients of $\catQGdot$.

\subsection{Definition}\label{subsec:CatQG}

We use the ``cyclic'' formulation of the categorified quantum group defined in \cite{BHLW16}, where it is denoted $\mc{U}^{cyc}_Q(\mf{g})$. It is shown in Theorem~2.1 in loc.\ cit.\ that $\mc{U}^{cyc}_Q(\mf{g})$ is equivalent to the 2-category defined in \cite{CL15}.

Note that we read our diagrams from \emph{right to left} and bottom to top following \cite{BHLW16}, \cite{CL15}, \cite{KL10}, and \cite{Lau10}. In contrast, Webster (\cite{Web15}, \cite{Web16}, \cite{Web17}) reads his diagrams from left to right.

\subsubsection{Choice of parameters}\label{subsec:parameters}

The definition of $\catQGdot$ depends on two additional pieces of information: a \emph{choice of scalars} $t_{ij}\in \kk$ for $i,j\in I$ and a choice of \emph{bubble parameters} $c_{i,\lambda}\in \kk$ for $i\in I$ and $\lambda\in X$.

Recall from \S\ref{section:CartanDatum} that we have fixed an orientation on the graph $\Gamma$ associated the Cartan datum. For the choice of scalars we set
\begin{equation}\label{eq:ChoiceOfScalars}
    t_{ij}=\begin{cases}
        1   & \text{if } i=j \text{ or } \langle i,j\rangle =0; \\
        1   & \text{if there is an arrow }j\to i; \\
        -1  & \text{if there is an arrow }i\to j.
    \end{cases}
\end{equation}

\noindent In \cite{Web16} this is called the \emph{geometric} choice of scalars.

For the bubble parameters we allow any choice of $c_{i,\lambda}\in \kk\setminus \{0\}$ for $i\in I$ and $\lambda\in X$ consistent with the conditions
\begin{equation}
    c_{i,\lambda+\alpha_j}=t_{ij}c_{i,\lambda}
\end{equation}

\noindent for $i,j\in I$ and $\lambda\in X$.

\subsubsection{KLR algebras}\label{subsec:KLR}

For $m\in \NN$, the Khovanov-Lauda-Rouquier algebra, or \emph{KLR algebra}, $R_m$ is a $\kk$-algebra with generators $\{ \varep_{\ui} \mid \ui\in I^m\}$, $\{y_k\mid k\in[1,m]\}$, and $\{\psi_l \mid l\in[1,m-1]\}$.

We will often represent elements of $R_m$ diagrammatically: write
\begin{equation*}\label{KLRRels1}
    \varep_{\ui} = 
    \begin{tikzpicture}[very thick,baseline=(current bounding box.center)] 
        \draw (-1,0)--(-1,1); 
        \node at (-1,-.3){$i_1$};
        \draw (-.5,0)--(-.5,1); 
        \node at (-.5,-.3){$i_2$};
        \node at (0,.5){$\cdots$};
        \draw (.5,0)--(.5,1); 
        \node at (.5,-.3){$i_m$};
    \end{tikzpicture}\qquad
    y_k \varep_{\ui} = 
    \begin{tikzpicture}[very thick,baseline=(current bounding box.center)] 
        \draw (-1,0)--(-1,1); 
        \node at (-1,-.3){$i_1$};
        \node at (-.5,.5){$\cdots$};
        \draw (0,0)--(-0,1); 
        \node at (0,-.3){$i_k$};
        \draw[fill] (0,.5) circle(2pt);
        \draw (1,0)--(1,1); 
        \node at (.5,.5){$\cdots$};
        \node at (1,-.3){$i_m$};
    \end{tikzpicture}\qquad
    \psi_l \varepsilon_{\ui} = 
    \begin{tikzpicture}[very thick,baseline=(current bounding box.center)] 
        \draw (-1,0)--(-1,1); 
        \node at (-1,-.3){$i_1$};
        \node at (-.5,.5){$\cdots$};
        \draw (0,0)--(1,1); 
        \node at (0,-.3){$i_l$};
        \draw (1,0)--(0,1);
        \node at (1,-.3){$i_{l+1}$};
        \node at (1.5,.5){$\cdots$};
        \draw (2,0)--(2,1); 
        \node at (2,-.3){$i_m$};
    \end{tikzpicture}
\end{equation*}
    
\noindent The product $a\cdot b$ is represented by placing the diagram for $a$ above that for $b$ and attempting to connect the strings. If the labels do not match then we get zero (i.e. the $\varepsilon_{\ui}$ are mutually orthogonal).

If $q(y)=\sum_{r=0}^m a_ry^r\in\kk[y]$ is a polynomial, we will denote the element $q(y_k)\varep_{\ui}$ diagrammatically by
\begin{equation}\label{eq:PolysInDots}
    \begin{tikzpicture}[very thick,baseline=(current bounding box.center),rounded corners]
        \draw (-3.25,.75)--(-3.25,-.5) node[at end, below]{$i_1$};
        \node at (-2.85,.125){$\cdots$};
        \draw (-2,.75)--(-2,.375);
        \draw (-2.5,-0.125) rectangle (-1.5,.375) node[midway]{$q(y)$};
        \draw (-2,-0.125)--(-2,-.5) node[at end, below]{$i_k$};
        \node at (-1.15,.125){$\cdots$};
        \draw (-0.75,.75)--(-0.75,-.5) node[at end, below]{$i_n$};
    \end{tikzpicture} \ = \ \scalebox{1.2}{${\displaystyle \sum_{r=0}^m a_r}$}
    \ \begin{tikzpicture}[very thick,baseline=(current bounding box.center),rounded corners]
        \draw (-3.0,.75)--(-3.0,-.5) node[at end, below]{$i_1$};
        \node at (-2.5,.125){$\cdots$};
        \draw[fill] (-2,.125) circle (2pt) node[right]{$r$};
        \draw (-2,.75)--(-2,-.5) node[at end, below]{$i_k$};
        \node at (-1.15,.125){$\cdots$};
        \draw (-0.75,.75)--(-0.75,-.5) node[at end, below]{$i_n$};
    \end{tikzpicture}
\end{equation}

\noindent where $r$ next to a dot indicates $r$ dots.

Isotopic diagrams are equal and subject to well-known local relations. In particular:
\begin{equation} \label{eq:nilHecke-1}
    \begin{tikzpicture}[scale=.55,baseline]
        \draw[very thick,postaction={decorate,decoration={markings,
        mark=at position .2 with {\arrow[scale=1.3]{}}}}](-4,0)
        +(-1,-1) -- +(1,1) node[below,at start]
        {$i$}; 
        \draw[very
        thick,postaction={decorate,decoration={markings, mark=at
        position .2 with {\arrow[scale=1.3]{}}}}](-4,0) +(1,-1) --
        +(-1,1) node[below,at start]
        {$j$}; \fill (-4.5,.5) circle (4pt);
        \node at (-2,0){=}; \draw[very
        thick,postaction={decorate,decoration={markings, mark=at
         position .8 with {\arrow[scale=1.3]{}}}}](0,0) +(-1,-1) --
        +(1,1) node[below,at start]
        {$i$}; \draw[very
        thick,postaction={decorate,decoration={markings, mark=at
        position .8 with {\arrow[scale=1.3]{}}}}](0,0) +(1,-1) --
        +(-1,1) node[below,at start]
        {$j$}; \fill (.5,-.5) circle (4pt); \node at
        (2,0){$+$};
        \node at (2.75,0){$\delta_{ij}$};
        \draw[very
        thick,postaction={decorate,decoration={markings, mark=at
        position .5 with {\arrow[scale=1.3]{}}}}](4,0) +(-.5,-1) --
        +(-.5,1) node[below,at start]
        {$i$}; \draw[very
        thick,postaction={decorate,decoration={markings, mark=at
         position .5 with {\arrow[scale=1.3]{}}}}](4,0) +(.5,-1) --
        +(.5,1) node[below,at start] {$j$};
    \end{tikzpicture}
\end{equation}
\begin{equation} \label{eq:black-bigon}
  \begin{tikzpicture}[very thick,scale=.65,baseline=(current bounding box.center)]
      \draw[postaction={decorate,decoration={markings, mark=at
          position .5 with {\arrow[scale=1.3] {}}}}] (-2.8,0) +(0,-1)
      .. controls (-1.2,0) ..  +(0,1) node[below,at
      start]{$i$}; \draw[postaction={decorate,decoration={markings,
          mark=at position .5 with {\arrow[scale=1.3] {}}}}] (-1.2,0)
      +(0,-1) .. controls (-2.8,0) ..  +(0,1) node[below,at
      start]{$j$}; \node at (-.5,0) {=};
    \end{tikzpicture}
    \left\{\begin{array}{cl}0\begin{tikzpicture}[very thick, baseline=(current bounding box.center)]
    \node at (0,0){};
    \node at (0,1){};
    \end{tikzpicture}
    & \text{if }\langle i,j\rangle=2, \\
    \begin{tikzpicture}[very thick, baseline=(current bounding box.center)]
    \draw (0,0)--(0,1);
    \draw[] (.5,0)--(.5,1);
    \node at (0,-.2){$i$};
    \node at (.5,-.2){$j$};
    \end{tikzpicture}& \text{if }\langle i,j\rangle=0, \\
     t_{ij}\begin{tikzpicture}[very thick, baseline=(current bounding box.center)]
    \draw (0,0)--(0,1);
    \draw[] (.5,0)--(.5,1);
    \node at (0,-.2){$i$};
    \node at (.5,-.2){$j$};
    \draw[fill] (0,.5) circle(2pt);
    \end{tikzpicture} + 
     t_{ji}\begin{tikzpicture}[very thick, baseline=(current bounding box.center)]
    \draw (0,0)--(0,1);
    \draw[] (.5,0)--(.5,1);
    \node at (0,-.2){$i$};
    \draw[fill] (.5,.5) circle(2pt);
    \node at (.5,-.2){$j$};
    \end{tikzpicture}&\text{if }\langle i,j\rangle=-1
    \end{array}
    \right.
\end{equation}

\noindent For the remaining relations the reader is referred to \cite[(2.8)-(2.14)]{CL15} (we have $r_i=1$, $d_{i,j}=1$, and $s_{ij}^{pq}=0$ for all $i,j,p,q$).

\subsubsection{Definition of $\catQG$}\label{subsec:DefCatQG}

We define a graded $\kk$-linear 2-category $\catQG=\mc{U}_{\kk}(\g)$ (so morphism spaces $\catQG(\lambda,\mu)$ in $\catQG$ are graded $\kk$-linear categories). Objects in $\catQG$ are weights $\lambda\in X$. The 1-morphisms $\lambda\to\mu$ are formal direct sums of grading shifts of symbols $1_{\mu}\mc{E}_{\ui}1_{\lambda}$, where $\ui=(i_1,\ldots, i_m)\in (\pm I)^m$ for some $m\in \NN$ such that
\begin{equation}
    \lambda+\sum_{j=1}^m \alpha_{i_j}=\mu,
\end{equation}
        
\noindent (recall that we write $\alpha_{-i}=-\alpha_i$). Since $\mu$ is uniquely determined we often drop it from our notation and write $\mc{E}_{\ui}1_{\lambda}$ for $1_{\mu}\mc{E}_{\ui}1_{\lambda}$.

The 2-morphisms in $\catQG$ are $\kk$-linear combinations of (grading shifts of) \emph{Khovanov-Lauda (KL) diagrams}. A KL diagram consists of finitely many oriented black strings in $\RR\times [0,1]$, labelled by elements of $I$ and decorated with finitely many dots, with the regions between strings labelled by weights. Diagrams have no triple points or tangencies and any open end of a string must meet one of the lines $y=0$ or $y=1$ at a distinct point from all other strings. The labelling of regions must be consistent with the local rules below:

\begin{equation}\label{eq:WeightsAcrossStrands}
    \begin{tikzpicture}[baseline=(current bounding box.center)]
    \draw[very thick,->] (-3,0)--(-3,1);
    \node at (-3,-.3){$i$};
    \node at (-3.7,.5){$\mu + \alpha_i$};
    \node at (-2.25,.5){$\mu$};
    \draw[very thick,<-] (1,0)--(1,1);
    \node at (1,-.3){$i$};
    \node at (.3,.5){$\mu-\alpha_i$};
    \node at (1.75,.5){$\mu $};
    \end{tikzpicture}
\end{equation}

\noindent Since the labelling of all regions is uniquely determined by that of a single one, we will often only label one region.

Take a KL diagram. Let $\lambda$ and $\mu$ be the weights of the right- and left-most regions respectively. Reading along the bottom ($y=0$) of the diagram yields a signed sequence $\ui=(i_1,\ldots ,i_m)\in (\pm I)^m$ where $|i_k|$ is the label of the $k$th string from the left, and $i_k\in +I$ (resp.\ $-I$) if the string is oriented upward (resp.\ downward). Similarly, reading along the top ($y=1$) of the diagram yields a signed sequence $\uj$. We assign the diagram a degree $d$ by taking the sum of the degrees of the elementary diagrams of which it is composed: a dot has degree 2, a crossing between and $i$-string and a $j$-string has degree $-\langle i,j\rangle$, and cups and caps have the following degrees:
\begin{equation}\label{eq:DegDiagrams2} 
    \deg \begin{tikzpicture}[baseline,very thick, scale=1.5] 
    \draw[->] (.2,.1) to[out=-120,in=-20] node[at end, above left, scale=.8]{$i$} (-.2,.1);
    \node[scale=.8] at (0,.3){$\lambda$}; \end{tikzpicture} = \langle i,\lambda\rangle - 1
\qquad \deg\begin{tikzpicture}[baseline,very thick,scale=1.5]
\draw[<-] (.2,.1)to[out=-120,in=-60] node[at end,above left,scale=.8]{$i$} (-.2,.1);
\node[scale=.8] at (0,.3){$\lambda$};\end{tikzpicture} =-\langle i,\lambda\rangle-1
\end{equation}
\begin{equation}\label{eq:DegDiagrams3}
  \deg\begin{tikzpicture}[baseline,very thick,scale=1.5]\draw[<-] (.2,.1)
    to[out=120,in=60] node[at end,below left,scale=.8]{$i$} (-.2,.1)
    ;\node[scale=.8] at (0,-.1){$\lambda$};\end{tikzpicture} =\langle i,\lambda\rangle-1 \qquad \deg\begin{tikzpicture}[baseline,very
  thick,scale=1.5]\draw[->] (.2,.1) to[out=120,in=60] node[at
    end,below left,scale=.8]{$i$} (-.2,.1);\node[scale=.8] at
    (0,-.1){$\lambda$};\end{tikzpicture} =-\langle i,\lambda\rangle-1.
\end{equation}

\noindent Then this diagram is a 2-morphism $1_{\mu}\mc{E}_{\ui}1_{\lambda}\Rightarrow 1_{\mu}\mc{E}_{\uj}1_{\lambda}\langle d\rangle$.

The \emph{horizontal composition} $D\circ D'$ of KL diagrams $D$ and $D'$ is given by placing $D$ to the left of $D'$ if the weights of the corresponding regions match. Their \emph{vertical composition} $D\cdot D'$ is given by placing $D$ on top of $D'$ and connecting strings if the corresponding 1-morphisms match. These definitions extend to all 2-morphisms. So ``$\cdot$'' denotes composition in the category $\catQG(\lambda,\mu)$ and `$\circ$' denotes composition
\begin{equation}
    \catQG(\nu,\mu)\times \catQG(\lambda,\mu)\longrightarrow \catQG(\lambda,\nu).
\end{equation}

The 2-morphisms in $\catQG$ are subject to the additional local relations listed in \cite[Definition~1.3]{BHLW16}. In particular, if we interpret diagrams in the KLR algebra as having all strings oriented \emph{upward} then 2-morphisms locally satisfy the KLR relations, bubbles are subject to the relations described in \S\ref{subsec:BubblesInU}, and we have the following for any $i\in I$ and $\lambda\in X$:
    \begin{equation}\label{eq:Extendedsl2}
		\begin{tikzpicture}[very thick, baseline=(current bounding box.center)]
			\draw[->] (-.5,-.5)--(-.5,.75) node[at start,below]{$i$};
			\draw[<-] (.25,-.5)--(.25,.75) node[at start,below]{$i$};
			\node[scale=1.3] at (.6,.55){$\lambda$};
		\end{tikzpicture}\ = \ - 
		\begin{tikzpicture}[very thick,scale=.65,baseline=(current bounding box.center)]
		      \draw[->] (-2.8,0) +(0,-1)
		      .. controls (-1.2,0) ..  +(0,1) node[below,at
		      start]{$i$}; 
			\node[scale=1.3] at (-.8,.9){$\lambda$};
			\draw[<-] (-1.2,0)
		      +(0,-1) .. controls (-2.8,0) ..  +(0,1) node[below,at start]{$i$}; 
		\end{tikzpicture} \ + \ 
        \begin{tikzpicture}[very thick, baseline=(current bounding box.center)]
		    \node[scale=1.3] at (-1.5,.3){$\displaystyle \sum_{\substack{\alpha+\beta+\gamma \\ =\langle i,\lambda\rangle - 1}}$};
		    \draw[->] (0,-.6) -- (0,-.3) arc (180:0:0.3) -- (0.6,-.6);
		    \node at (0,-.8) {$i$};
		    \draw[fill] (.6,-.3) circle (1.5pt) node[right] {$\gamma$}; 
		    \draw[fill] (.6,1.2) circle (1.5pt) node[right]{$\alpha$};
   		   \draw[<-] (0,1.5) -- (0,1.2) arc (180:360:0.3) -- (0.6,1.5);
		    \node at (.6,1.7){$i$};
		    \node at (0,.5){$i$};
		    \draw[fill] (.8,.5) circle (1.5pt) node[right]{$\spadesuit + \beta$};
		    \draw[->] (.5,.75) arc(90:450:0.3);
		    \node[scale=1.3] at (1.5,1.3){$\lambda$};
	    \end{tikzpicture}
    \end{equation}
    
The cyclic duality in $\catQG$ (see \cite[\S1.2]{BHLW16}) means that rotating any relation by 180 degrees yields another valid relation in $\catQG$. In particular, rotating the second relation in \cite[Proposition~3.3]{BHLW16} yields the following for any $i\in I$, $\lambda\in X$, and $s\geq 0$:
\begin{equation}\label{eq:CurlRel}
    \begin{tikzpicture}[very thick, baseline=(current bounding box.center), scale=1.2]
        \draw[<-] (1,0) to [out=90, in=85](1.05,0.5);
        \draw (1.05,0.5) arc (-175:175:2mm);
        \draw (1.05,0.46) to [out=95, in=270] (1,1);
        \node at (1,-.25){$i$};
        \node[scale=1.2] at (1.6,0){$\lambda$};
        \filldraw (1.44,0.5) circle (1.5pt);
        \node at (1.7,0.5){$s$};
    \end{tikzpicture} \ = \
    \begin{tikzpicture}[very thick, baseline=(current bounding box.center),scale=1.2]
        \node[scale=1.3] at (-1.8,-.1){$\displaystyle \sum_{\alpha+\beta=s+\langle i,\lambda\rangle}$};
        \draw[->] (0,.5)--(0,-.5);
        \filldraw (0,0) circle (1.5pt);
        \node[scale=1] at (-.4,0) {$\alpha$};
        \node at (0,-.75){$i$};
        \draw[postaction={decorate,decoration={markings,
         mark=at position .6 with {\arrow[scale=1.3]{>}}}},very thick,scale=.66] (1.5,0) circle (15pt);
        \node at (.76,.45){$i$};
        \draw[fill] (1.23,.25) circle[radius=1.5pt];
        \node[scale=0.9] at (1.75,.4){$\spadesuit+\beta$};
        \node[scale=1.2] at (2.13,-.3){$\lambda$};
    \end{tikzpicture}
\end{equation}

\subsubsection{Categorified quantum group}

The categorified quantum group $\catQGdot=\catQGk$ is the idempotent completion of $\catQG$. More precisely, the morphism space $\catQGdot(\lambda,\mu)$ is the idempotent completion of $\catQG(\lambda,\mu)$. It is a graded $\kk$-linear 2-category. The starred variants $\catQGstar$ and $\catQGdotstar$ are defined by adding stars to the morphisms categories.

The split Grothendieck group $K_0(\catQGdot)$ of $\catQGdot$ is a locally unital $\ZZ$-algebra:
\begin{equation}
    K_0(\catQGdot)=\bigoplus_{\lambda,\mu\in \ob(\catQG)}1_{\mu}K_0(\catQGdot(\lambda,\mu))1_{\lambda}
\end{equation}

\noindent with multiplication induced by horizontal composition. There is an isomorphism of locally unital $\ZZ$-algebras
\begin{equation}
    \QGdotZ\longrightarrow K_0(\catQGdot).
\end{equation}

\noindent where $\QGdotZ$ is the integral idempotent quantum group from \S\ref{subsec:QG}. 

This map exists and is surjective by \cite[Theorem~1.1]{KL10} and injectivity is equivalent to non-degeneracy of the graphical calculus by Theorem~1.2 in loc. cit. In finite-type non-degeneracy follows from the decategorification of cyclotomic quotients proved by \cite{KK12} and \cite{Web16} (c.f. \S\ref{subsec:cyclo}).

\subsection{Bubbles and symmetric functions}\label{subsec:BubblesInU}

We will use the following shorthand for bubbles:
\begin{equation}\label{eq:SpadeNotation}
        \begin{tikzpicture}[scale=.8,very thick, baseline=(current bounding box.center)]
        \draw[postaction={decorate,decoration={markings,
    mark=at position .5 with {\arrow[scale=1.3]{<}}}},very thick] (0,1) circle (15pt);
        \node at (-.54,1.65){$i$};
        \draw[fill] (.38,1.35) circle[radius=2pt];
        \node[scale=.7] at (.85,1.65){$\spadesuit+r$};
        \node[scale=1.2] at (1,.8){$\lambda$};
        \end{tikzpicture}
         := 
        \quad
        \begin{tikzpicture}[scale=.8,very thick, baseline=(current bounding box.center)]
        \draw[postaction={decorate,decoration={markings,
    mark=at position .5 with {\arrow[scale=1.3]{<}}}},very thick] (0,1) circle (15pt);
        \node at (-.54,1.65){$i$};
        \draw[fill] (.38,1.35) circle[radius=2pt];
        \node[scale=.7] at (1.2,1.6){$\langle i, \lambda \rangle - 1+r$};
        \node[scale=1.2] at (1,.8){$\lambda$};
        \end{tikzpicture} \qquad
    \begin{tikzpicture}[scale=.8,very thick, baseline=(current bounding box.center)]
        \draw[postaction={decorate,decoration={markings,
    mark=at position .5 with {\arrow[scale=1.3]{>}}}},very thick] (0,1) circle (15pt);
        \node at (-.54,1.65){$i$};
        \draw[fill] (.38,1.35) circle[radius=2pt];
        \node[scale=.7] at (.85,1.65){$\spadesuit+r$};
        \node[scale=1.2] at (1,.8){$\lambda$};
        \end{tikzpicture}  := \quad
        \begin{tikzpicture}[scale=.8,very thick, baseline=(current bounding box.center)]
        \draw[postaction={decorate,decoration={markings,
    mark=at position .5 with {\arrow[scale=1.3]{>}}}},very thick] (0,1) circle (15pt);
        \node at (-.54,1.65){$i$};
        \draw[fill] (.38,1.35) circle[radius=2pt];
        \node[scale=.7] at (1.2,1.6){$-\langle i,\lambda \rangle -1 + r$};
        \node[scale=1.2] at (1,.8){$\lambda$};
        \end{tikzpicture}
\end{equation}

\noindent where $\lambda\in X$, $i\in I$, and $r\in\ZZ$. When $r>|\langle i,\lambda\rangle|$ these bubbles have degree $2r$. When $r\leq |\langle i,\lambda\rangle|$ the number of dots is negative and so this doesn't make sense as a 2-morphism. We resolve this by adding this ``fake bubble'' as a new generator of degree $2r$ and impose the following relations: bubbles of negative degree are zero, degree zero bubbles satisfy
\begin{equation}
    \begin{tikzpicture}[scale=.7,very thick, baseline=(current bounding box.center)]
        \draw[postaction={decorate,decoration={markings,
    mark=at position .6 with {\arrow[scale=1.3]{<}}}},very thick] (0,0) circle (15pt);
        \node at (-.54,.65){$i$};
        \draw[fill] (.38,.35) circle[radius=2pt];
        \node[scale=.7] at (1.03,.55){$\spadesuit+0$};
        \node[scale=1.2] at (1.5,-.2){$\lambda$};
        \end{tikzpicture}\ 
        =c_{i,\lambda} 1_{\lambda}, 
        \hspace{1in}
        \begin{tikzpicture}[scale=.7,very thick, baseline=(current bounding box.center)]
        \draw[postaction={decorate,decoration={markings,
    mark=at position .6 with {\arrow[scale=1.3]{>}}}},very thick] (0,0) circle (15pt);
        \node at (-.54,.65){$i$};
        \draw[fill] (.38,.35) circle[radius=2pt];
        \node[scale=.7] at (1.03,.55){$\spadesuit+0$};
        \node[scale=1.2] at (1.5,-.2){$\lambda$};
        \end{tikzpicture}\ 
        =c_{i,\lambda}^{-1}1_{\lambda},
\end{equation}

\noindent and higher degree bubbles satisfy the equations arising from the homogeneous terms in $x$ of the infinite Grassmannian equation:
\begin{equation}\label{eq:InfGrass}
\left(
\begin{tikzpicture}[very thick, baseline=(current bounding box.center)]
        \node[scale=1.4] at (-1.25,0){$\displaystyle \sum_{r=0}^\infty$};
        \draw[postaction={decorate,decoration={markings,
    mark=at position .6 with {\arrow[scale=1.3]{>}}}},very thick,scale=.66] (0,0) circle (15pt);
        \node at (-.34,.45){$i$};
        \draw[fill] (.23,.25) circle[radius=1.5pt];
        \node[scale=.7] at (.73,.4){$\spadesuit+r$};
        \node[scale=1.3] at (.63,-.3){$\lambda$};
        \node[scale=1.2] at (1.4,0){$x^r$};
        \end{tikzpicture} \right)
\left(
\begin{tikzpicture}[very thick, baseline=(current bounding box.center)]
        \node[scale=1.4] at (-1.25,0){$\displaystyle \sum_{s=0}^\infty$};
        \draw[postaction={decorate,decoration={markings,
    mark=at position .6 with {\arrow[scale=1.3]{<}}}},very thick,scale=.66] (0,0) circle (15pt);
        \node at (-.34,.45){$i$};
        \draw[fill] (.23,.25) circle[radius=1.5pt];
        \node[scale=.7] at (.73,.4){$\spadesuit+s$};
        \node[scale=1.3] at (.63,-.3){$\lambda$};
        \node[scale=1.2] at (1.4,0){$x^s$};
        \end{tikzpicture} \right)
    =\ 1_{\lambda}.
\end{equation}

\noindent From this it follows that all fake bubbles can actually be written in terms of real bubbles. 

Recall the graded algebra $\Pi=\bigotimes_{i\in I}\Sym$ from \S\ref{subsec:SymmetricFunc}. For any $\lambda\in X$, there is a homomorphism of graded algebras
\begin{equation}
    b_{\lambda}:\Pi\longrightarrow \catQGdotstar(1_{\lambda},1_{\lambda})
\end{equation}
    
\noindent sending
\begin{equation} 
    (-1)^re_{i,r}\longmapsto c_{i,\lambda}~~
    \begin{tikzpicture}[very thick, baseline=(current bounding box.center)]
        \draw[postaction={decorate,decoration={markings,
        mark=at position .5 with {\arrow[scale=1.3]{>}}}},very thick] (0,.5) circle (15pt);
        \node at (-.5,1.15){$i$};
        \draw[fill] (.38,.85) circle[radius=2pt];
        \node[scale=.8] at (.93,1.15){$\spadesuit+r$};
        \node[scale=1.3] at (1,.5){$\lambda$};
    \end{tikzpicture}
    \qquad h_{i,r}\longmapsto c_{i,\lambda}^{-1}~~
    \begin{tikzpicture}[very thick, baseline=(current bounding box.center)]
        \draw[postaction={decorate,decoration={markings,
        mark=at position .5 with {\arrow[scale=1.3]{<}}}},very thick] (0,.5) circle (15pt);
        \node at (-.5,1.15){$i$};
        \draw[fill] (.38,.85) circle[radius=2pt];
        \node[scale=.8] at (.93,1.15){$\spadesuit+r$};
        \node[scale=1.3] at (1,.5){$\lambda$};
    \end{tikzpicture}
\end{equation}

\noindent By non-degeneracy, this is an isomorphism.

\begin{remark}
    In \cite[\S3.4.4]{Lau12} and \cite[\S5.1]{CL15} the authors identify elementary symmetric fuctions with \emph{clockwise bubbles}. Our homomorphism differs from theirs by the automorphism of $\Pi$ interchanging $(-1)^re_{i,r}$ and $h_{i,r}$ and fixing $p_{i,r}$. We believe our choice is more natural given the relationship between bubbles and the deformed cyclotomic relation \eqref{eq:defcyclorel}.
\end{remark}

\begin{remark}
    Observe that we now have isomorphisms relating symmetric functions, bubbles, and the Cartan subalgebra of $\g$. More precisely, for any $\lambda\in X$ there are isomorphisms of graded algebras:
    \begin{equation}
        \begin{tikzcd}
            & \Pi \ar[ld] \ar[rd] &  \\ 
            1_{\lambda}U(\h[t])1_{\lambda} \ar[rr] &  & \catQGdotstar(1_{\lambda},1_{\lambda})
        \end{tikzcd}
    \end{equation}

    \noindent where the left diagonal map was defined in \ref{subsec:SymmetricFunc}. The horizontal map descends to the restriction of the isomorphism $\currdot\cong\Tr(\catQGstar)$ in the trace. 
\end{remark}

It will be convenient for us to consider the image under $b_{\lambda}$ of the generating functions $e_i(x)$, $h_i(x)$, and $p_i(x)$ from \S\ref{subsec:SymmetricFunc}. For example, the relation \eqref{eq:InfGrass} is just the image under $b_{\lambda}$ of $h_i(x)e_i(-x)=1$. 

Generating functions also greatly simplify the statement of bubble slide equations. To save space, we do not state these diagrammatically. Instead for $i\in\pm I$ we use $1_{\E_{i}}$  to denote a string labelled by $\lvert i\rvert$, oriented upward if $i\in +I$ and downward if $i\in -I$, and let $y$ to denote a dot on that string. Recall that ``$\circ$'' denotes horizontal composition in $\catQGdot$.

The reader is invited to compare these with the equations \eqref{eq:GenFuncVariables}.

\begin{lemma}\label{Lem:SlidesCatQG}
    Take $i\in \pm I$, $j\in I$, and $\lambda\in X$. Then
    \begin{align}
        \begin{split}
            b_{\lambda+\alpha_i}(e_j(x))\circ1_{\E_{i}}  &= (1+yx)^{\langle j,i\rangle}\circ b_{\lambda}(e_j(x)), \\
            b_{\lambda+\alpha_i}(h_j(x))\circ1_{\E_{i}}  &= (1-yx)^{-\langle j,i\rangle}\circ b_{\lambda}(h_j(x)), \\
            b_{\lambda+\alpha_i}(p_j(x))\circ1_{\E_{i}}  &= 1_{\E_{i}}\circ b_{\lambda}(p_j(x))+ \langle j,i\rangle \frac{yx}{1-yx}.
        \end{split}
    \end{align}
\end{lemma}

\begin{proof}
    The first two follow directly from \cite[\S3.2]{BHLW16}. The third follows from these, the coproduct $\delta$, and \eqref{eq:GenFuncVariables}.
\end{proof}

\subsection{2-representations and cyclotomic quotients}\label{subsec:cyclo}

In this subsection we recall the definition of a 2-representation of $\catQGdot$ and 2-natural transformations between them. We also recall the undeformed and deformed cyclotomic quotients of $\catQGdot$ - 2-representations of $\catQGdot$ that categorify irreducible modules over the quantum group.

\subsubsection{2-representations}

A 2-representation of $\catQGdot$ on a category $\mc{M}$ consists of a weight decomposition $\mc{M}=\bigoplus_{\lambda\in X}\mc{M}(\lambda)$ into subcategories and compatible functors from $\catQGdot(\lambda,\mu)$ to the category of functors $\mc{M}(\lambda)\to\mc{M}(\mu)$. A 2-representation is graded if $\mc{M}$ is graded and all the functors are graded. This induces a 2-representation of $\catQGdotstar$ on $\mc{M}^*$.

A graded 2-representation of $\catQGdot$ on $\mc{M}$ induces a $K_0(\catQGdot)\cong \QGdotZ$-module structure on
\begin{equation}
    K_0(\mc{M})=\bigoplus_{\lambda\in\ob(\catQG)}1_{\lambda}K_0(\mc{M}(\lambda))
\end{equation}

\noindent compatible with the action of $\ZZ[q^{\pm 1}]$.

A 2-natural transformation $\eta$ between 2-representations on $\mc{M}$ and $\mc{N}$ consists of functors
\begin{equation}
    \eta(\lambda):\mc{M}(\lambda)\longrightarrow \mc{N}(\lambda)
\end{equation}

\noindent for all $\lambda\in X$ together with compatible natural isomorphisms
\begin{equation}
    \eta(x):x_{\mc{N}}\circ\eta(\lambda)\Rightarrow \eta(\mu)\circ x_{\mc{M}}
\end{equation}

\noindent for any 1-morphism $x\in\ob(\catQG(\lambda,\mu))$, where $x_{\mc{M}}$ (resp.\ $x_{\mc{N}}$) denotes the functor $\mc{M}(\lambda)\to \mc{M}(\mu)$ (resp.\ $\mc{N}(\lambda)\to\mc{N}(\mu)$) associated to $x$. We call $\eta$ a 2-natural isomorphism if the $\eta(\lambda)$ are all equivalences. 

If $\mc{M}$, $\mc{N}$ and all $\eta(\lambda)$ and $\eta(x)$ are graded we call $\eta$ a graded 2-representation. This induces a 2-representation from $\mc{M}^*$ to $\mc{N}^*$. A graded 2-representation of $\catQGdot$ induces $\QGdotZ$-module homomorphism $K_0(\mc{M})\to K_0(\mc{N})$. These are isomorphisms if $\eta$ is a 2-natural isomorphism.

\subsubsection{Cyclotomic quotients}

If $\lambda\in X^+$, the corresponding \emph{deformed cyclotomic quotient} $\defcyclolam=\defcyclolamk$ of $\catQGdot$ is the graded category obtained from $\bigoplus_{\mu\in X}\catQGdot(\lambda,\mu)$ by imposing the following global relation for any $i\in I$:
\begin{equation}
    \begin{tikzpicture}[very thick, baseline=(current bounding box.center)]
        \node at (0,0){$\cdots$};
        \draw[->] (.5,-.5)--(.5,.5);
        \node[scale=1.2] at (1,.2){$\lambda$};
        \node at (.5,-.7){$i$};
    \end{tikzpicture}\hspace{.1in}
    =0;
\end{equation}

\noindent that is, any diagram with an upward string at the far right is equal to zero. This preserves the direct sum decomposition $\defcyclolam=\bigoplus_{\mu\in X}\defcyclolam(\mu)$ according to the left-most weight, and horizontal composition in $\catQGdot$ induces an action of $\catQGdot$ on $\defcyclolam$ by placing a diagram on the left. So $\defcyclolam$ is a graded 2-representation of $\catQGdot$.

Taking $s=0$ in \eqref{eq:CurlRel} yields the deformed cyclotomic relation in $\defcyclolam$:
\begin{equation}\label{eq:defcyclorel}
    \begin{tikzpicture}[very thick, baseline=(current bounding box.center)]
        \node[scale=1.3] at (-2,0){$\displaystyle \sum_{r=0}^{\langle i,\lambda \rangle}$};
        \draw[->] (-1,.5)--(-1,-.5);
        \filldraw (-1,0) circle (1.5pt);
        \node[scale=.8] at (-.3,.1) {$\langle i,\lambda \rangle - r$};
        \node at (-1,-.75){$i$};
        \draw[postaction={decorate,decoration={markings,
         mark=at position .6 with {\arrow[scale=1.3]{>}}}},very thick,scale=.66] (1.5,0) circle (15pt);
        \node at (.76,.45){$i$};
        \draw[fill] (1.23,.25) circle[radius=1.5pt];
        \node[scale=.7] at (1.63,.4){$\spadesuit+r$};
        \node[scale=1.2] at (1.9,-.3){$\lambda$};
    \end{tikzpicture} \ = \ 
    \begin{tikzpicture}[very thick, baseline=(current bounding box.center)]
        \draw[<-] (1,0) to [out=90, in=85](1.05,0.5);
        \draw (1.05,0.5) arc (-175:175:2mm);
        \draw (1.05,0.46) to [out=95, in=270] (1,1);
        \node at (1,-.25){$i$};
        \node[scale=1.2] at (1.6,0){$\lambda$};
    \end{tikzpicture} \ = \ 0.
\end{equation}

The \emph{undeformed cyclotomic quotient} $\cyclolam=\cyclolamk$ of $\catQGdot$ is obtained from $\defcyclolam$ by setting any diagram with a positive degree bubble at the far right equal to zero. In $\cyclolam$ we have the undeformed cyclotomic relation:
\begin{equation}\label{eq:UndefCycloRel}
    \begin{tikzpicture}[very thick, baseline=(current bounding box.center)]
        \node at (-3,.5){$\cdots$};
        \draw[->] (-2.5,1)--(-2.5,0);
        \node at (-2.5,-.25){$i$};
        \filldraw (-2.5,.5) circle (1.5pt);
        \node[scale=.8] at (-2,.5){$\langle i, \lambda \rangle$};
        \node[scale=1.2] at (-2,1){$\lambda$};
    \end{tikzpicture} \ = \ 0
\end{equation}

\noindent at the far right of any diagram. The weight decomposition and action of $\catQGdot$ on $\defcyclolam$ are preserved, so $\cyclolam$ is also a graded 2-representation of $\catQGdot$.

It was proved independently by \cite[Corollary~3.22]{Web17} and \cite[Theorem~6.2]{KK12} that for any $\lambda\in X^+$ there are isomorphisms of $\QGdotZ$-modules:
\begin{equation}\label{eq:CyclotomicGG}
    K_0(\defcyclolam)\cong K_0(\cyclolam)\cong \VZq{\lambda},
\end{equation}

\noindent where $\VZq{\lambda}$ is the irreducible $\QGdotZ$-module from \S\ref{subsec:QG}.

\section{Bubble slides in the tensor product algebra}\label{sec:DeformedTPC}

Recall from \S\ref{subsec:TPWeylMods} that we have fixed $n\in \NN$ and a sequence $\ulambda=(\lam{1},\ldots ,\lam{n})$ of dominant weights and that, if $k\in [1,n]$ and $i\in I$, then $\lam{k}_i=\langle i,\lam{k}\rangle$. The corresponding graded algebra $\ATP$ consists of certain symmetric polynomials in a set $Z$ of indeterminates.

In this section we recall Webster's categorification $\TP$ of a tensor product $\VZqulam$ of irreducibles for $\QGdotZ$ and its deformation $\defTP$, and prove equations for passing bubbles through red strings in $\defTP$ analogous to the bubble slide equations in $\catQGdot$ (see \S\ref{subsec:BubblesInU}). Note that to match with the conventions of \cite{BHLW16} our categories differ from Webster's by a reflection in a vertical line. Also, since we work with a fixed sequence $\ulambda$ of dominant weights we will label red strings in $\TP$ by $(k)$ with $k\in [1,n]$ rather than by the actual weights $\lam{k}$.

The category $\TP$ is equivalent to the category of graded projective modules over the tensor product algebra $T^{\ulambda}$ from \cite{Web17}, but it is more convenient for us to define $\TP$ by generators and relations as in \cite{Web15}. Morphisms spaces in $\TP$ are spanned by string diagrams containing red strings which separate tensor factors and there is a graded 2-representation of $\catQGdot$ on $\TP$ by placing diagrams on the left and composing.

The category $\defTP$ is equivalent to the category of graded projective modules over the algebra in \cite[\S3.2]{Web12}. It is obtained by deforming the defining relations for $\TP$ over $\ATP$ such that setting all $z\in Z$ equal to zero recovers $\TP$. The effect of this is that dots in $\defTP$, rather than being nilpotent, can have any $z\in Z$ as a generalized eigenvalue (see \S\ref{sec:unfurling}). The relationship between $\TP$ and $\defTP$ is analogous to that between the undeformed $\cyclolam$ and deformed $\defcyclolam$ cyclotomic quotients of $\catQGdot$. There is a graded 2-representation of $\catQGdot$ on $\defTP$ and morphism spaces in the starred category $\defTPstar$ are graded right $\ATP$-modules.

In \S\ref{subsec:TPSymFun} we prove new equations for passing bubbles through red strings in $\defTP$. We state these in terms of the coproduct $\delta$ on $U(\h[t])$ which shows that the actions of $\ATP$ on $\WWulam$ and $\defTPstar$ are compatible. We also the bubble slides in terms of the generating functions $e_i(x)$, $h_i(x)$, and $p_i(x)$ for symmetric functions. This allows us to study the spectrum of a dot in $\TPstar$ in \S\ref{sec:unfurling}.

\subsection{Tensor product algebras}

\subsubsection{An auxiliary category}\label{subsec:TPC}

We begin by defining an auxiliary graded $\kk$-linear category $\TPaux$ from which both $\TP$ and $\defTP$ can be obtained. We will not need $\TPaux$ outside this subsection.

Objects in $\TPaux$ are formal direct sums of grading shifts of \emph{Stendhal pairs} $(\ui,\kappa)$, where $\ui\in (\pm I)^m$ for some $m\in \NN$ and $\kappa$ is a weakly increasing function from $[1,2,\ldots,n+1]$ to $[0,1,\ldots, m]$ with $\kappa(n+1)=m$ (the equivalent objects in \cite{Web15} and \cite{Web17} are called tricolore quadruples and double Stendhal triples respectively).

Morphisms in $\TPaux$ are $\kk$-linear combinations of grading shifts of \emph{Stendhal diagrams} (double Stendhal diagrams in \cite{Web17}). A Stendhal diagram consists of finitely many strings in $\RR\times [0,1]$. Each string is either:
\begin{itemize}
    \item coloured black, given an orientation, labelled with an element $i\in I$, and decorated with finitely many dots; or
    \item coloured red and labelled with $(k)$ for $k\in [1,n]$.
\end{itemize}

\noindent Diagrams have no triple points or tangencies and any open end of a string must meet one of the lines $y=0$ or $y=1$ at a distinct point from all other strings. Red strings have no critical points (that is, they never turn back on themselves) and two red strings cannot cross. Reading the labels of red strings from left to right as the intersect a horizontal line $y=c$ must give the sequence $(n),(n-1),\ldots ,(1)$.

Regions between strings are labelled by weights. The right-most region is labelled by 0 and the labels of the other regions are determined by the consistency rules in \eqref{eq:WeightsAcrossStrands} and the additional condition:
\begin{equation}\label{LabellingAcrossStrands}
     \begin{tikzpicture}[baseline,very thick]
        \draw[wei,postaction={decorate,decoration={markings,
        mark=at position .5 with
        }}] (0,-.5) -- node[below,at start]{$(k)$}  (0,.5);
        \node at (-1,0) {$\mu+\lam{k}$};
        \node at (1,.05) {$\mu$};
    \end{tikzpicture}
\end{equation}

\noindent Since the labelling of regions is uniquely determined we will often not record it.

Take a Stendhal diagram. As with 2-morphisms in $\catQG$, recording the labels and orientations of the black strings along the bottom of the diagram yields a signed sequence $\ui=(i_1,\ldots ,i_m)\in (\pm I)^m$. For $k\in[1,n]$, let $\kappa(k)\in \NN$ be the number of black strings to the right of the red $(k)$-string reading along the line $y=0$, and let $\kappa(n+1)=m$. This defines a weakly increasing function from $[1,\ldots,n+1]$ to $[0,\ldots ,m]$, so $(\ui,\kappa)$ is a Stendahl pair. Similarly, reading the top of the diagram we get an Stendhal pair $(\ui',\kappa')$. We assign the diagram a degree $d$ by taking the sum of the degrees of elementary diagrams: in addition to those set in $\catQG$, such as \eqref{eq:DegDiagrams3}, we let
\begin{equation}
    \deg
    \begin{tikzpicture}[very thick,baseline=(current bounding box.center)]
        \draw[<-] (0,1)--(1,0);
        \node at (1,-.3){$i$};
        \draw[-,wei] (1,1)--(0,0);
        \node at (0,-.3){$(k)$};
    \end{tikzpicture}\
    =\ \deg 
    \begin{tikzpicture}[very thick,baseline=(current bounding box.center)]
        \draw[<-] (1,1)--(0,0);
        \node at (0,-.3){$i$};
        \draw[-,wei] (0,1)--(1,0);
        \node at (1,-.3){$(k)$};
    \end{tikzpicture}\ = \ 0,
\end{equation}
\begin{equation}
    \deg 
    \begin{tikzpicture}[very thick,baseline=(current bounding box.center)]
        \draw[->] (1,1)--(0,0);
        \node at (0,-.3){$i$};
        \draw[-,wei] (0,1)--(1,0);
        \node at (1,-.3){$(k)$};
    \end{tikzpicture}\
    =\ 
    \deg \begin{tikzpicture}[very thick,baseline=(current bounding box.center)]
        \draw[->] (0,1)--(1,0);
        \node at (1,-.3){$i$};
        \draw[-,wei] (1,1)--(0,0);
        \node at (0,-.3){$(k)$};
    \end{tikzpicture}\
    =\ \lam{k}_i.
\end{equation}

\noindent This diagram is a morphism from $(\ui,\kappa)$ to $(\ui',\kappa')\langle d\rangle$. Composition of morphisms is induced by vertical composition of diagrams as in \S\ref{subsec:DefCatQG}.

We sometimes conflate a Stendhal pair $S$ and the identity $1_S$ on $S$ and refer to, for example, the red string in $S$ labelled by $(k)$.

Morphisms in $\TPaux$ are subject to local relations: black strings satisfy the relations of $\catQG$ (see \S\ref{subsec:DefCatQG}) and for any $i\in I$ and $k\in [1,n]$ the following hold as well as their reflections in a vertical line:
\begin{equation}\label{eq:DotsAndRed}
    \begin{tikzpicture}[very thick,baseline,scale=.5]
        \draw[<-](-3,0) +(-1,-1) -- +(1,1) node[below,at start]{$i$};
        \draw[wei](-3,0) +(1,-1) -- +(-1,1) node[below,at start]{$(k)$};
        \fill (-3.5,-.5) circle (5pt);
        \node at (-1,0) {=};
        \draw[<-](1,0) +(-1,-1) -- +(1,1) node[below,at start]{$i$};
        \draw[wei](1,0) +(1,-1) -- +(-1,1) node[below,at start]{$(k)$};
        \fill (1.5,.5) circle (5pt);
    \end{tikzpicture} \hspace{1cm}
    \begin{tikzpicture}[very thick,baseline,scale=.5,x={(-1,0)},y={(0,1)}]
        \draw[->](-3,0) +(-1,-1) -- +(1,1) node[below,at start]{$i$};
        \draw[wei](-3,0) +(1,-1) -- +(-1,1) node[below,at start]{$(k)$};
        \fill (-3.5,-.5) circle (5pt);
        \node at (-1,0) {=};
        \draw[->](1,0) +(-1,-1) -- +(1,1) node[below,at start]{$i$};
        \draw[wei](1,0) +(1,-1) -- +(-1,1) node[below,at start]{$(k)$};
        \fill (1.5,.5) circle (5pt);
    \end{tikzpicture}
\end{equation}
\begin{equation}\label{eq:UpAndRed}
    \begin{tikzpicture}[very thick,baseline=(current bounding box.center),scale=.5]
        \draw[->] (-2.8,0)  +(0,-1) .. controls (-1.2,0) ..  +(0,1) node[below,at start]{$i$};
        \draw[wei] (-1.2,0)  +(0,-1) .. controls (-2.8,0) ..  +(0,1) node[below,at start]{$(k)$};
        \node at (-.3,0) {=};
        \draw[wei] (3.3,0)  +(0,-1) -- +(0,1) node[below,at start]{$(k)$};
        \draw[->] (2.1,0)  +(0,-1) -- +(0,1) node[below,at start]{$i$};
    \end{tikzpicture} 
    %
\end{equation}

\noindent We also impose the relations \cite[(4.4a)-(4.4c)]{Web17} and their reflections in a vertical line for any labelling of red and black strings (the reader can ignore the orientation on red strings). Finally we also set to zero any \emph{violated} diagrams; that is, diagrams which at some horizontal slice $y=c$ have a black string to the right of all red strings.

\subsubsection{Tensor product algebras}\label{subsec:DefTPC}

The following category was introduced in \cite[\S4]{Web17}, where it is presented as the category of graded projective modules over an algebra:

\begin{definition}\label{def:TPcat}
    Let $\TP$ be the idempotent completion of the category obtained from $\TPaux$ by imposing the following additional local relations as well as the reflection of \eqref{eq:UndeformedTPRels1} in a vertical line:
    \begin{equation}\label{eq:UndeformedTPRels1}
        \begin{tikzpicture}[very thick,baseline=(current bounding box.center),scale=.5]
            \draw[wei] (-1.2,0)  +(0,-1) -- +(0,1) node[below,at start]{$(k)$};
            \draw[<-] (-2.5,0)  +(0,-1) -- +(0,1) node[below,at start]{$i$};
            \fill (-2.5,0) circle (5pt) node[left=3pt]{$\lam{k}_i$};
            \node at (-.1,0) {=};
            \draw[<-] (1,0)  +(0,-1) .. controls (2.2,0) ..  +(0,1) node[below,at start]{$i$};
            \draw[wei] (2.2,0)  +(0,-1) .. controls (1,0) ..  +(0,1) node[below,at start]{$(k)$};
        \end{tikzpicture}
        %
    \end{equation}
    \begin{equation}\label{eq:UndeformedTPRels2}
        \begin{tikzpicture}[very thick,baseline,scale=.5]
            \draw[<-] (-3,0)  +(1,-1) -- +(-1,1) node[at start,below]{$j$};
            \draw[<-] (-3,0) +(-1,-1) -- +(1,1)node [at start,below]{$i$};
            \draw[wei] (-3,0)  +(0,-1) .. controls (-2,0) .. node[below, at start]{$(k)$}  +(0,1);
            \node at (-1,0) {=};
            \draw[<-] (1,0)  +(1,-1) -- +(-1,1) node[at start,below]{$j$};
            \draw[<-] (1,0) +(-1,-1) -- +(1,1) node [at start,below]{$i$};
            \draw[wei] (1,0) +(0,-1) .. controls (0,0) ..  node[below, at start]{$(k)$} +(0,1);   
            \node at (2.6,0) {$+ $};
            \draw[<-] (10.85,0)  +(1,-1) -- +(1,1) node[midway,circle,fill,inner sep=1.8pt,label=right:{$q$}]{} node[at start,below]{$i$};
            \draw[<-] (10.35,0) +(-1,-1) -- +(-1,1) node[midway,circle,fill,inner sep=1.8pt,label=left:{$p$}]{} node [at start,below]{$i$};
            \draw[wei] (10.6,0) +(0,-1) -- node[below, at start]{$(k)$} +(0,1);
            \node[scale=1.3] at (6,-.6){$\displaystyle \delta_{i,j}\hspace{-1em}\sum_{p+q=\lam{k}_i-1} $}  ;
        \end{tikzpicture}
    \end{equation}
\end{definition}

\noindent It is a graded $\kk$-linear category. There is a weight decomposition $\TP=\bigoplus\TP(\mu)$ according to the left-most weight $\mu\in X$ of a diagram, and placing diagrams in $\catQGdot$ on the left of those in $\TP$ defines a graded 2-representation of $\catQGdot$ on $\TP$.

The 2-representation of $\catQGdot$ on $\TP$ induces a $\QGdotZ$-action on the split Grothendieck group $K_0(\TP)$. By \cite[Theorems~4.38]{Web16} there is an isomorphism of $\QGdotZ$-modules
\begin{equation}
    K_0(\TP)\longrightarrow \VZqulam,
\end{equation}

\noindent where $\VZqulam$ is the product of irreducible modules from \S\ref{subsec:QG}. To reflect this, we will sometimes refer to $\TP$ as a \emph{tensor product categorification}.

\subsubsection{Deformed tensor product algebras}

Recall the graded algebras $\Pi=\bigotimes_{i\in I}$ and $\ATP$ and the projections $a_k:\Pi\to\bA_{\lam{k}}\subseteq\ATP$ from \S\ref{subsec:TPC}. Let $\TPaux\otimes\ATP$ denote the extension of scalars of $\TPaux$ to $\ATP$. This is a graded category and $\TPauxstar$ is enriched over graded right $\ATP$-modules.

Morphisms in $\TPaux\otimes\ATP$ are $\ATP$-linear combinations of Stendhal diagrams. If $D$ is a Stendhal diagram and $a\in \ATP$ then we will represent the morphism $D\cdot a$ in $\TPaux\otimes\ATP$ diagrammatically by placing $a$ in a box somewhere in $D$. The placement of the box does not change the morphism.

The following was introduced in \cite[\S3.2]{Web12} where it was presented as a category of graded projective modules over an algebra:

\begin{definition}\label{def:DeformedTPC}
    Let $\defTP$ be the idempotent completion of the category obtained from $\TPaux\otimes\ATP$ by imposing the following additional local relations as well as the reflection of \eqref{eq:DeformedRels1} in a vertical line:
    \begin{equation}\label{eq:DeformedRels1} 
        \begin{tikzpicture}[very thick,baseline=(current bounding box.center),rounded corners, scale=1.05]
            \node[scale=1.3] at (-2,-.2){$\displaystyle\sum_{p+q=\lam{k}_i}$};
            \draw (-1,.3) rectangle (.5,-.3) node[midway]{$a_k(e_{i,p})$};
            \draw[<-] (1.1,-.5)--(1.1,.5) node[at start,below]{$i$};
            \draw[fill] (1.1,0) circle[radius=2pt] node[left=3pt]{$q$};
            \draw[wei] (1.7,-.5)--(1.7,.5) node[at start,below]{$(k)$};
            \node at (2.2,0){$=$};
            \draw[<-] (2.7,-.5) .. controls (3.4,0) .. (2.7,.5) node[at start,below]{$i$};
            \draw[wei] (3.2,-.5) .. controls (2.6,0) .. (3.2,.5) node[at start,below]{$(k)$};
        \end{tikzpicture}
    \end{equation}
    \begin{equation}\label{eq:DeformedRels3}
        \begin{tikzpicture}[very thick, baseline=(current bounding box.center),rounded corners,scale=1.05]
            \draw[->] (-5,.5)--(-4,-.5) node[at end,below]{$j$};
            \draw[->] (-4,.5)--(-5,-.5) node[at end,below]{$i$};
            \draw[wei] (-4.5,-.5) .. controls (-4,0) .. (-4.5,.5) node[at start,below]{$(k)$};
            \node at (-3.5,0){$=$};
            \draw[->] (-3,.5)--(-2,-.5) node[at end,below]{$j$};
            \draw[->] (-2,.5)--(-3,-.5) node[at end,below]{$i$}; 
            \draw[wei] (-2.5,-.5) .. controls (-3,0) .. (-2.5,.5) node[at start,below]{$(k)$};
            \node at (-2,0){$+$};
            \node[scale=1.3] at (.2,-.2){$\displaystyle \delta_{i,j}\hspace{-.5em} \sum_{r=0}^{\lam{k}_i-1} \hspace{-.5em} \sum_{\substack{p+q= \\ \lam{k}_i-r-1}}\hspace{-.6em}(-1)^r $}  ;
            \draw (2.25,.3) rectangle (3.6,-.3) node[midway]{$a_k(e_{i,r})$};
            \draw[->] (4,.5)--(4,-.5) node[at end,below]{$i$};
            \draw[fill] (4,0) circle[radius=2pt] node[left]{$p$};
            \draw[wei] (4.5,.5)--(4.5,-.5) node[at end,below]{$(k)$};
            \draw[->] (5,.5)--(5,-.5) node[at end, below]{$j$};
            \draw[fill] (5,0) circle[radius=2pt] node[right]{$q$};
        \end{tikzpicture}
    \end{equation}
\end{definition}

\noindent This a graded $\kk$-linear category and $\defTPstar$ is enriched over graded right $\ATP$-modules.

Observe that relations \eqref{eq:DeformedRels1}-\eqref{eq:DeformedRels3} reduce to the relations \eqref{eq:UndeformedTPRels1}-\eqref{eq:UndeformedTPRels2} in the undeformed category $\TP$ if we specialize $z=0$ for all $z\in Z$. So if $\kk$ denotes the unique simple graded $\ATP$-module (on which all $z\in Z$ act as zero) then tensoring morphism spaces over $\ATP$ with $\kk$ gives a functor from $\defTPstar$ to $\TPstar$. Restricting to degree zero morphisms gives a functor from $\defTP$ to $\TP$.

As with $\TP$, there is a weight decomposition $\defTP=\bigoplus\defTP(\mu)$ according to the left-most weight $\mu\in X$ of a diagram and placing diagrams in $\catQGdot$ on the left gives a graded 2-representation of $\catQGdot$ on $\defTP$.

In \S\ref{sec:Flatness} we will show that the split Grothendieck group of $\defTP$ is isomorphic to $\VZqulam$ as a $\QGdotZ$-module, which justifies our referring to $\defTP$ as a \emph{deformed tensor product categorification}.

\subsection{Bubble slides}\label{subsec:TPSymFun}

Recall the algebra $\Pi=\bigotimes_{i\in I}\Sym$ from \S\ref{subsec:SymmetricFunc}. We can interpret symmetric functions in $\defTP$ either through the isomorphism with bubbles $b_{\mu}:\Pi\to\catQGstar(1_{\mu},1_{\mu})$ from \S\ref{subsec:BubblesInU} or through the projections $a_k:\Pi\to\bA_{\lam{k}}\subseteq\ATP$. Our description of how the two of these interact takes the form of bubble slides through red strings and mirrors the relationship between the actions of $U(\mf{h}[t])$ and $\ATP$ on the tensor product $\WWulam$ of global Weyl modules.

Recall the coproduct $\delta:\Pi\to\Pi\otimes \Pi$ from \S\ref{subsec:SymmetricFunc}.

\begin{proposition}\label{prop:SymFuncTPC}
    Take $\mu\in X$ and $k\in [1,n]$. If $f\in\Pi$ and $\delta(f)=\sum_s g_s\otimes g'_s$ then
    \begin{equation}\label{eq:SymFuncTPC}
        \begin{tikzpicture}[very thick, baseline=(current bounding box.center)]
            \draw[rounded corners] (-2.25,.3) rectangle (-.3,-.3) node[midway]{$b_{\mu+\lambda^{(k)}}(f)$};
            \draw[wei] (0,.5)--(0,-.5);
            \node at (0,-.7){$(k)$};
            \node[scale=1.3] at (.5,.5){$\mu$};
        \end{tikzpicture}
        \ \ = \ 
        \ 
        \begin{tikzpicture}[very thick, baseline=(current bounding box.center)]
            \draw[wei] (-.5,.5)--(-.5,-.5);
            \node[scale=1.3] at (-3,-.2){$\displaystyle \sum_s $};
            \draw[rounded corners] (-2.1,.3) rectangle (-.65,-.3) node[midway]{$a_k(g_s')$};
            \node at (-.5,-.7){$(k)$};
            \draw[rounded corners] (0,.3) rectangle (1.05,-.3) node[midway]{$b_\mu(g_s)$};
            \node[scale=1.3] at (1.4,.5){$\mu$};
        \end{tikzpicture} 
    \end{equation}
\end{proposition}

Take $i\in I$ and recall the generating functions $e_i(x)$, $h_i(x)$, and $p_i(x)$ from \S\ref{subsec:SymmetricFunc} and the set $Z_i^{(k)}$ of indeterminates from \S\ref{subsec:TPWeylMods}. In Lemma~\ref{Lem:SlidesCatQG} we stated bubble slides in $\catQGdotstar$ for $e_i(x)$, $h_i(x)$, and $p_i(x)$. The proposition yields the analogous equations in $\defTPstar$ (the products and sums below are taken over all $z\in Z_i^{(k)}$):
\begin{equation}\label{eq:SlidesTPGenFunc1}
    \begin{tikzpicture}[very thick, baseline=(current bounding box.center)]
            \draw[rounded corners] (-2.35,.3) rectangle (-.2,-.3) node[midway]{$b_{\mu+\lambda^{(k)}}(e_i(x))$};
            \draw[wei] (0,.65)--(0,-.6) node[at end,below]{$(k)$};
            \node[scale=1.3] at (.5,.5){$\mu$};
        \end{tikzpicture}
        \ \ = \ 
        \ 
        \begin{tikzpicture}[very thick, baseline=(current bounding box.center)]
            \draw[wei] (-.5,.65)--(-.5,-.6) node[at end,below]{$(k)$};
            \node[scale=1.4] at (-3,-.3){$ $};
            \draw[rounded corners] (-3,.45) rectangle (-.65,-.45) node[midway]{$\displaystyle \prod (1+zx)$};
            \draw[rounded corners] (-.25,.35) rectangle (1.2,-.35) node[midway]{$b_\mu(e_i(x))$};
            \node[scale=1.3] at (1.4,.5){$\mu$};
        \end{tikzpicture} 
\end{equation}

\begin{equation}\label{eq:SlidesTPGenFunc2}
    \begin{tikzpicture}[very thick, baseline=(current bounding box.center)]
            \draw[rounded corners] (-2.45,.3) rectangle (-.2,-.3) node[midway]{$b_{\mu+\lambda^{(k)}}(h_i(x))$};
            \draw[wei] (0,.8)--(0,-.6) node[at end,below]{$(k)$};
            \node[scale=1.3] at (.5,.5){$\mu$};
        \end{tikzpicture}
        \ \ = \ 
        \ 
        \begin{tikzpicture}[very thick, baseline=(current bounding box.center)]
            \draw[wei] (-.5,.8)--(-.5,-.6) node[at end,below]{$(k)$};
            \node[scale=1.4] at (-3,-.3){};
            \draw[rounded corners] (-2.9,.55) rectangle (-.7,-.35) node[midway]{$\displaystyle\prod(1-zx)^{-1}$};
            \draw[rounded corners] (-.25,.45) rectangle (1.5,-.25) node[midway]{$b_\mu(h_i(x))$};
            \node[scale=1.3] at (1.7,.5){$\mu$};
        \end{tikzpicture} 
\end{equation}

\begin{equation}\label{eq:SlidesTPGenFunc3}
    \begin{tikzpicture}[very thick, baseline=(current bounding box.center)]
            \draw[rounded corners] (-2.45,.3) rectangle (-.2,-.3) node[midway]{$b_{\mu+\lambda^{(k)}}(p_i(x))$};
            \draw[wei] (0,.8)--(0,-.6) node[at end,below]{$(k)$};
            \node[scale=1.3] at (.5,.5){$\mu$};
        \end{tikzpicture}
        \ \ = \ 
        \ 
        \begin{tikzpicture}[very thick, baseline=(current bounding box.center)]
            \draw[wei] (-.5,.8)--(-.5,-.6) node[at end, below]{$(k)$};
            \draw[rounded corners] (-.3,.3) rectangle (1.25,-.3) node[midway]{$b_\mu(p_i(x))$};
            \node[scale=1.3] at (1.4,.5){$\mu$};
        \end{tikzpicture} 
        \ + \ 
        \begin{tikzpicture}[very thick, baseline=(current bounding box.center)]
            \draw[wei] (-.2,.8)--(-.2,-.6) node[at end,below]{$(k)$};
            \node[scale=1.4] at (-3,-.3){};
            \draw[rounded corners] (-2.8,.6) rectangle (-.65,-.45) node[midway]{$\displaystyle\sum \frac{zx}{1-zx}$};
            \node[scale=1.3] at (0.2,.5){$\mu$};
        \end{tikzpicture} 
\end{equation}

The idea behind the proof of the proposition is simple, but the reality is a little fiddly. Through explicit diagrammatic calculations based on the deformed relations in Definition~\ref{def:DeformedTPC}, we will show that \eqref{eq:SymFuncTPC} holds for all $f$ in a generating set for $\Pi$. The difficulty is that a priori the deformed relations only hold when the bubbles are \emph{real}, so the generating set we choose depends on the weight $\mu$.

\begin{lemma}\label{lem:BubblesSymFun}
    Take $\mu\in X$, $i\in I$, $k\in [1,n]$, and $s\geq 1$. 
    \begin{enumerate}[label=(\roman*)]
        \item\label{item:BubblesSymFun1} If $s>\lam{k}_i+\langle i,\mu\rangle$ then
       \begin{equation}\label{eq:BubblesSymFun1}
       \begin{tikzpicture}[very thick, baseline=(current bounding box.center)]
            \draw[postaction={decorate,decoration={markings,
            mark=at position .5 with {\arrow[scale=1]{>}}}},very thick] (0,0) circle (10pt);
            \node[scale=.67] at (-.5,.35){$i$};
            \draw[fill] (.23,.25) circle[radius=1.5pt];
            \node[scale=.8] at (.73,.45){$\spadesuit+s$};
            \draw[wei] (1.7,.5)--(1.7,-.5) node[at end, below]{$(k)$};
            \node[scale=1.5] at (2.2,.5){$\mu$};
            \node at (2.7,0){$=$};
            \node[scale=1.4] at (3.2,0){$\displaystyle \sum_{r=0}^s$};
            \node[scale=1] at (3.9,0){$(-1)^r$};
            \draw[rounded corners] (4.5,.3) rectangle (5.95,-.3) node[midway]{$a_k(e_{i,r})$};
            \draw[wei] (6.2,-.5)--(6.2,.5)node[at start, below]{$(k)$};;
            \draw[postaction={decorate,decoration={markings,
            mark=at position .5 with {\arrow[scale=1]{>}}}},very thick] (7.2,0) circle (10pt);
            \node[scale=.67] at (6.7,.35){$i$};
            \draw[fill] (7.43,.25) circle[radius=2pt];
            \node[scale=.8] at (7.93,.5){$\spadesuit+s-r$};
            \node[scale=1.5] at (8.3,-.1){$\mu$};
        \end{tikzpicture}
       \end{equation}
        \item\label{item:BubblesSymFun2} If $s>-\langle i,\mu\rangle$ then
        \begin{equation}\label{eq:BubblesSymFun2}
            \begin{tikzpicture}[very thick, baseline=(current bounding box.center)]
            \draw[postaction={decorate,decoration={markings,
            mark=at position .5 with {\arrow[scale=1]{<}}}},very thick] (1,0) circle (10pt);
            \node[scale=.67] at (.5,.35){$i$};
            \draw[fill] (1.23,.25) circle[radius=1.5pt];
            \node[scale=.8] at (1.73,.45){$\spadesuit+s$};
            \draw[wei] (.25,.5)--(.25,-.5) node[at end, below]{$(k)$};
            \node[scale=1.5] at (2.2,-.2){$\mu$};
            \node at (2.7,0){$=$};
            \node[scale=1.4] at (3.2,0){$\displaystyle \sum_{r=0}^s$};
            \node[scale=1] at (3.9,0){$(-1)^r$};
            \draw[rounded corners] (4.5,.3) rectangle (5.95,-.3) node[midway]{$a_k(e_{i,r})$};
            \draw[wei] (8.2,-.5)--(8.2,.5) node[at start, below]{$(k)$};;
            \draw[postaction={decorate,decoration={markings,
            mark=at position .5 with {\arrow[scale=1]{<}}}},very thick] (6.5,0) circle (10pt);
            \node[scale=.67] at (6.2,.35){$i$};
            \draw[fill] (6.75,.25) circle[radius=1.5pt];
            \node[scale=.8] at (7.43,.5){$\spadesuit+s-r$};
            \node[scale=1.5] at (8.8,.3){$\mu$};
        \end{tikzpicture}
        \end{equation}
        
        \item\label{item:BubblesSymFun3} If $-\lam{k}_i<\langle i,\mu\rangle<0$ and $s<\lam{k}_i$ then
        \begin{equation}\label{eq:BubblesSymFun3}
            \begin{tikzpicture}[very thick, baseline=(current bounding box.center)]
                    \node[scale=1.4] at (-2.5,0){$\displaystyle \sum_{r=0}^s$};
                    \node at (-1.8,0){$(-1)^r$};
                    \draw[rounded corners] (-1.3,.3) rectangle (.2,-.3) node[midway]{$a_k(e_{i,r})$};
                    \node[scale=1.4] at (.7,-.18){$\displaystyle \sum_{p+q=s-r}$};
                    \draw[postaction={decorate,decoration={markings,
                        mark=at position .5 with {\arrow[scale=1]{<}}}},very thick] (2,0) circle (10pt);
                    \node[scale=.67] at (1.5,.35){$i$};
                    \draw[fill] (2.23,.25) circle[radius=1.5pt];
                    \node[scale=.8] at (2.73,.35){$\spadesuit+p$};
                    \draw[wei] (3.43,.5)--(3.43,-.5)node[at end, below]{$(k)$};;
                    \draw[postaction={decorate,decoration={markings,
                        mark=at position .5 with {\arrow[scale=1]{>}}}},very thick] (4.43,0) circle (10pt);
                    \node[scale=.67] at (3.93,.35){$i$};
                    \draw[fill] (4.68,.25) circle[radius=1.5pt];
                    \node[scale=.8] at (5.13,.35){$\spadesuit+q$};
                    \node[scale=1.4] at (5.53,-.6){$\mu$};
                    \node at (5.93,0){$=$};
                    \node[scale=1.4] at (6.3,0){$0$};
            \end{tikzpicture}
        \end{equation}
        
    \end{enumerate}
\end{lemma}

\begin{proof}
    \ref{item:BubblesSymFun1} By the assumption on $s$, the bubble on the left is real, so we can pull the right edge through the red string and use the mirror image of deformed relation \eqref{eq:DeformedRels1} in a vertical line to pull the rest of the bubble through. This yields the right hand side of the equation, but with the sum running to $\lam{k}_i$. The difference between these two sums either involves negative degree bubbles, which are zero, or symmetric polynomials $a_k(e_{i,r})$ with $r>\lam{k}_i$, which are also zero. So the equation holds.
    \newline
        
    \noindent \ref{item:BubblesSymFun2} This is similar to \ref{item:BubblesSymFun1}, but we pull the bubble left through the red string.
    \newline
        
    \noindent \ref{item:BubblesSymFun3} This is more involved. Set $s_1:=\max\{0,s+\langle i,\mu\rangle\}$ and $s_2:=\max\{s-1,-\langle i,\mu\rangle\}$. Observe that $s_1+s_2=s-1$ and
    \begin{equation}\label{eq:BubbleProofIneq}
        0\leq s_1< \lam{k}_i+\langle i,\mu\rangle,\qquad 0\leq s_2< -\langle i,\mu\rangle.
    \end{equation}
    
    \noindent Consider the following diagram:
    \begin{equation}\label{eq:BubblesSymFun4}
        \begin{tikzpicture}[very thick, baseline=(current bounding box.center)]
                    \node[scale=1.4] at (-2.5,0){$\displaystyle \sum_{r=0}^{\lam{k}_i-1}$};
                    \node at (-1.8,0){$(-1)^r$};
                    \draw[rounded corners] (-1.3,.3) rectangle (.2,-.3) node[midway]{$a_k(e_{i,r})$};
                    \node[scale=1.4] at (.7,0){$\displaystyle \sum$};
                    \draw[postaction={decorate,decoration={markings,
                        mark=at position .5 with {\arrow[scale=1]{<}}}},very thick] (2,0) circle (10pt);
                    \node[scale=.67] at (1.5,.35){$i$};
                    \draw[fill] (2.23,.25) circle[radius=1.5pt];
                    \node[scale=.8] at (2.73,.35){$s_1+a$};
                    \draw[wei] (3.43,.5)--(3.43,-.5) node[at end, below]{$(k)$};;
                    \draw[postaction={decorate,decoration={markings,
                        mark=at position .5 with {\arrow[scale=1]{>}}}},very thick] (4.43,0) circle (10pt);
                    \node[scale=.67] at (3.93,.35){$i$};
                    \draw[fill] (4.63,.25) circle[radius=1.5pt];
                    \node[scale=.8] at (5.13,.35){$s_2+b$};
                    \node[scale=1.4] at (5.63,-.2){$\mu$};
            \end{tikzpicture}
    \end{equation}
    
    \noindent where the second sum is over all $a,b\geq 0$ with $a+b=\lam{k}_i-r-1$. We can rewrite this in ``spade notation'' as
    \begin{equation}
        \begin{tikzpicture}[very thick, baseline=(current bounding box.center)]
                    \node[scale=1.4] at (-3,0){$\displaystyle \sum_{r=0}^{\lam{k}_i -1}$};
                    \node at (-2.2,0){$(-1)^r$};
                    \draw[rounded corners] (-1.7,.3) rectangle (-.2,-.3) node[midway]{$a_k(e_{i,r})$};
                    \node[scale=1.4] at (.7,0){$\displaystyle \sum$};
                    \node[scale=.9] at (.7,-.5){$p+q=s-r$};
                    \draw[postaction={decorate,decoration={markings,
                        mark=at position .5 with {\arrow[scale=1]{<}}}},very thick] (2,0) circle (10pt);
                    \node[scale=.67] at (1.5,.35){$i$};
                    \draw[fill] (2.23,.25) circle[radius=1.5pt];
                    \node[scale=.8] at (2.73,.35){$\spadesuit+p$};
                    \draw[wei] (3.43,.5)--(3.43,-.5)node[at end, below]{$(k)$};;
                    \draw[postaction={decorate,decoration={markings,
                        mark=at position .5 with {\arrow[scale=1]{>}}}},very thick] (4.43,0) circle (10pt);
                    \node[scale=.67] at (3.93,.35){$i$};
                    \draw[fill] (4.63,.25) circle[radius=1.5pt];
                    \node[scale=.8] at (5.13,.35){$\spadesuit+q$};
                    \node[scale=1.4] at (5.63,-.2){$\mu$};
            \end{tikzpicture}
    \end{equation}
    
    \noindent where $p$ and $q$ in the second sum must satisfy
    \begin{equation}
        p\geq s_1 - \lam{k}_i -\langle i, \mu \rangle +1, \qquad   q\geq s_2 - \langle i, \mu \rangle +1.
    \end{equation}
    
    \noindent By \eqref{eq:BubbleProofIneq} the expressions on the right-hand side of both inequalities are less than or equal to zero and bubbles of negative degree are zero, so we can change the second summation to be over $p,q\geq 0$ without changing the value of the expression. Now the first sum is empty for $s<r<\lam{k}_i$ so we change the upper limit to $s$, which yields the left hand side \eqref{eq:BubblesSymFun3}.
    
    Now we claim that \eqref{eq:BubblesSymFun4} is equal to zero. Since $s_1,s_2\geq 0$, all the bubbles in the diagram are real, so we can apply the deformed relation \eqref{eq:DeformedRels3} to get
    \begin{equation*}
        \begin{tikzpicture}[baseline=(current bounding box.center),rotate=90,very thick]
         \draw[<-] (-0.6,0) .. controls ++(0,.4) and ++(0,-.5) .. (-0.0,1);
          \draw[->] (0.0,0) .. controls ++(0,.4) and ++(0,-.5) .. (-0.6,1);
           \draw[] (-0.6,1) .. controls ++(0,.4) and ++(0, .4) .. (-0.0,1);
          \draw[] (0.0,0) .. controls ++(0,-.5) and ++(0,-.5) .. (-0.6,0);
          \filldraw  (-0.05,1.2) circle (1.5pt);
          \filldraw   (-0.1,-.3) circle (1.5pt);
        \node at (0.15,1.2) { $s_1$};
        \node at (-.2,1.5){$i$};
          \node at (0.1,-.4) { $s_2$};
          \draw[wei] (-.8,.5).. controls (-.35,.1) .. (.1,.5) node[at start, below]{$(k)$};
          \node[scale=1.4] at (.2,-.9){$\mu$};
          \node at (-.35,-1.4){$-$};
        \end{tikzpicture}
        \begin{tikzpicture}[baseline=(current bounding box.center),rotate=90,very thick]
             \draw[<-] (-0.6,0) .. controls ++(0,.4) and ++(0,-.5) .. (-0.0,1);
              \draw[->] (0.0,0) .. controls ++(0,.4) and ++(0,-.5) .. (-0.6,1);
               \draw[] (-0.6,1) .. controls ++(0,.4) and ++(0, .4) .. (-0.0,1);
              \draw[] (0.0,0) .. controls ++(0,-.5) and ++(0,-.5) .. (-0.6,0);
              \filldraw  (-0.05,1.2) circle (1.5pt);
              \filldraw   (-0.1,-.3) circle (1.5pt);
            \node at (0.15,1.2) { $s_1$};
            \node at (-.2,1.5){$i$};
              \node at (0.1,-.4) { $s_2$};
              \draw[wei] (-.8,.5).. controls (-.35,.9) .. (.1,.5) node[at start, below]{$(k)$};;
              \node[scale=1.4] at (.2,-.9){$\mu$};
              \node at (-.35,-1.4){$=$};
        \end{tikzpicture}
        \begin{tikzpicture}[baseline=(current bounding box.center),rotate=90,very thick]
             \draw[<-] (-0.6,0) .. controls ++(0,.4) and ++(0,-.5) .. (-0.0,1);
              \draw[->] (0.0,0) .. controls ++(0,.4) and ++(0,-.5) .. (-0.6,1);
               \draw[] (-0.6,1) .. controls ++(0,.4) and ++(0, .4) .. (-0.0,1);
              \draw[] (0.0,0) .. controls ++(0,-.5) and ++(0,-.5) .. (-0.6,0);
              \filldraw  (-0.05,1.2) circle (1.5pt);
              \filldraw   (-0.1,-.3) circle (1.5pt);
            \node at (0.15,1.2) { $s_1$};
            \node at (-.2,1.5){$i$};
              \node at (0.1,-.4) { $s_2$};
              \draw[wei] (-.8,-.8)--(.1,-.8) node[at start, below]{$(k)$};
              \node[scale=1.4] at (.1,-1.2){$\mu$};
              \node at (-.35,-1.4){$-$};
        \end{tikzpicture}
        \begin{tikzpicture}[baseline=(current bounding box.center),rotate=90,very thick]
             \draw[<-] (-0.6,0) .. controls ++(0,.4) and ++(0,-.5) .. (-0.0,1);
              \draw[->] (0.0,0) .. controls ++(0,.4) and ++(0,-.5) .. (-0.6,1);
               \draw[] (-0.6,1) .. controls ++(0,.4) and ++(0, .4) .. (-0.0,1);
              \draw[] (0.0,0) .. controls ++(0,-.5) and ++(0,-.5) .. (-0.6,0);
              \filldraw  (-0.05,1.2) circle (1.5pt);
              \filldraw   (-0.1,-.3) circle (1.5pt);
            \node at (0.15,1.2) { $s_1$};
            \node at (-.2,-.5){$i$};
              \node at (0.1,-.4) { $s_2$};
              \draw[wei] (-.8,1.5)--(.1,1.5) node[at start, below]{$(k)$};
              \node[scale=1.4] at (-1,-.2){$\mu$};
        \end{tikzpicture}
    \end{equation*}
    
    \noindent We claim that both of these ``infinity'' diagrams are zero.
    
    For the rightmost one, consider \eqref{eq:CurlRel} with $m=s_2$ and $\lambda=\mu$. Adding $s_1$ dots to the free string and closing it with a loop on the left yields
    \begin{equation}
        \begin{tikzpicture}[baseline=(current bounding box.center),rotate=90,very thick]
        \raisebox{5mm}{
         \draw[<-] (-0.6,0) .. controls ++(0,.4) and ++(0,-.5) .. (-0.0,1);
          \draw[->] (0.0,0) .. controls ++(0,.4) and ++(0,-.5) .. (-0.6,1);
           \draw[] (-0.6,1) .. controls ++(0,.4) and ++(0, .4) .. (-0.0,1);
          \draw[] (0.0,0) .. controls ++(0,-.5) and ++(0,-.5) .. (-0.6,0);
          \filldraw  (-0.05,1.2) circle (1.5pt);
          \filldraw   (-0.1,-.3) circle (1.5pt);
        \node at (0.15,1.2) { $s_1$};
        \node at (-.2,1.5){$i$};
          \node at (0.1,-.4) { $s_2$};}
          \node[scale=1.4] at (.8,-.9){$\mu$};
        \end{tikzpicture}
      \
        \begin{tikzpicture}[very thick, baseline=(current bounding box.center)]
        \node at (1,0){$=$};
        \node[scale=1.4] at (1.5,0){$-$};
        \node[scale=1.4] at (2,0){$\displaystyle \sum$};
        \node[scale=.9] at (2,-.6){$\alpha + \beta = s;$};
        \draw[postaction={decorate,decoration={markings,
            mark=at position .5 with {\arrow[scale=1]{<}}}},very thick] (3.3,0) circle (10pt);
        \node[scale=.67] at (2.8,.35){$i$};
        \draw[fill] (3.53,.25) circle[radius=1.5pt];
        \node[scale=.8] at (4.03,.35){$\spadesuit+\alpha$};
        \draw[postaction={decorate,decoration={markings,
            mark=at position .5 with {\arrow[scale=1]{>}}}},very thick] (5.3,0) circle (10pt);
        \node[scale=.67] at (4.8,.35){$i$};
        \draw[fill] (5.53,.25) circle[radius=1.5pt];
        \node[scale=.8] at (6.03,.35){$\spadesuit+\beta$};
        \node[scale=1.4] at (6.83,.5){$\mu$};
        \end{tikzpicture}
    \end{equation}
    
    \noindent with the sum over $\alpha\geq 0$ and $\beta\geq s_2+\langle i,\mu\rangle+1$. Again, since $s_2+\langle i,\mu\rangle +1\leq 0$ and bubbles of negative degree are zero, we can just take $\beta\geq 0$. Now this is the homogeneous component of the infinite Grassmannian \eqref{eq:InfGrass} of degree $s>0$ so is zero.
    
    The calculation for the other ``infinity" diagram is similar using the second identity in \cite[Proposition~3.3]{BHLW16}. The claim follows.
\end{proof}

\begin{proof}[Proof of Proposition~\ref{prop:SymFuncTPC}]
    We fix $i\in I$ and show the claim for the copy of $\Sym$ in $\Pi$ indexed by $i$. We will prove \eqref{eq:SymFuncTPC} for all $f$ in a generating set of $\Sym$. Since we have only proved the bubble slides in Lemma~\ref{lem:BubblesSymFun} for real bubbles, the generating set we choose depends on $\lam{k}_i$. There are three cases.
    
    \begin{enumerate}[leftmargin=*]
        \item\label{item:BubbleSlidesCase1} If $\langle i,\mu\rangle\leq -\lam{k}_i$ then \eqref{eq:BubblesSymFun1} holds for all $s>0$. Since
        \begin{equation}
            \delta(e_{i,s})=\sum_{p+q=s}^se_{i,p}\otimes e_{i,q},
        \end{equation}
        
        \noindent the bubble slides hold for all $f=e_{i,s}$, $s\geq 1$.
        
        \item If $\langle i,\mu\rangle\geq 0$ then \eqref{eq:BubblesSymFun2} holds for all $s>0$. Since
        \begin{equation}
            h_{i,s}\otimes 1=\sum_{p+q=s}(-1)^p(1\otimes e_{i,p})\delta(h_{i,q}),
        \end{equation}
        
        \noindent by induction on $s$ the bubble slides hold for all $f=h_{i,s}$, $s\geq 1$.
        
        \item If $-\lam{k}_i < \langle i,\mu\rangle <0$ then \eqref{eq:BubblesSymFun3} holds for $0<s<\lam{k}_i$ and \eqref{eq:BubblesSymFun1} holds for $s\geq \lam{k}_i$. Since
        \begin{equation}
            0=\sum_{p+q+r=s}(-1)^{p+q}(e_{i,p}\otimes e_{i,q})\delta(h_{i,r}),
        \end{equation}
        
        \noindent by induction the bubble slides hold for $f=h_{i,s}$ for $0<s<\lam{k}_i$ and they hold for $f=e_{i,s}$ for $s\geq \lam{k}_i$ as in (\ref{item:BubbleSlidesCase1}).
    \end{enumerate}
\end{proof}

\section{Unfurling 2-representations}\label{sec:unfurling}

Recall from \S\ref{subsec:DefTPC} that $\defTP$ is a deformation of the tensor product categorification $\TP$ over the algebra $\ATP$ which consists of certain symmetric polynomials in a set $Z$ of indeterminates and is regarded as a subset of  $\KK=\overline{\kk(Z)}$ (see \S\ref{subsec:TPWeylMods}). In this section we study $\defTP$ at the generic point of the deformation; that is, we study the idempotent completion $\genTP$ of $\defTPstar\otimes_{\ATP}\KK$.

Define a new symmetric simply-laced Cartan datum with indexing set ${\tilde{I}=I\times Z}$, weight lattice $\tilde{X}=X\times Z$, and symmetric bilinear form obtained from that on $X$ via
\begin{equation}
    ((\mu,z),(\mu',z'))=\delta_{z,z'}(\mu,\mu').
\end{equation}

\noindent for $(\mu,z),(\mu',z')\in \Xtild$. We call $\tilde{I}$ an \emph{unfurling} of $I$ and we identify the corresponding Lie algebra $\gtild$ with a direct sum $\g^{\oplus Z}$ of copies of $\g$ indexed by $z\in Z$.

Morphism spaces in $\genTP$ are (ungraded) $\KK$-vector spaces and the deformed relations \eqref{eq:DeformedRels1}-\eqref{eq:DeformedRels3} in $\defTP$ imply that a dot acting on a black string in a Stendhal pair $S\in\ob(\genTP)$ has spectrum contained in $Z$. In \cite{Web16}, Webster showed that the decomposition of $S$ into generalized eigenspaces indexed by $Z$ induces a 2-representation of the categorified quantum group $\catQGstarKtild$ for $\gtild$ on $\genTP$ where a dot in $\catQGstarKtild$ acts locally nilpotently. Moreover, he showed that there is a 2-natural isomorphism
\begin{equation}\label{eq:UnfurlEquivIntro}
    \eta:\mc{U}^{\tilde{\lambda},*}_{\KK}(\tilde{\mf{g}})\longrightarrow \genTP,
\end{equation}

\noindent between $\genTP$ and a cyclotomic quotient of $\catQGstarKtild$. Since the trace decategorification of cyclotomic quotients is known from \cite{BHLW17}, this allows us to determine the structure of $\Tr(\genTP)$ in \S\ref{Sec:MainTheorem}.

Actually, the setting considered in \cite{Web16} is slightly different; Webster considers a categorification of a tensor product of highest \emph{and lowest} weight modules with a less generic deformation. The main ideas are the same, but since \cite{Web16} is a preprint and we need to understand the structure of the 2-representation of $\catQGstarKtild$ on $\genTP$ in some detail for \S\ref{sec:Flatness} and \S\ref{Sec:MainTheorem}, we give a relatively complete description of its construction.

\begin{remark}
    Since the underlying field and Lie algebra vary in this section, we will take care to include them in our notation. This applies particularly to the categorified quantum group $\catQGk$, cyclotomic quotients $\cyclolamstark$, current algebra $\currk$, and Weyl modules $\Wulamk$. Since we only ever consider (deformed) tensor product categorifications over $\kk$ with respect to $\g$, we still write these as $\TP$ and $\defTP$.
\end{remark}

%

\subsection{Spectrum of dots}

Recall that if $q(w)=\sum_{r=0}^m a_rw^r$ is a polynomial then we will write
\begin{equation}
    \begin{tikzpicture}[very thick,baseline=(current bounding box.center),rounded corners]
            \draw (-2,.75)--(-2,.5);
            \draw (-2.5,0) rectangle (-1.5,.5) node[midway]{$q(y)$};
            \draw[->] (-2,0)--(-2,-.5) node[at end, below]{$i$};
    \end{tikzpicture} \ = \ a_m
    \ \begin{tikzpicture}[very thick,baseline=(current bounding box.center),rounded corners]
        \draw[->] (-2,.75)--(-2,.-.5) node[at end, below]{$i$};
        \draw[fill] (-2,.15) circle (1.5pt) node[right]{$m$};
    \end{tikzpicture} \ + \ \cdots \ + a_1
    \begin{tikzpicture}[very thick,baseline=(current bounding box.center),rounded corners]
        \draw[->] (-2,.75)--(-2,.-.5) node[at end, below]{$i$};
	    \draw[fill] (-2,.15) circle (1.5pt) ;
    \end{tikzpicture}\ + \ a_0
    \begin{tikzpicture}[very thick,baseline=(current bounding box.center),rounded corners]
        \draw[->] (-2,.75)--(-2,.-.5) node[at end, below]{$i$};
    \end{tikzpicture}
\end{equation}

Recall that $\genTP$ is the idempotent completion of ${\defTPstar\otimes_{\ATP}\KK}$. Since the elements of $Z$ are scalars in $\KK$, the deformed relation \eqref{eq:DeformedRels1} factors in $\genTP$ to give
\begin{equation}\label{eq:DeformedRelsFactored}
    \begin{tikzpicture}[very thick,rounded corners, baseline=(current bounding box.center),scale=.5]
        \draw[fill=white] (-3.1,.4) rectangle (-.9,-.4) node[midway]{$r(y)$}; 
        \draw[->] (-2,-.4)--(-2,-1) node[below,at end]{$i$};
        \draw (-2,1)--(-2,.4);
        \draw[wei] (-.4,1)--(-.4,-1) node[below,at end]{$(k)$};
    \end{tikzpicture} \ = \ 
    \begin{tikzpicture}[very thick,baseline=(current bounding box.center),scale=.5]
        \draw[,-] (-2.8,0)  +(0,-1) .. controls (-1.2,0) ..  +(0,1) node[below,at start]{$i$};
        \draw[wei] (-1.2,0)  +(0,-1) .. controls (-2.8,0) ..  +(0,1) node[below,at start]{$(k)$};
    \end{tikzpicture} \hspace{.4in}
\end{equation}

\noindent where $r(w)$ is the product of $(w-z)$ over $z\in Z_i^{(k)}$. This allows us to determine the spectrum of a dot on a black string in $\genTP$.

\begin{lemma}\label{Lem:EvaluesOfDots}
    Take a Stendhal pair $S\in\ob(\genTP)$ and let $B$ denote a black string in $S$ with label $i\in I$. Suppose that the first red string $R$ to the right of $B$ is labelled by $(k)$. Then a dot $y$ acting on $B$ satisfies a polynomial with roots in $Z^{(\leq k)}=\bigcup_{l\leq k}Z^{(l)}$.
\end{lemma}

\begin{proof}
    We proceed by induction on the total number of strings (red or black) to the right of $B$.
    
    First suppose there are no black strings between $B$ and $R$. Let $S'$ denote the Stendhal pair obtained from $S$ by swapping $B$ and $R$ and denote their images in $S'$ by $B'$ and $R'$, respectively. By the inductive hypothesis, a dot on $B'$ satisfies a polynomial $q'(w)$ with roots in $Z^{(\leq k-1)}$. The relations \eqref{eq:UpAndRed} and \eqref{eq:DeformedRelsFactored} imply that in $\genTP(S,S)$,
    \begin{equation*}
          \begin{tikzpicture}[very thick, rounded corners,baseline=(current bounding box.center)]
		\draw[very thick] (-1,-.5)--(-1,-.25) node[below, at start]{$i$};
		\draw (-1.5,-.25) rectangle (-.5,.25) node[midway]{$q'(y)$};
		\draw[->,very thick] (-1,.25)--(-1,.75);
		\draw[wei] (0,-.5)--(0,.75) node[below,at start]{$(k)$};
	\end{tikzpicture} \ = \
	\begin{tikzpicture}[very thick,rounded corners,scale=.65,baseline=(current bounding box.center)]
		 \draw[fill=white!30] (-2,-.375) rectangle (-.5,.375) node[midway]{$q'(y)$};
		 \draw[wei] (-1.2,0)+(0,-1) .. controls (-2.8,0) ..  +(0,1) node[below,at start]{$(k)$};
		\draw[ postaction={decorate,decoration={markings, mark=at
		     	     position .5 with {\arrow[scale=1.95] {}}}}] (-2.7,.9)--(-2.8,1);
		\begin{scope}[on background layer] \draw[very thick, ->] (-2.8,0) +(0,-1)
		    	  .. controls (-.6,0) ..  +(0,1) node[below,at
		   	   start]{$i$}; \end{scope}
	\end{tikzpicture} \ = \ 0, \qquad
	  \begin{tikzpicture}[very thick, rounded corners,baseline=(current bounding box.center)]
		\draw[<-] (-1,-.5)--(-1,0) node[below, at start]{$i$};
		\draw (-1.75,0) rectangle (-.25,.5) node[midway]{$q'(y)r(y)$};
		\draw[] (-1,.5)--(-1,.75);
		\draw[wei] (0,-.5)--(0,.75) node[below,at start]{$(k)$};
	\end{tikzpicture} \ = \
	\begin{tikzpicture}[very thick,rounded corners,scale=.65,baseline=(current bounding box.center)]
		 \draw[fill=white!30] (-2,-.375) rectangle (-.5,.375) node[midway]{$q'(y)$};
		 \draw[wei] (-1.2,0)+(0,-1) .. controls (-2.8,0) ..  +(0,1) node[below,at start]{$(k)$};
		\begin{scope}[on background layer] \draw[very thick, <-] (-2.8,0) +(0,-1)
		    	  .. controls (-.6,0) ..  +(0,1) node[below,at
		   	   start]{$i$}; \end{scope}
	\end{tikzpicture} \ = \ 0,
    \end{equation*}
    
    \noindent where $r(w)$ is the product of $(w-z)$ taken over all $z \in Z_i^{(k)}$. The claim follows.
    
    Now suppose that the string $C$ immediately to the right of $B$ in $S$ is black and labelled by $j\in I$. Let $S'$ denote the Stendhal pair obtained by swapping $B$ and $C$ and denote their images in $S'$ by $B'$ and $C'$ respectively. By the inductive hypothesis, a dot acting on $C$ in $S$ (resp.\ $B'$ in $S'$) satisfies a polynomial $q(w)$ (resp.\ $q'(w)$) with roots in $Z^{(\leq k)}$.
    
    First suppose that $B$ and $C$ have the same orientation. Without loss of generality we may assume they are both oriented down. Recall that the KLR relations \eqref{eq:nilHecke-1} and \eqref{eq:black-bigon} hold on downward-oriented black strings.
    
    If $\langle i,j\rangle=0$ then we can pull $B$ through $C$ and dots commute with the crossing, so $q'(y)$ acts as zero on $B$. If $\langle i,j\rangle =-1$ then
    \begin{equation}
        0 \ = \
	    \begin{tikzpicture}[very thick,rounded corners,scale=.65, baseline=(current bounding box.center)]
			 \draw[fill=white!30] (-2,-.375) rectangle (-.5,.375) node[midway]{$q'(y)$};
			 \draw[<-] (-1.2,0)+(0,-1) .. controls (-2.8,0) ..  +(0,1) node[below,at start]{$j$};
			\begin{scope}[on background layer] \draw[very thick, <-] (-2.8,0) +(0,-1)
		    	  .. controls (-.6,0) ..  +(0,1) node[below,at
		   	   start]{$i$}; \end{scope}
	    \end{tikzpicture} \ = \ t_{ij} \ 
	    \begin{tikzpicture}[very thick, rounded corners,baseline=(current bounding box.center)]
			\draw[<-] (-1,-.5)--(-1,0) node[below, at start]{$i$};
			\draw (-1.5,0) rectangle (-.4,.5) node[midway]{$yq'(y)$};
			\draw[] (-1,.5)--(-1,.75);
			\draw[<-] (0,-.5)--(0,.75) node[below,at start]{$j$};
		\end{tikzpicture} \ + \ t_{ji} \
	    \begin{tikzpicture}[very thick, rounded corners,baseline=(current bounding box.center)]
			\draw[<-] (-1,-.5)--(-1,0) node[below, at start]{$i$};
			\draw (-1.5,0) rectangle (-.5,.5) node[midway]{$q'(y)$};
			\draw[] (-1,.5)--(-1,.75);
			\draw[<-] (0,-.5)--(0,.75) node[below,at start]{$j$};
			\draw[fill] (0,.125) circle (1.5pt);
		\end{tikzpicture} \ 
    \end{equation}
            
    \noindent Since $t_{ij}=-t_{ji}$ by \eqref{eq:ChoiceOfScalars}, repeated applications of this show that
    \begin{equation}
       \begin{tikzpicture}[very thick, rounded corners,baseline=(current bounding box.center)]
		    \draw[<-] (-1,-.5)--(-1,0) node[below, at start]{$i$};
		    \draw (-1.75,0) rectangle (-.25,.5) node[midway]{$q'(y)q(y)$};
		    \draw[] (-1,.5)--(-1,.75);
		    \draw[<-] (0,-.5)--(0,.75) node[below,at start]{$j$};
	    \end{tikzpicture} \ = \
	    \begin{tikzpicture}[very thick, rounded corners,baseline=(current bounding box.center)]
			\draw[<-] (-1,-.5)--(-1,0) node[below, at start]{$i$};
			\draw (-1.5,0) rectangle (-.6,.5) node[midway]{$q'(y)$};
			\draw[] (-1,.5)--(-1,.75);
			\draw[<-] (0,-.5)--(0,0) node[below,at start]{$j$};
			\draw (-.4,0) rectangle (.5,.5) node[midway]{$q(y)$};
			\draw[] (0,.5)--(0,.75);
		\end{tikzpicture} \ = \ 0.
    \end{equation}
    
    \noindent If $i=j$ then let
    \begin{equation}
        \begin{tikzpicture}[very thick,baseline=(current bounding box.center)]
            \node[scale=1.4] at (-1.3,0){$\tau$};
            \node[scale=1.2] at (-.7,0){$=$};
            \draw[<-] (.5,-.5)--(-.5,.5) node[below, at start]{$i$};
            \draw[<-] (-.5,-.5)--(.5,.5) node[below, at start]{$i$};
            \draw[fill] (-.25,-.25) circle (1.5pt);
        \end{tikzpicture} \ - \ 
        \begin{tikzpicture}[very thick,baseline=(current bounding box.center)]
            \draw[<-] (.5,-.5)--(-.5,.5) node[below, at start]{$i$};
            \draw[fill] (.25,-.25) circle (1.5pt);
            \draw[<-] (-.5,-.5)--(.5,.5) node[below, at start]{$i$};
        \end{tikzpicture} \ + \
        \begin{tikzpicture}[very thick,baseline=(current bounding box.center)]
            \draw[<-] (.5,-.5)--(.5,.5) node[below, at start]{$i$};
            \draw[<-] (-.5,-.5)--(-.5,.5) node[below, at start]{$i$};
        \end{tikzpicture}
    \end{equation}
    
    \noindent As in \cite[Lemma~2.1]{BK09}, this squares to the identity on $S$, and conjugating by $\tau$ sends a dot on $C$ to a dot on $B$. So $q(y)$ acts on $B$ as zero. 
    
    Now suppose that $B$ and $C$ have opposite orientations. Assume that $B$ (resp.\ $C$) is oriented up (resp.\ down). The other case is similar. If $i\neq j$ then we can pull $B$ through $C$ and dots commute with the crossing so $q'(y)$ acts as zero on $B$. If $i=j$ then for $r\in \NN$, the relation \eqref{eq:Extendedsl2} and KLR relations imply that
    \begin{equation}
		\begin{tikzpicture}[very thick, baseline=(current bounding box.center)]
			\draw[->] (-.5,-.5)--(-.5,.75) node[at start,below]{$i$};
			\draw[fill] (-.5,.15) circle (1.5pt) node[left]{$r$};
			\draw[<-] (.25,-.5)--(.25,.75) node[at start,below]{$i$};
			\node[scale=1.3] at (.6,.55){$\mu$};
		\end{tikzpicture}\ = \ - 
		\begin{tikzpicture}[very thick,scale=.65,baseline=(current bounding box.center)]
		      \draw[->] (-2.8,0) +(0,-1)
		      .. controls (-1.2,0) ..  +(0,1) node[below,at
		      start]{$i$}; 
			\node[scale=1.3] at (-.8,.9){$\mu$};
			\draw[<-] (-1.2,0)
		      +(0,-1) .. controls (-2.8,0) ..  +(0,1) node[below,at
		      start]{$i$}; \draw[fill] (-1.6,0) circle (2.25pt) node[right]{$r$};
		\end{tikzpicture} \ + \ 
        \begin{tikzpicture}[very thick, baseline=(current bounding box.center)]
		    \node[scale=1.3] at (-1.5,.3){$\displaystyle \sum_{\substack{\alpha+\beta+\gamma \\ =r+\langle i,\mu\rangle - 1}}$};
		    \draw[->] (0,-.6) -- (0,-.3) arc (180:0:0.3) -- (0.6,-.6);
		    \node at (0,-.8) {$i$};
		    \draw[fill] (.6,-.3) circle (1.5pt) node[right] {$\gamma$}; 
		    \draw[fill] (.6,1.2) circle (1.5pt) node[right]{$\alpha$};
   		   \draw[<-] (0,1.5) -- (0,1.2) arc (180:360:0.3) -- (0.6,1.5);
		    \node at (.6,1.7){$i$};
		    \node at (0,.5){$i$};
		    \draw[fill] (.8,.5) circle (1.5pt) node[right]{$\spadesuit + \beta$};
		    \draw[->] (.5,.75) arc(90:450:0.3);
		    \node[scale=1.3] at (2,1.3){$\mu$};
	    \end{tikzpicture}
    \end{equation}
            
    \noindent so if $q'(w)=c_mw^m+\cdots +c_0$ then
    \begin{equation}
        \begin{tikzpicture}[very thick, rounded corners,baseline=(current bounding box.center)]
			\draw[] (-1,-.5)--(-1,-.25) node[below, at start]{$i$};
			\draw (-1.5,-.25) rectangle (-.5,.25) node[midway]{$q'(y)$};
			\draw[->] (-1,.25)--(-1,.75);
			\draw[<-] (0,-.5)--(0,.75) node[below,at start]{$i$};
			\node[scale=1.3] at (.4,.25){$\mu$};
		\end{tikzpicture} \ = \ 
		\begin{tikzpicture}[very thick, baseline=(current bounding box.center)]
            \node[scale=1.3] at (-3.5,0){$\displaystyle \sum_{r=0}^m c_r$};
	        \node[scale=3] at (-2.75,0){$($};
	        \node[scale=1.3] at (-1.55,0){$\displaystyle \sum_{\substack{\alpha+\beta+\gamma \\ =r+\langle i, \mu \rangle - 1}}$};
	    \end{tikzpicture}
	    \begin{tikzpicture}[very thick, baseline=(current bounding box.center),scale=.8]
	        \draw[->] (0,-.6) -- (0,-.3) arc (180:0:0.3) -- (0.6,-.6);
		    \node at (0,-.8) {$i$};
		    \draw[fill] (.6,-.3) circle (1.5pt) node[right,scale=.9] {$\gamma$}; 
		    \draw[fill] (.6,1.2) circle (1.5pt) node[right,scale=.9]{$\alpha$};
   		    \draw[<-] (0,1.5) -- (0,1.2) arc (180:360:0.3) -- (0.6,1.5);
		    \node at (.6,1.7){$i$};
		    \node at (0,.5){$i$};
		    \draw[fill] (.8,.5) circle (1.5pt) node[right,scale=.9]{$\spadesuit + \beta$};
		    \draw[->] (.5,.75) arc(90:450:0.3);
		    \node[scale=1.2] at (1.8,1.3){$\mu$};
	        \node[scale=3] at (2.55,.5){$)$};
	    \end{tikzpicture}
    \end{equation}
            
    \noindent Apply $q(y)$ to the top of the left string on both sides of the equation. On the right hand side we can slide these new dots through the cup and so the whole expression is equal to zero. So $q'(y)q(y)$ acts on $B$ as zero and the claim holds.
\end{proof}

\subsection{2-representations on \texorpdfstring{$\mc{G}$\textsuperscript{\underline{$\lambda$}}}{GLambda}}


Recall that the unfurled Cartan datum has indexing set $\Itild=I\times Z$ and weight lattice $\Xtild=X\times Z$. The associated graph $\tilde{\Gamma}\cong\Gamma\times Z$ inherits an orientation from $\Gamma$. We write $p$ for the projection $\Xtild\to X$. For notational convenience we identify $(\pm\Itild)^m=(\pm I)^m\times Z^m$ for $m\in \NN$, so, for example, $\E_{(-i,z)}$ means $\E_{-(i,z)}$ for $i\in I$ and $z\in Z$.

The definition of the corresponding categorified quantum group $\dot{\mc{U}}_{\KK}(\tilde{\mf{g}})$ depends on a choice of parameters - see \S\ref{subsec:parameters}. We use the geometric choice of scalars:
\begin{equation}
    t_{(i,z),(j,z')}=
    \begin{cases}
        t_{i,j} & \text{if }z=z' \\
        1       & \text{if }z\neq z'
    \end{cases}
\end{equation}

\noindent and the following bubble parameters:
\begin{equation}
    c_{(i,z),\mutild}=c_{i,\mutild_z},
\end{equation}

\noindent where $(i,z),(j,z')\in \Itild$ and $\mutild_z$ is the $z$-component of $\mutild\in\Xtild$.

There is an ungraded 2-representation of $\catQGstarK$ on $\genTP$ arising from the 2-representation of $\catQGstark$ on $\defTPstar$. Sections 3 and 4 of \cite{Web15} establish the following:

\begin{theorem}\cite[Theorem~3.13]{Web15}\label{Thm:2RepUnfurled}
    There is a 2-representation of $\catQGstarKtild$ on $\genTP$ such that dots in $\catQGstarKtild$ act locally nilpotently.
\end{theorem}

\noindent In this subsection we describe how to construct this 2-representation. The reader is referred to \cite[Theorem~3.13]{Web15} for the proof that it is well-defined.


From the 2-representation of $\catQGstarK$ there is a decomposition $\genTP=\bigoplus \genTP(\mu)$ according to the weight $\mu\in X$ of the left-hand region of a diagram and compatible functors from $\catQGstarK(\mu,\nu)$ to the category of functors $\genTP(\mu)\to\genTP(\nu)$ for any ${\mu,\nu\in X}$.

In particular there are functors $\E_i1_{\mu}$ from  $\genTP(\mu)$ to $\genTP(\mu+\alpha_i)$ for all $i\in \pm I$ (recall that $\alpha_{-i}=-\alpha_i$). Define
\begin{equation}
    \E_i=\bigoplus_{\mu\in X} \E_i1_{\mu},\qquad \E=\bigoplus_{i\in I}\E_i,
\end{equation}

\noindent regarded as endofunctors of $\genTP$. For $(i,z)\in\pm\Itild$, let $\E_{(i,z)}$ denote the functor sending $A\in\ob(\genTP)$ to the generalized $z$-eigenspace of a dot acting on $\E_iA$ and given by restriction on morphisms. By Lemma~\ref{Lem:EvaluesOfDots} we can decompose $\E_i=\bigoplus_{z\in Z}\E_{(i,z)}$.

By \cite[Theorem~5.25]{Rou08}, to define a 2-representation of $\catQGstarKtild$ it remains to give a refined weight decomposition
\begin{equation}
    \genTP(\mu)=\bigoplus_{\mutild\in p^{-1}(\mu)}\genTP(\mutild)
\end{equation}

\noindent such that for any $(i,z)\in\pm\Itild$, the restriction $\E_{(i,z)}1_{\mutild}$ of $\E_{(i,z)}$ sends
\begin{equation}
    \genTP(\mutild)\longrightarrow \genTP(\mutild+\alpha_{(i,z)}),
\end{equation}

\noindent and to give a suitable action of the unfurled KLR algebra $\Rtild_m$ on $\E^m$.

\subsubsection{Weight decomposition of $\genTP$}\label{subsec:WeightDecomp}

The definition of the categories $\genTP(\mutild)$ below is equivalent to that in \cite[\S3.2]{Web15}.

Take a weight $\mu\in X$ and a Stendhal pair $S=(\ui,\kappa)\in\ob(\genTP(\mu))$ with $\ui\in (\pm I)^m$ and for $k\in[1,m]$ let $y_k$ denote a dot acting on the black string in $S$ labelled by $\lvert i_k\rvert$. By Lemma~\ref{Lem:EvaluesOfDots}, $y_k$ has generalised eigenvalues in $Z$. For a sequence $\uz=(z_1,\ldots ,z_m)\in Z^m$ let $S_{\uz}\in\ob(\genTP)$ denote the simultaneous generalised eigenspace of $S$ where $y_k$ has eigenvalue $z_k$ (note that the $y_k$ commute). Let $\mc{I}_{\uz}$ denote the 2-sided ideal of $\genTP(S_{\uz},S_{\uz})$ generated by $y_k-z_k$ for $1\leq k\leq m$.

Recall the generating functions $e_j(x)$ for $j\in I$ from \S\ref{subsec:SymmetricFunc} and the isomorphism $b_{\mu}$ with bubbles from \S\ref{subsec:BubblesInU}. The dot $y_k$ acts as $z_k$ in $\genTP(S_{\uz},S_{\uz})\big/ \mc{I}_{\uz}$ so by the bubble slides in Lemma~\ref{Lem:SlidesCatQG} and \eqref{eq:SlidesTPGenFunc1}, $b_{\mu}(e_j(x))$ acts in $\genTP(S_{\uz},S_{\uz})\big/ \mc{I}_{\uz}$ as a rational function in $x$. We define the weight $\mutild\in\Xtild$ of $S_{\uz}$ as the unique weight such that
\begin{equation}\label{eq:UnfurledWeight}
    b_{\mu}(e_j(x))\circ 1_{S_{\uz}} \equiv \prod_{z\in Z}(1+zx)^{
    \langle (j,z),\mutild\rangle}1_{S_{\uz}}\qquad (\text{mod}~\mc{I}_{\uz})
\end{equation}

\noindent for all $j\in I$. From the bubble slides and induction, $\sum_z \langle (j,z),\mutild\rangle=\langle j,\mu\rangle$ and so $\mutild\in p^{-1}(\mu)$.

Let $\genTP(\mutild)$ denote the full, additive subcategory of $\genTP$ generated by objects $S_{\uz}$ with weight $\mutild$ defined as above, and closed under taking direct summands. Since Stendhal pairs generate $\defTPstar$, there is an induced decomposition $\genTP(\mu)=\bigoplus\genTP(\mutild)$ taken over all $\mutild\in p^{-1}(\mu)$. If $(i,z)\in\pm\Itild$, then $\E_{(i,z)}$ denotes the generalized $z$-eigenspace of a dot on $\E_i$, so by Lemma~\ref{Lem:SlidesCatQG} acting by $\E_{(i,z)}$ contributes a factor of $(1+zx)^{\langle j,i\rangle}$ to \eqref{eq:UnfurledWeight}. Thus
\begin{equation}
    \E_{(i,z)}1_{\mutild}:\genTP(\mutild)\longrightarrow\genTP(\mutild+\alpha_{(i,z)})
\end{equation}

\noindent as required.


%


\subsubsection{Actions of KLR algebras}\label{subsec:KLRcompletions}

Take $m\in \NN$. Let $R_m$ denote the KLR algebra associated to $I$ defined over the field $\KK$ (see \S\ref{subsec:KLR}). Since the KLR relations hold on upward oriented strings in $\catQGk$, the 2-representation of $\catQGstark$ on $\genTP$ leads to algebra homomorphisms
\begin{equation}
    R_m\longrightarrow \genTP(\E^mA, \E^mA)
\end{equation}

\noindent which are natural in $A\in\ob(\genTP)$. The idempotent $\varep_{\ui}\in R_m$ acts as projection onto $\E_{\ui}A$ and by Lemma~\ref{Lem:EvaluesOfDots} a dot $y_k\varep_{\ui}$ acts with generalized eigenvalues in $Z$. We can package such actions in a completion $\Rhat_m$ in which we can formally separate the spectrum of the dot into components indexed by $z\in Z$. 


For $N\in \NN$, let $J_m^{(N)}$ be the two-sided ideal of $R_m$ generated by elements
\begin{equation}
    \left(\prod_{z\in Z} (y_k-z)^N\right)\varep_{\ui}
\end{equation}

\noindent as $\ui$ ranges over $I^m$ and $k$ ranges over $[1,m]$. These form a decreasing chain ${J_m^{(1)}\supseteq J_m^{(2)}\supseteq J_m^{(3)}\supseteq\cdots}$ so we can form the completion:
\begin{equation}
    \Rhat_m=\lim_{\longleftarrow} R_m\big/ J_m^{(N)}.
\end{equation}

%

If $A\in\ob(\genTP)$ then each $y_k\varep_{\ui}$ acts on $\E^mA$ with generalised eigenvalues in $Z$, so there is an induced homomorphism from $\Rhat_m$ to $\genTP(\E^mA,\E^mA)$. By abstract Jordan decomposition there is a unique decomposition $\varep_{\ui}=\sum_{\uz\in Z^m}\varep_{(\ui,\uz)}$ into mutually orthogonal idempotents such that if $\uz=(z_1,\ldots ,z_m)\in Z^m$ then $\varep_{(\ui,\uz)}$ projects onto the simultaneous generalized $(z_1,\ldots ,z_m)$-eigenspace for $(y_1\varep_{\ui},\ldots ,y_m\varep_{\ui})$ on $\E_{\ui}A$ and this operation is natural in $A$.

Let $\Rtild_m$ be the KLR algebra for $\Itild$ defined over $\KK$. To distinguish them from elements $R_m$, we write the generators of $\Rtild_m$ as $\{\eptild_{(\ui,\uz)}\mid (\ui,\uz)\in \tilde{I}^m\}$, $\{\ytild_k \mid 1\leq k\leq m\}$, and $\{\psitild_l \mid 1\leq l\leq m-1\}$. Note in particular that $\eptild_{(\ui,\uz)}$ denotes a generator of $\Rtild_m$, whereas $\varep_{(\ui,\uz)}$ denotes the element of $\Rhat_m$ projecting onto a generalized $\uz$-eigenspace for dots.

We can package actions of $\Rtild_m$ in which dots act \emph{nilpotently} into a completion. For $N\in\NN$, let $\tilde{J}^{(N)}_m$ be the two-sided ideal of $\Rtild_m$ generated by elements 
$\ytild_k^N\eptild_{(\ui,\uz)}$ as $(\ui,\uz)$ ranges over $\tilde{I}^m$ and $k$ ranges over $[1,m]$. Let $\Rtildhat_m$ denote the completion of $\Rtild_m$ along these ideals. 

\begin{proposition}\cite[Proposition~3.3]{Web16}\label{prop:KLRCompletions}
    There is an algebra isomorphism ${\Rtildhat_m\to \Rhat_m}$ sending
    \begin{equation}
        \eptild_{(\ui,\uz)}\longmapsto \varep_{(\ui,\uz)} \qquad \ytild_k\eptild_{(\ui,\uz)}\longmapsto (y_k-z_k)\varep_{(\ui,\uz)}
    \end{equation}
    
    \noindent for $(\ui,\uz)\in \Itild^m$ and $k\in[1,m]$.
\end{proposition}

\noindent The formula for the image of $\psitild_l\eptild_{(\ui,\uz)}$ is more involved and we will not need it explicitly.

Since dots in $R_m$ act on $\genTP$ with generalized eigenvalues in $Z$, the action of $R_m$ on $\E^mA$ for an object $A$ induces an action of the completion $\Rhat_m$. The action of $\Rtild_m$ comes from the compositions
\begin{equation}
    \Rtild_m\longrightarrow\Rtildhat_m\longrightarrow \Rhat_m\longrightarrow \genTP(\E^mA,\E^mA)
\end{equation}

\noindent which are natural in $A\in\ob(\genTP)$. For $(\ui,\uz)\in \Itild^m$, $\eptild_{(\ui,\uz)}$ acts as projection onto
\begin{equation}
    \E_{(\ui,\uz)}A=\E_{(i_1,z_1)}\cdots \E_{(i_m,z_m)}A,
\end{equation}

\noindent and each $\ytild_k\varep_{(\ui,\uz)}$ acts nilpotently. This completes our description of the construction of the 2-representation of $\catQGstarKtild$ on $\genTP$.

\subsection{Equivalence with cyclotomic quotient}

In this subsection we analyze the structure of $\genTP$ as a 2-representation of $\catQGstarKtild$.

Let $S\in\ob(\genTP)$ be a Stendhal pair. Take $k\in[1,n]$, $i\in I$, and $z\in Z_i^{(k)}$. Let $S_1$ be the Stendhal pair obtained from $S$ by placing a black string $B_1$ labelled by $i$ immediately to the left of the red $(k)$-string. Define $S_2$ by placing a black string $B_2$ labelled by $i$ and of the same orientation as $B_1$ immediately to the \emph{right} of the red string. For $j=1,2$, let $T_j\in\ob(\genTP)$ denote the generalized $z$-eigenspace for a dot on $B_j$ - a direct summand of $S_j$.

\begin{lemma}\label{Lem:CommutationInGenTP}
    There is an isomorphism $T_1\cong T_2$ and this intertwines a dot acting on $B_1$ with a dot on $B_2$.
\end{lemma}

\begin{proof}
    For $j=1,2$, we can consider dots acting on $B_j$ in $S_j$ as coming from an action of the KLR algebra $R_m$ with $m=1$. Recall the completion $\Rhat_1$ from \S\ref{subsec:KLRcompletions} and the element $\varep_{(i,z)}\in\Rhat_1$. This acts as projection onto $T_j$ and is independent of $j$. We can identify
    \begin{equation}
        \genTP(T_j,T_j)\cong \varep_{(i,z)}\genTP(S_j,S_j)\varep_{(i,z)}.
    \end{equation}
    
    Define Stendhal diagrams in $D_1\in\genTP(S_1,S_2)$ and $D_2\in\genTP(S_2,S_1)$ by
    \begin{equation}
        D_1 \ = \
        \begin{tikzpicture}[very thick,baseline=(current bounding box.center)]
    	    \draw[-] (.5,.5)--(-.5,-.5) node[at end, below]{$i$};
    		\draw[wei] (-.5,.5) -- (.5,-.5) node[at end,below]{$(k)$};
        \end{tikzpicture} \qquad \text{ and } \qquad
        D_2 \ = \ 
        \begin{tikzpicture}[very thick,baseline=(current bounding box.center)]
    	    \draw[-] (.5,-.5)--(-.5,.5) node[at start, below]{$i$};
    		\draw[wei] (-.5,-.5) -- (.5,.5) node[at start,below]{$(k)$};
        \end{tikzpicture}
    \end{equation}
    
    \noindent where the black string is given the same orientation as $B_1$ and $B_2$. Note that since dots can pass through red strings by \eqref{eq:DotsAndRed}, $D_1$ and $D_2$ intertwine the actions of $\Rhat_1$ on $S_1$ and $S_2$ so we can project them to morphisms between $T_1$ and $T_2$ by multiplying by $\varep_{(i,z)}$ on either the top or bottom.
    
    If $B_1$ and $B_2$ are both oriented upward then by relation \eqref{eq:UpAndRed} in $\defTP$, $D_1\cdot D_2=1_{S_2}$ and $D_2\cdot D_1=1_{S_1}$ so $\varep_{(i,z)}D_1$ and $\varep_{(i,z)}D_2$ are mutually inverse isomorphisms $T_1\cong T_2$. Suppose $B_1$ and $B_2$ are oriented downward. Let $r(w)=\prod (w-z')$ where the product is taken over $z'\in Z_i^{(k)}$. By \eqref{eq:DeformedRelsFactored}, separating the red and black strings in $D_1\cdot D_2$ and $D_2\cdot D_1$ introduces a factor of $r(y)$. Since $z\notin Z_i^{(k)}$, $r(w)$ and $(w-z)$ are coprime and so there exists $s\in\Rhat_1$ such that $r(y)s\varep_{(i,z)}=\varep_{(i,z)}$. Thus $s\varep_{(i,z)}D_1$ and $\varep_{(i,z)}D_2$ give mutually inverse isomorphisms $T_1\cong T_2$.
\end{proof}

\begin{corollary}\label{Cor:CommutationInGenTP}
    Let $S\in\ob(\genTP)$ be a Stendhal pair and take $k\in [1,n]$ and $i\in \pm I$. Suppose that there are no black strings to the left of the red $(k+1)$-string in $S$, and let $S^+$ denote the Stendhal pair obtained from $S$ by placing a black string labelled by $\lvert i\rvert$ and oriented according to the sign of $i$ immediately to the right of this red string. Then there is an isomorphism
    \begin{equation}
        \bigoplus_{z\in Z^{(\leq k)}} \E_{(i,z)}S\longrightarrow S^+
    \end{equation}
    
    \noindent intertwining a dot on the left-most black string in $\E_iS$ with a dot on the new string in $S^+$.
\end{corollary}

\begin{proof}
    If $z\in Z^{(\leq k)}$ then by repeated applications of Lemma~\ref{Lem:CommutationInGenTP}, the generalized $z$-eigenspace of a dot on the new string in $S^+$ is isomorphic to $\E_{(i,z)}S$. But by Lemma~\ref{Lem:EvaluesOfDots} the eigenvalues of the dot in $S^+$ are contained in $Z^{(\leq k)}$, so the claim holds.
\end{proof}

Let $\emptyset\in\ob(\genTP)$ denote the trivial Stendhal pair with no black strings. Define a dominant weight for $\gtild$ by
\begin{equation}\label{eq:LamTilde}
    \lamtild=\sum_{i\in I}\sum_{z\in Z_i} \Lambda_{(i,z)},
\end{equation}

\noindent where $Z_i=\bigcup_k Z_i^{(k)}$. By the construction of the weight decomposition \S\ref{subsec:WeightDecomp} and the bubble slides \eqref{eq:SlidesTPGenFunc1} $\emptyset$ has weight $\lamtild$ in $\genTP$. Repeated applications of Corollary~\ref{Cor:CommutationInGenTP} show that $\emptyset$ generates $\genTP$ under the action of $\catQGstarKtild$. Moreover, $\E_{(i,z)}\emptyset=0$ for all $(i,z)\in\Itild$ and dots in $\catQGstarKtild$ act locally nilpotently on $\genTP$ by Theorem~\ref{Thm:2RepUnfurled}, so $\emptyset$ is a highest weight object in $\genTP$.

Recall the definition of cyclotomic quotients from \S\ref{subsec:cyclo}. By \cite[Theorem~4.25]{Rou12}, the undeformed cyclotomic quotient $\cyclostartild$ of $\catQGstarKtild$ of weight $\lamtild$ is the universal highest weight categorification on which dots act nilpotently, so there is a 2-natural transformation
\begin{equation}
    \eta:\cyclostartild\longrightarrow \genTP
\end{equation}

\noindent sending $1_{\lamtild}$ to $\emptyset$ and this is an isomorphism if and only if $\emptyset\ncong 0$ in $\genTP$. Webster shows this by explicitly constructing an inverse to $\eta$:

\begin{theorem}\cite[Lemma~4.8]{Web17}\label{Thm:GenTPEquiv}
    There is a 2-natural isomorphism
    \begin{equation}
        \eta:\cyclostartild\longrightarrow \genTP.
    \end{equation}
    
    \noindent sending $1_{\lamtild}$ to $\emptyset$.
\end{theorem}

\section{Flatness of \texorpdfstring{$\check{\mc{X}}$\textsuperscript{\underline{$\lambda$},*}}{XLambda*}}\label{sec:Flatness}

In this section we use \S\ref{sec:unfurling} to show that $\defTPstar$ is a flat deformation of $\TPstar$ over $\ATP$ and deduce that $K_0(\defTP)\cong K_0(\TP)$ and $\defTP$ is equivalent to a deformed cyclotomic quotient if $\ulambda=(\lambda)$. The results in this section are tangential to the rest of the paper, but we include them since they follow relatively easily from \S\ref{sec:unfurling} and are not recorded elsewhere.

Recall from \S\ref{subsec:TPWeylMods} that $\ATP\leq \KK=\overline{\kk(Z)}$ is a graded local ring with unique simple graded module $\kk$ on which all $z\in Z$ acts as zero. The following is well-known, but we state it precisely for sake of clarity:

\begin{lemma}[Upper semi-continuity of dimension]\label{Lem:UpperSemi}
    Let $M$ be a finitely-generated graded right $\bA_{\ulambda}$-module. Then
    \begin{equation}
        \dim_{\kk}M\otimes_{\bA_{\ulambda}}\kk \geq \dim_{\KK}M\otimes_{\bA_{\ulambda}}\KK
    \end{equation}
    
    \noindent with equality if and only if $M$ is a free graded $\bA_{\ulambda}$-module.
\end{lemma}

\noindent Note that since $\ATP$ is a graded local ring, $M$ is free if and only if it is flat.

So to establish that morphism spaces in $\defTPstar$ are flat $\ATP$-modules it suffices to compare the dimensions of morphism spaces in $\TPstar=\defTPstar\otimes_{\ATP}\kk$ and $\defTPstar\otimes_{\ATP}\KK$, or its idempotent completion $\genTP$. Dimensions of morphism spaces are encoded in the \emph{Euler form} on the Grothendieck group, so the problem essentially reduces to one of bilinear forms on $\g$-modules.

\begin{proposition}\label{Prop:Flatness}
    Morphism spaces in $\defTPstar$ are free over $\ATP$.
\end{proposition}

\begin{proof}
    Recall that $K_0(\TPstar)$ is obtained from $K_0(\TP)$ by tensoring over $\ZZ[q^{\pm 1}]$ with the module $\ZZ$ on which $q$ acts as 1, so if $\UEAdotZg$ is the idempotent form of the integral universal enveloping algebra for $\g$ and $\VZulam$ is a tensor product of integral highest weight modules for $\UEAdotZg$ then there is an isomorphism $K_0(\TPstar)\cong \VZulam$ of $\UEAdotZg$-modules sending the class of a Stendhal pair $S\in\ob(\TPstar)$ to a vector $v_S\in\VZulam$.
    
    The set $\Itild=I\otimes Z$ indexes simple roots for the unfurled Lie algebra $\gtild=\g^{\oplus Z}$. By Theorem~\ref{Thm:GenTPEquiv} there is a 2-natural isomorphism between $\genTP$ and $\cyclostartild$ - the undeformed cyclotomic quotient of $\catQGstarKtild$ of weight
    \begin{equation}
        \lamtild=\sum_{i\in I}\sum_{z\in Z_i} \Lambda_{(i,z)}.
    \end{equation}
    
    \noindent So there is an isomorphism $K_0(\genTP)\cong \VZlamtild$ of $\UEAdotZgtild$-modules sending the class of a Stendhal pair $S\in\ob(\genTP)$ to a vector $\vtild_S\in\VZlamtild$.
    
    For $k\in [1,n]$ there is an inclusion of $\UEAdotZg$-modules
    \begin{equation}\label{eq:InclusionHWMods}
        \VZ{\lam{k}}\longrightarrow \bigotimes_{i\in I}\VZ{\Lambda_i}^{\otimes \lam{k}_i}
    \end{equation}

    \noindent sending the cyclic vector $v_{\lam{k}}$ to the tensor product of the cyclic vectors. Taking the tensor product over all $k$ we get an embedding $\VZulam\into\VZlamtild$, where we use the identification $\gtild=\g^{\oplus Z}$. We claim this sends $v_S\mapsto \vtild_S$ for any Stendhal pair $S$.
    
    This is clear for the trivial Stendhal pair $\emptyset$ with no black strings. Take $k\in [1,n]$ and assume the claim holds for some $S$ with no black strings to the left of the red 
    $(k+1)$-string. Take $i\in \pm I$ and let $S^+$ denote the Stendhal pair obtained from $S$ as in Corollary~\ref{Cor:CommutationInGenTP}; that is, by placing a black string labelled by $\lvert i\rvert$ and oriented according to the sign of $i$ immediately to the right of this red string. We prove the claim for $S^+$. It then follows in general by induction.
    
    By \cite[Theorem~4.38]{Web17}, $v_{S^+}$ is obtained from $v_S\in\VZulam$ by applying $e_i$ to the first $k$ tensor factors. So by the construction of the map \eqref{eq:InclusionHWMods} and the inductive hypothesis, $v_{S^+}$ is mapped to
    \begin{equation}
        \sum_{z\in Z^{(\leq k)}}e_{(i,z)}\vtild_S
    \end{equation}
    
    \noindent in $\VZlamtild$, where $Z^{(\leq k)}=\bigcup_{l\leq k}Z^{(l)}$. But by Corollary~\ref{Cor:CommutationInGenTP}, 
    \begin{equation}
        S^+\cong\bigoplus_{z\in Z^{(\leq k)}}\E_{(i,z)}S
    \end{equation}
    
    \noindent in $\genTP$, so their classes are equal in the Grothendieck group and the claim follows.
    
    Recall that for a $\kk$-linear category $\mc{C}$, the Euler form is the bilinear form on the Grothendieck group $K_0(\mc{C})$ defined by
    \begin{equation}
        \langle [x],[y]\rangle=\dim_{\kk}\mc{C}(x,y)
    \end{equation}
    
    \noindent for $x,y\in \ob(\mc{C})$. By \cite[Theorem~4.38]{Web15}, the isomorphisms $K_0(\genTP)\cong \VZlamtild$ and $K_0(\TPstar)\cong \VZulam$ intertwine the Euler forms with the Shapovalov form and factorwise Shapovalov form, respectively. By the uniqueness of contravariant forms on highest weight modules, the inclusions \eqref{eq:InclusionHWMods}, and consequently the embedding $\VZulam\into \VZlamtild$, are isometries.
    
    Now take Stendhal pairs $S$ and $S'$. Then
    \begin{equation}
        \dim_{\kk}\defTPstar(S,S')=\langle v_S,v_{S'}\rangle=\langle \vtild_S,\vtild_{S'}\rangle=\dim_{\KK}\genTP(S,S'),
    \end{equation}
    
    \noindent so $\defTPstar(S,S')$ is a free $\ATP$-module by the upper semi-continuity of dimension - Lemma~\ref{Lem:UpperSemi}. Since flatness is preserved under taking direct sums and direct summands, this shows that every morphism space in $\defTPstar$ is free.
\end{proof}

Recall that there is a functor from $\defTP$ to $\TP$ given by taking the degree zero component of the projection $\defTPstar\to \TPstar$ defined by tensoring over $\ATP$ with $\kk$.

\begin{corollary}\label{Cor:DeformedGG}
    The functor $\defTP\to\TP$ induces an isomorphism of $\QGdotZ$-modules
    \begin{equation}
        K_0\left( \check{\mc{X}}^{\ulambda}\right)\longrightarrow K_0\left( \mc{X}^{\ulambda}\right)
    \end{equation}
    
    \noindent and so $K_0(\defTP)\cong V_q^{\ZZ}(\ulambda)$.
\end{corollary}

\begin{proof}
    It suffices to show that the functor realises a bijection between isomorphism classes of indecomposables. Since every indecomposable object is a direct summand of (a grading shift of) a Stendhal pair, it suffices to show that if $S$ is a Stendhal pair then idempotents (and isomorphisms between them) lift uniquely under the algebra homomorphism
    \begin{equation}\label{eq:SpecialisationHom}
        \defTP(S,S)\longrightarrow \TP(S,S)
    \end{equation}
    
    The kernel $K$ of this map is the degree zero component of $\defTPstar(S,S)\cdot(\ATP)_+$, where $(\ATP)_+$ is the augmentation ideal of $\ATP$. So $K^N$ is contained in the degree zero component of $\defTPstar(S,S)\cdot(\ATP)_+^N$. Degrees of morphisms between Stendhal pairs are bounded below (this follows from Lemma~\ref{lemma:SpanningSet} which is independent of this proposition) so $K$ is nilpotent ideal of $\defTP(S,S)$. Since idempotents can be lifted uniquely modulo nilpotent ideals, the claim holds.
\end{proof}

If $\ulambda=(\lambda)$ consists of a single dominant weight then we write $\defTPlam$ for $\defTP$. Recall the deformed cyclotomic quotient $\defcyclolam$ from \S\ref{subsec:cyclo}.

\begin{corollary}\label{Cor:EquivDefCyclo}
    If $\lambda\in X^+$ is a dominant weight then there is a graded 2-natural isomorphism
    \begin{equation}
        \xi:\defcyclolam\longrightarrow \defTPlam
    \end{equation}
    
    \noindent sending $1_{\lambda}$ to the trivial Stendhal pair $\emptyset$.
\end{corollary}

\begin{proof}
    Horizontal composition induces a graded right action of $\catQGdotstar(1_{\lambda},1_{\lambda})$ on morphism spaces in $\defcyclolamstar$. Since an upward string at the far right is zero in $\defcyclolamstar$, real clockwise bubbles act as zero. So the bubble isomorphism $b_{\lambda}$ from \S\ref{subsec:BubblesInU} induces a right action of $\Pi$ such that elementary symmetric functions $e_{i,r}$ with $r>\langle i,\lambda\rangle$ act as zero. These generate the kernel of the projection $\Pi\to \bA_{\lambda}$, so the action of $\Pi$ factors through $\bA_{\lambda}$.
    
    Let $\emptyset\in\ob(\defTP)$ denote the trivial Stendhal pair with no black strings. By relation \eqref{eq:UpAndRed}, $\E_i\emptyset=0$ for any $i\in I$, so $\emptyset$ is a highest weight object of $\defTPlam$ of weight $\lambda$. Moreover, $\emptyset$ generates $\defTPlam$ under the action of $\catQGdot$ so by \cite[Theorem~4.25]{Rou12} there is an essentially surjective 2-natural transformation
    \begin{equation}
        \xi:\defcyclolam\longrightarrow \defTPlam
    \end{equation}
    
    \noindent sending $1_{\lambda}$ to $\emptyset$, and $\xi$ is an isomorphism if and only if the induced map
    \begin{equation}
        \defcyclolamstar(1_{\lambda},1_{\lambda})\longrightarrow \defTPlamstar(\emptyset,\emptyset)
    \end{equation}
    
    \noindent is an isomorphism. But the bubble slides \eqref{eq:SlidesTPGenFunc1}-\eqref{eq:SlidesTPGenFunc3} in $\defTPlam$ imply that $\xi$ respects the $\bA_{\lambda}$-actions and both $\defcyclolamstar(1_{\lambda},1_{\lambda})$ and $\defTPlamstar(\emptyset,\emptyset)$ are generated by the identity morphism under $\bA_{\lambda}$, so in particular $\defTPlamstar(\emptyset,\emptyset)\cong\bA_{\lambda}$. The claim follows.
\end{proof}

\section{Trace of \texorpdfstring{$\mc{X}$\textsuperscript{\underline{$\lambda$},$*$}}{XLambda*}}\label{sec:trace}

In this section we begin discussing the trace of $\TPstar$ and $\defTPstar$. After recalling the necessary background, we show that $\TrTPstar$ and $\TrdefTPstar$ are spanned by the classes of diagrams with no crossings between red and black strings. We use this to determine the trace of the unstarred categories $\TP$ and $\defTP$ and find an upper bound for $\dim_{\kk} \TrTPstar$.

\subsection{Trace decategorification}

First we recall the necessary background on trace decategorification and the relevant results from \cite{BHLW17} on the trace of the categorified quantum group and its cyclotomic quotients.

\subsubsection{Definition}

The \emph{trace} or zeroth Hochschild homology of a $\kk$-linear category $\mc{C}$, denoted $\Tr(\mc{C})$, is the $\kk$-vector space defined by:

\begin{equation*}
	\Tr(\mc{C})= \bigg( \bigoplus_{x\in \ob(\mathcal{C})} \mc{C}(x,x) \bigg) \big/\operatorname{span}_{\Bbbk}\{fg-gf\}, 
\end{equation*}

\noindent where the span is over all $f\in\mc{C}(x,y)$ and $g\in\mc{C}(y,x)$ for $x,y\in \ob(\mathcal{C})$. If $f\in\mc{C}(x,x)$ then write $[f]$ for the class of $f$ in $\Tr(\mc{C})$. In diagrammatic categories, applying the trace relation $fg=gf$ can be thought of as cutting a diagram at a horizontal line and swapping the top and bottom parts of the diagram.



The trace and the split Grothendieck group of $\mc{C}$ are related by the $\kk$-linear \emph{Chern character map}:
\begin{align}
    \begin{split}
        h_{\mc{C}}: K^{\kk}_0(\mc{C}) &\longrightarrow \Tr(\mc{C}) \\
        [x]&\longrightarrow [1_x].
    \end{split}
\end{align}

\noindent This is injective under relatively weak hypotheses (see \cite[Proposition 2.4]{BHLW17}), but often fails to be surjective. Unlike the split Grothendieck group, the trace is invariant under taking the Karoubi envelope (c.f. \cite[Proposition 3.2]{BHLZ14}).

If $\mc{C}$ is a graded category then grading shift $\langle 1\rangle$ induces an automorphism $q$ on the trace. This gives $\Tr(\mc{C})$ a $\kk[q^{\pm 1}]$-module structure with respect to which $h_{\mc{C}}$ is a homomorphism of $\kk[q^{\pm 1}]$-modules.

If $\mc{C}^*$ is the corresponding starred category then the $q$-action on $\Tr(\mc{C}^*)$ is trivial, but since $\mc{C}^*$ is enriched over graded vector spaces, the trace $\Tr(\mc{C}^*)$ is a graded vector space also. The corresponding Chern character map $h_{\mc{C}^*}$ is a homomorphism of graded vector spaces where $K_0^{\kk}(\mc{C}^*)$ is concentrated in degree zero.

Since the morphism spaces in $\mc{C}^*$ are larger than in $\mc{C}$, we should expect $\Tr(\mc{C}^*)$ to be richer than $\Tr(\mc{C})$. In fact in this paper the Chern character maps $h_{\mc{C}}$ for the unstarred categories are always isomorphisms, so we focus our attention almost entirely on the traces of starred categories $\Tr(\mc{C}^*)$.

If $\mc{C}$ is enriched over (graded) right $A$-modules for some $\kk$-algebra $A$ then $\Tr(\mc{C})$ is a (graded) right $A$-module. Moreover, the trace commutes with base change; if $\mc{C}$ is a $\kk$-linear category and $A$ is a $\kk$-algebra then
\begin{equation}
    \Tr(\mc{C}\otimes A)\cong\Tr(\mc{C})\otimes A
\end{equation}
\noindent as right $A$-modules.

\subsubsection{Trace decategorification and 2-representations}


The trace $\Tr(\catQGstar)\cong\Tr(\catQGdotstar)$ of the (starred) categorified quantum group is a locally unital graded $\kk$-algebra:
\begin{equation}
    \Tr(\catQGstar)=\bigoplus_{\mu,\nu\in X}1_{\nu}\Tr(\catQGstar(\mu,\nu))1_{\mu}
\end{equation}

\noindent with multiplication given by horizontal composition. A (graded) 2-representation of $\catQGdotstar$ on $\mc{M}=\bigoplus_{\mu\in X} \mc{M}(\mu)$ induces a locally unital (graded) $\Tr(\catQGstar)$-module structure on
\begin{equation}
    \Tr(\mc{M})=\bigoplus_{\mu\in X}1_{\mu}\Tr(\mc{M}(\mu)).
\end{equation}

A 2-natural transformation $\eta$ between 2-representations on $\mc{M}$ and $\mc{N}$ induces a homomorphism of locally unital $\Tr(\catQGstar)$-modules $\Tr(\eta):\Tr(\mc{M})\to \Tr(\mc{N})$. If $\eta$ is a 2-natural isomorphism then $\Tr(\eta)$ is an isomophism. If $\mc{M}$, $\mc{N}$, and $\eta$ are graded then $\Tr(\eta)$ respects this structure.

In particular, $\Tr(\defcyclolamstar)$, $\Tr(\cyclolamstar)$, $\TrdefTPstar$, and $\TrTPstar$ are locally unital graded $\Tr(\catQGstar)$-modules. The action is given by placing a diagram on the left if weights match, and taking the class in the trace. Moreover, since morphism spaces in $\defTPstar$ are enriched over graded right $\ATP$-modules, $\TrdefTPstar$ has the structure of a graded $(\Tr(\catQGstar),\ATP)$-bimodule.


\subsubsection{Results of \cite{BHLW17}}

Recall from \S\ref{Subsec:CurrentDef} that the current algebra $\currdot$ is a locally unital graded $\kk$-algebras and for $\lambda\in X^+$ it has (locally unital) graded modules $\WW(\lambda)$ and $W(\lambda)$ called global and local Weyl modules respectively. Recall, also for $\lambda\in X^+$, the deformed and undeformed cyclotomic quotients $\defcyclolam$ and $\cyclolam$ - graded 2-representations of $\catQGdot$. Finally recall the isomorphisms $b_{\mu}$ ($\mu\in X$) between symmetric functions and bubbles from \S\ref{subsec:BubblesInU}.

The following theorem comprises \cite[Theorems~7.4, 7.5, and 8.4]{BHLW17}:

\begin{theorem}\label{Thm:TraceCatQG}
    There is an isomorphism of locally unital graded $\kk$-algebras
    \begin{equation}
        \rho: \currdot \longrightarrow \Tr(\catQGstar)
    \end{equation}
    
    \noindent sending
    \begin{equation*}\label{Thm: TraceofQuantumGroup}
        (e_i \otimes t^r)1_\mu \longmapsto \left[
    \begin{tikzpicture}[very thick, baseline=(current bounding box.center)]
        \draw[->] (0,0)--(0,1);
        \node at (0,-.2){$i$};
        \node at (.5,.9){$\mu$};
        \draw[fill] (0,.5) circle[radius=2pt];
        \node at (.3,.5){$r$};
    \end{tikzpicture}\right],\quad
    (\xi_i\otimes t^r)1_{\mu}\longrightarrow[b_{\mu}(p_{i,r})],\quad
    (f_i\otimes t^r)1_{\mu}\longmapsto \left[ \begin{tikzpicture}[very thick, baseline=(current bounding box.center)]
        \draw[<-] (0,0)--(0,1);
        \node at (0,-.2){$i$};
        \node at (.5,.9){$\mu$};
        \draw[fill] (0,.5) circle[radius=2pt];
        \node at (.3,.5){$r$};
    \end{tikzpicture} \right]
    \end{equation*}
    
    \noindent for any $\mu\in X$, $i\in I$, and $r\in \NN$. Moreover, for any $\lambda\in X^+$ there are isomorphisms of graded modules
    \begin{equation}
        \WW(\lambda)\longrightarrow \Tr(\defcyclolamstar(\g)) \quad \text{and} \quad W(\lambda)\longrightarrow \Tr(\cyclolamstar)
    \end{equation}
    
    \noindent intertwining $\rho$ and sending the distinguished generator $w_{\lambda}$ to the class $[1_{\lambda}]$ of the empty diagram.
\end{theorem}

\noindent So $\TrTPstar$ is a graded $\currdot$-module and $\TrdefTPstar$ is a graded $(\currdot,\ATP)$-bimodule.

In \cite[Theorem~8.1]{BHLW17} the authors also showed that the Chern character maps $h_{\catQGdot}$, $h_{\defcyclolam}$, and $h_{\cyclolam}$ for the unstarred categories are isomorphisms, thereby determining the traces of these categories. Our proof that $h_{\defTPstar}$ and $h_{\TPstar}$ are isomorphisms follows theirs (see Proposition~\ref{prop:DecatUnstarred}).
 

\subsection{Spanning set}\label{subsec:SpanningSet}

In this subsection we show that $\TrdefTPstar$ and $\TrTPstar$ are spanned by diagrams with no crossings between red and black strings and deduce that the Chern character maps $h_{\TP}$ and $h_{\defTP}$ for the unstarred categories are isomorphisms. 

We begin by constructing a spanning set for morphism spaces in $\defTPstar$ and $\TPstar$ following \cite[\S3.2]{KL10}. Choosing this set carefully allows us to use an inductive argument in the proof of Lemma~\ref{lemma:SpanningSet}.

Let $C$ denote the set of compositions $\nu=(\nu(1),\ldots,\nu(n))$ of length $n$. Write $\lvert\nu\rvert=\sum \nu(j)$ and let $C(m)$ denote the set of $\nu\in C$ with $\lvert \nu\rvert=m$. Each $C(m)$ is a poset under the reverse dominance order: $\nu\leq\nu'$ if
\begin{equation}
    \sum_{j=1}^k \nu(j)\geq\sum_{j=1}^k \nu'(j)
\end{equation}

\noindent for all $k\in [1,n]$. If $S=(\ui,\kappa)$ is a Stendhal pair then define a composition
\begin{equation}
    \nu_{S}=(\kappa(2)-\kappa(1),\ldots ,\kappa(n+1)-\kappa(n)),
\end{equation}

\noindent so $\nu_S(k)$ is the number of black strings between the red strings labelled by $(k)$ and $(k+1)$. Observe that ``right'' crossings move us down in the partial order and ``left'' crossings move us up:
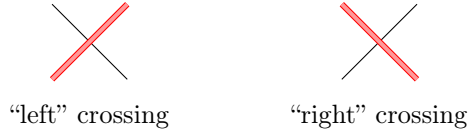
\begin{equation*}
    \begin{tikzpicture}{very thick}
        \draw[-] (.5,-.5)--(-.5,.5);
        \draw[wei] (-.5,-.5)--(.5,.5);
        \node at (0,-1){``left'' crossing};
    \end{tikzpicture}
    \qquad\qquad
    \begin{tikzpicture}{very thick}
        \draw[-] (-.5,-.5)--(.5,.5);
        \draw[wei] (.5,-.5)--(-.5,.5);
        \node at (0,-1){``right'' crossing};
    \end{tikzpicture}
\end{equation*}

Take Stendhal pairs $S=(\ui,\kappa)$ and $S'=(\ui',\kappa')$. Any Stendhal diagram with $S$ as its bottom and $S'$ as its top induces a matching on the disjoint union $\ui\sqcup\ui'$ by pairing elements of $\pm I$ that are connected by strings. Any such matching either connects an occurrence of $\pm i\in\pm I$ in $\ui$ with one in $\ui'$, or occurrences of $i$ and $-i$ that lie either both in $\ui$ or both in $\ui'$.

For each such matching we fix a diagram $D$ that attaches matched elements. We require that:
\begin{enumerate}
    \item $D$ is minimal (no two strings of any color cross more than once);
    \item $D$ has no closed loops (so no bubbles);
    \item there are no dots on any of the strings of $D$;
    \item\label{item:SpanningSet} on a given red string, all ``right'' crossings with black strings occur below all ``left'' crossings:
\end{enumerate}
    
\noindent Fix a point on each black string of $D$ away from intersections. Let $B_{S,S'}$ denote the union over all matchings of diagrams obtained from $D$ by placing an arbitrary number of dots at the chosen points on $D$.

\begin{lemma}\label{lemma:SpanningSet}
    The set $B_{S,S'}$ generates the morphism space $\TPstar(S,S')$ (resp.\ $\defTPstar(S,S')$) as a $\kk$-vector space (resp.\ $\ATP$-module).
\end{lemma}

\begin{proof}
    This follows from the spanning set argument in \cite[\S3.2]{KL10} together with the new bubble slides from Proposition~\ref{prop:SymFuncTPC}.
\end{proof}

\begin{lemma}\label{lem:TraceSpanning}
    The trace $\TrTPstar$ (resp. $\TrdefTPstar$) is generated as a $\kk$-vector space (resp. $\ATP$-module) by the classes of Stendhal diagrams with no crossings between red and black strings. 
\end{lemma}

\begin{proof}
    We show the statement for the undeformed category $\TPstar$; the argument for $\defTPstar$ is identical. It suffices to show that $\TrTPstar$ is generated as an vector space by the classes of diagrams in $B_{S,S}$ with no red-black crossings, taken over all Stendhal pairs $S$. 
    
    Take a Stendhal pair $S$ and suppose $D\in B_{S,S}$ has a red-black crossing. Condition~(\ref{item:SpanningSet}) in the definition of $B_{S,S}$ implies that $D$ factors through a Stendhal pair $S'$ with $\nu_{S'}<\nu_S$. Let $D=D_1\cdot D_2$ be the corresponding factorization of $D$ and set $D'=D_2\cdot D_1\in\TPstar(S',S')$. Then $[D]=[D']$ in the trace, $D'$ lies in the span of $B_{S',S'}$, and $\nu_{S'}<\nu_S$. The claim follows by induction on $C$.
\end{proof}

Together with the results of \cite{BHLW17} this allows us to determine the structure of the trace of the unstarred categories $\defTP$ and $\TP$.

\begin{corollary}\label{prop:DecatUnstarred}
    The Chern character maps
    \begin{equation}
            h_{\TP}:K_0^{\kk}(\TP) \longrightarrow \TrTP,\hspace{.5in}
            h_{\defTP}:K_0^{\kk}(\defTP) \longrightarrow \TrdefTP
    \end{equation}
    
    \noindent are isomorphisms, so by Corollary~\ref{Cor:DeformedGG} both $\TrTP$ and $\TrdefTP$ are isomorphic to $\VZqulam$ with scalars extended to $\kk$.
    %
\end{corollary}

\begin{proof}
    We prove the claim for the undeformed tensor product category $\TP$; the argument for $\defTP$ is identical. The field $\kk$ is perfect since it has characteristic 0, and morphism spaces in $\TP$ are finite-dimensional vector spaces by Lemma~\ref{lemma:SpanningSet} so $h_{\TP}$ is injective by \cite[Proposition 2.4]{BHLW17}. The same argument as Lemma~\ref{lem:TraceSpanning} shows that $\TrTP$ is spanned by classes of Stendhal diagrams of degree zero with no red-black crossings, so by \cite[Corollary~6.3]{BHLW17} $\TrTP$ is spanned by concatenations of idempotents projecting onto divided powers, separated by red strings. Since $\TP$ is idempotent complete, any idempotent lies in the image of $h_{\TP}$ and so it is surjective.
\end{proof}

\subsection{Upper bound on dimension}

In this subsection we show that the dimension of $\TrTPstar$ is bounded above by the dimension of the tensor product $\Wulam$ of local Weyl modules. This will allow us to show that $\TrdefTPstar$ is free over $\ATP$ in Corollary~\ref{Cor:TraceFree} using the upper semi-continuity of dimension.

Recall the cyclotomic quotients from \S\ref{subsec:cyclo}. We wish to relate $\TP$ and the undeformed cyclotomic quotients $\cyclo{\lam{k}}$  and their traces. Consider the product category
\begin{equation}
    \cycloulam=\cyclo{\lam{1}}\times\cdots\times\cyclo{\lam{n}}.
\end{equation}

\noindent Objects (resp. morphisms) in $\cycloulam$ are tuples of objects (resp. morphisms) with composition defined component-wise. The trace of a product category is isomorphic to the product of the traces.

We would like to define a functor sending a tuple of objects in $\cycloulam$ to the obvious horizontal composition in $\TP$ with different components separated by red strings. But in $\TP$ the cyclotomic-type relations \eqref{eq:UpAndRed} and \eqref{eq:UndeformedTPRels1} only allow us to pull black strings through red strings, while in cyclotomic quotients they give zero, so we would need to pass to a filtration (or stratification) of $\TP$ for this functor to be well-defined. In \cite[\S6]{Web17} Webster constructed such a functor $\Strat$. To deal with taking quotients of objects in $\TP$, he defined $\Strat$ as a functor from $\cycloulam$ to the category of representations of $\TP$ - an abelian category in which $\TP$ embeds fully-faithfully. 

To avoid this complication we ignore objects; instead of defining a functor we construct homomorphisms from morphism spaces in $\cycloulamstar$ to quotients of morphism spaces in $\TPstar$. In the trace this gives surjections
\begin{equation}
    \Tr(\cycloulamstar)_{\nu}\longrightarrow \TrTPstar_{\leq \nu}\big/ \TrTPstar_{< \nu},
\end{equation}


\noindent where the $\Tr(\cycloulamstar)_{\nu}$ and $\TrTPstar_{\leq \nu}$ form a grading and a filtration of $\Tr(\cycloulamstar)$ and $\TrTPstar$ respectively, both indexed by the poset $C$ of compositions $\nu$ of length $n$ (see \S\ref{subsec:SpanningSet}). The upper bound on dimension follows from the trace decategorification of cyclotomic quotients.



\begin{corollary}\label{cor:DimSpecialPoint}
    The dimension of $\TrTPstar$ is bounded above:
    \begin{equation}
        \dim_{\kk}\TrTPstar\leq \dim_{\kk}\Wulam.
    \end{equation}
\end{corollary}

\begin{proof}
    For $m\geq0$ let $Q_m\in\ob(\TP)$ be the direct sum of all Stendhal pairs $(\ui,\kappa)$ with $\ui\in (-I)^m$, so there are $m$ black strings and they are all oriented downward. By the categorified commutation relation for $\E_{+i}$ and $\E_{-i}$ and the fact that by \eqref{eq:UpAndRed} upward black strings commute with red strings, the objects $Q_m$ additively generate $\TP$. So \cite[Lemma~2.1]{BHLW17} implies that $\TrTPstar$ is isomorphic to the trace of the full subcategory of $\TPstar$ with objects $Q_m$. Since there are no morphisms between $Q_m$ and $Q_{m'}$ for $m\neq m'$, this implies that the trace decomposes as a direct sum:
    \begin{equation}
        \TrTPstar=\bigoplus_{m\geq 0}\TrTPstar_m,
    \end{equation}
    
    \noindent where $\TrTPstar_m$ is the image of the morphism space $\EndQ$ in the trace.
    
    We can refine this by a filtration indexed by the poset $C(m)$ of compositions of $m$ (c.f. \S\ref{subsec:cyclo}). For $\nu\in C(m)$ let $\EndQ_{\leq \nu}$ and $\EndQ_{< \nu}$ denote the subalgebras of $\EndQ$ spanned by diagrams that factor through a Stendhal pair $S$ with $\nu_S\leq \nu$ and $\nu_S< \nu$ respectively. Let $\TrTPstar_{\leq \nu}$ and $\TrTPstar_{<\nu}$ denote their images in the trace.
    
    The spaces $\TrTPstar_{\leq \nu}$ with $\lvert\nu\rvert=m$ form a filtration of $\TrTPstar_m$. The components of the associated graded space are defined by
    \begin{equation}
        \TrTPstar_{\nu}=\TrTPstar_{\leq \nu}\big/ \TrTPstar_{<\nu}.
    \end{equation}
    
    \noindent for $\nu\in C(m)$. In particular, this implies that
    \begin{equation}
        \dim_{\kk}\TrTPstar_m\geq\sum_{\nu\in C(m)}\dim_{\kk}\TrTPstar_{\nu}.
    \end{equation}
    
    Now consider the undeformed cyclotomic quotients $\cyclo{\lam{k}}$. For $k\in[1,n]$ and $m\geq 0$ let $P^{(k)}_{m}\in\ob(\cyclo{\lam{k}})$ be the direct sum of all $\E_{\ui}1_{\lam{k}}$ with $\ui\in(-I)^{m}$. For a composition $\nu\in C$, the tuple  $P_{\nu}=(P_{\nu(1)}^{(1)},\cdots,P_{\nu(n)}^{(n)})$ is an object in the product category $\cycloulam$.
    
    As with $\TP$, the objects $P_{\nu}$ additively generate $\cycloulam$ and there are no non-zero morphisms between $P_{\nu}$ and $P_{\nu'}$ for $\nu\neq\nu'$, so there is a direct sum decomposition
    \begin{equation}
        \Tr(\cycloulamstar)=\bigoplus_{\nu}\Tr(\cycloulamstar)_{\nu},
    \end{equation}
    
    \noindent where $\Tr(\cycloulamstar)_{\nu}$ is the image of $\EndP$ in the trace.
    
    For a composition $\nu$ with $\lvert\nu\rvert=m$, define an algebra homomorphism
    \begin{equation}
        \EndP\longrightarrow \EndQ_{\leq \nu}\big/ \EndQ_{< \nu}
    \end{equation}
    
    \noindent by sending a tuple $(D_1,\ldots ,D_n)$ of diagrams to the coset of
    \begin{equation}
	    \begin{tikzpicture}[very thick,baseline=(current bounding box.center)]
            \node (a) [inner xsep=10pt, inner ysep=6.6pt,draw] at (2,.25){$D_n$};
            \draw (a.-130) -- +(0,-.33);
            \draw (a.-50) -- +(0,-.33);
            \draw (a.130) -- +(0,.33);
            \draw (a.50) -- +(0,.33);
            \draw[wei] (3,-.5) -- +(0,1.5) node[at        start, below]{$(n)$};
            \draw[wei] (6,-.5) -- +(0,1.5) node[at start, below]{$(n-1)$};
            \node (a) [inner xsep=10pt, inner ysep=6.6pt,draw] at (4.5,.25){$D_{n-1}$};
            \draw (a.-130) -- +(0,-.33);
            \draw (a.-50) -- +(0,-.33);
            \draw (a.130) -- +(0,.33);
            \draw (a.50) -- +(0,.33);
            \draw[wei] (8,-.5) -- +(0,1.5) node[at start,below]{$(2)$}; 
            \node (a) [inner xsep=10pt, inner ysep=6.6pt,draw] at (9.5,.25){$D_1$};
            \draw (a.-130) -- +(0,-.33);
            \draw (a.-50) -- +(0,-.33);
            \draw (a.130) -- +(0,.33);
            \draw (a.50) -- +(0,.33);
            \node at (7,0){$\cdots$}; 
            \draw[wei] (11,-.5) -- +(0,1.5) node[at start,below]{$(1)$}; 
        \end{tikzpicture}
    \end{equation}
    
    \noindent If $D_k$ has a downward string with $\lam{k}_i$ dots (or an upward string) at the far right then by \eqref{eq:UndeformedTPRels1} (resp. \eqref{eq:UpAndRed}) we can pass that string through the red $(k)$-string in $\TP$. So we have an element of $\EndQ_{< \nu}$ and this map respects the cyclotomic relations \eqref{eq:UndefCycloRel}.
    
    It isn't obvious that the homomorphism is well-defined, since it doesn't respect the weights of regions and the defining relations depend on these weights. However, by \cite[Proposition~3.13]{Web17} the algebra $\EndP$ is generated by diagrams whose black strings have no critical points (that is they never turn back on themselves) subject only to KLR relations and the cyclotomic relation. Since these relations are independent of the weight of the region, the map is well-defined.
    
    Passing to the trace we get a linear map
    \begin{equation}
        \Tr(\cycloulamstar)_{\nu}\longrightarrow \TrTPstar_{\nu}
    \end{equation}
    
    \noindent for any composition $\nu$. By Lemma~\ref{lem:TraceSpanning}, $\TrTPstar$ is spanned by diagrams with no red-black crossings, so this is surjective. Since the trace commutes with products, Theorem~\ref{Thm:TraceCatQG} implies that
    \begin{align}
        \begin{split}
            \dim_{\kk}W(\ulambda)   & = \sum_{\nu\in C}\dim_{\kk}\Tr(\cyclostar{\ulambda})_{\nu}\\
                & \geq \sum_{\nu\in C} \dim_{\kk}\TrTPstar_{\nu}\\
                & \geq \sum_{m\geq 0} \dim_{\kk} \TrTPstar_m \\
                & = \dim_{\kk}\TrTPstar
        \end{split}
    \end{align}
    
    \noindent as required.
\end{proof}

\section{Proof of Theorem~\ref{thm:introtheorem}}\label{Sec:MainTheorem}

In this section we prove our main theorem on the trace decategorification of $\TPstar$ and $\defTPstar$. First we introduce some notation to make the statement more precise: for $u_1,\ldots ,u_n\in \currdot$, recursively define elements
\begin{equation}
    w(u_1,\ldots ,u_k)\in \WW(\lam{1})\otimes\cdots \WW(\lam{k})
\end{equation}

\noindent for $k\in[1,n]$ by setting $w(u_1):=u_1w_{\lam{1}}$ and
\begin{equation}\label{eq:wRecursive}
    w(u_1,\ldots ,u_k):=u_k(w(u_1,\ldots ,u_{k-1})\otimes w_{\lam{k}}).
\end{equation}

\noindent So $w(u_1,\ldots ,u_n)$ is a well-defined element of $\WWulam$.

Recall the isomorphism $\rho:\currdot\to\Tr(\catQGstar)$ from Theorem~\ref{Thm:TraceCatQG}.

\begin{theorem}\label{Thm:IsoFromWeyl}
    There are isomorphisms
    \begin{equation}
        \WWulam\longrightarrow \TrdefTPstar, \qquad \Wulam\longrightarrow \TrTPstar
    \end{equation}
    
    \noindent of graded $(\currdot,\bA_{\ulambda})$-bimodules and graded $\currdot$-modules respectively, sending $w(u_1,\ldots ,u_n)$ to the class of the diagram
    \begin{equation}\label{eq:ImageMainIso}
	    \begin{tikzpicture}[very thick,baseline=(current bounding box.center)]
            \node (a) [inner xsep=10pt, inner ysep=6.6pt,draw] at (2,.25){$\rho(u_n)$};
            \draw (a.-130) -- +(0,-.33);
            \draw (a.-50) -- +(0,-.33);
            \draw (a.130) -- +(0,.33);
            \draw (a.50) -- +(0,.33);
            \draw[wei] (3,-.5) -- +(0,1.5) node[at        start, below]{$(n)$};
            \draw[wei] (6,-.5) -- +(0,1.5) node[at start, below]{$(n-1)$};
            \node (a) [inner xsep=10pt, inner ysep=6.6pt,draw] at (4.5,.25){$\rho(u_{n-1})$};
            \draw (a.-130) -- +(0,-.33);
            \draw (a.-50) -- +(0,-.33);
            \draw (a.130) -- +(0,.33);
            \draw (a.50) -- +(0,.33);
            \draw[wei] (8,-.5) -- +(0,1.5) node[at start,below]{$(2)$}; 
            \node (a) [inner xsep=10pt, inner ysep=6.6pt,draw] at (9.5,.25){$\rho(u_1)$};
            \draw (a.-130) -- +(0,-.33);
            \draw (a.-50) -- +(0,-.33);
            \draw (a.130) -- +(0,.33);
            \draw (a.50) -- +(0,.33);
            \node at (7,0){$\cdots$}; 
            \draw[wei] (11,-.5) -- +(0,1.5) node[at start,below]{$(1)$}; 
        \end{tikzpicture}
    \end{equation}
      
    \noindent for any $u_1,\ldots ,u_n\in \currdot$.
\end{theorem}

In \S\ref{subsec:HomFromVerma} we construct a $\currdot$-module homomorphism from the tensor product of Verma-like modules $M(\ulambda)$ to $\TrdefTPstar$ and show it is surjective and compatible with the right actions by symmetric functions. Then in \S\ref{subsec:GenericIsoWeyl} we use the results of \S\ref{sec:unfurling} to show that this descends to an isomorphism at the generic point:
\begin{equation}
    \WWulam\otimes_{\ATP}\KK\longrightarrow \TrdefTPstar\otimes_{\ATP}\KK.
\end{equation}

\noindent Finally we use the upper bound on dimension from Proposition~\ref{cor:DimSpecialPoint} and upper semi-continuity of dimension to establish the theorem.

\subsection{Homomorphism from \texorpdfstring{$M($\underline{$\lambda$}$)$}{MLambda}}\label{subsec:HomFromVerma}

Recall the Verma-like modules $M(\lam{k})$ from \S\ref{subsec:Verma}. For $u_1,\ldots ,u_n\in\currdot$, recursively define elements
\begin{equation}
    m(u_1,\ldots ,u_k)\in M(\lam{1})\otimes\cdots\otimes M(\lam{k})
\end{equation}

\noindent by analogy with \eqref{eq:wRecursive}, so $m(u_1,\ldots ,u_n)\in M(\ulambda)$.

\begin{lemma}\label{lemma:map}
    There is a unique homomorphism of graded $\currdot$-modules
    \begin{equation}
        \varphi:M(\ulambda)\longrightarrow \Tr(\check{\mc{X}}^{\ulambda,*})
    \end{equation}
    
    \noindent sending $m(u_1,\ldots ,u_n)$ to the class of the diagram \eqref{eq:ImageMainIso} for any $u_1,\ldots,u_n\in \currdot$.
    %
%
\end{lemma}

\begin{proof}
    We proceed by induction on $n$; the number of tensor factors. If $n=1$ then by Corollary~\ref{Cor:EquivDefCyclo} there is a 2-natural isomorphism between $\defTPlamstar$ and the deformed cyclotomic quotient $\defTPlamstar$, so the claim follows from Theorem~\ref{Thm:TraceCatQG}.
    
	
	Assume the claim holds for $n\in \NN$ and take $\mu\in X^+$. Recall the Lie algebra $\mf{p}=\mf{h}\oplus\mf{n}[t]$ and the one dimensional $\mf{p}$-module $\kk_{\mu}$ from \S\ref{subsec:Verma}. The $\dot{U}(\mf{g}[t])$-module map from the inductive hypothesis induces a $\mf{p}$-module map $M(\ulambda)\to\Tr(\defTPstar)$. There is a map of graded $\kk$-vector spaces from $\Tr(\check{\mc{X}}^{\ulambda,*})$ to $\Tr(\check{\mc{X}}^{(\ulambda,\mu),*})$ induced by placing the new red string at the far left of the diagram. By \eqref{LabellingAcrossStrands} this increases the weight by $\mu$ and so there is an induced $\mf{h}$-module map from $M(\ulambda)\otimes \kk_{\mu}$ to $\Tr(\check{\mc{X}}^{(\ulambda,\mu),*})$. In fact this is a $\mf{p}$-module map by \eqref{eq:DotsAndRed} and \eqref{eq:UpAndRed}. By Frobenius reciprocity and the tensor identity this yields a $\mf{g}[t]$-module homomorphism $M(\ulambda,\mu)\to \Tr(\check{\mc{X}}^{(\ulambda,\mu),*})$ of the desired form. Uniqueness is clear.
\end{proof}

Recall that $M(\ulambda)$ carries a right action by $\Pi^{\otimes n}$. We can consider $\Tr(\check{\mc{X}}^{\ulambda,*})$ as a right $\Pi^{\otimes n}$-module via projection $\bigotimes a_k:\Pi^{\otimes n}\to \bA_{\ulambda}$.

\begin{proposition}\label{prop:PhiBimod}
    The map $\varphi$ above is a surjective homomorphism of $(\currdot,\Pi^{\otimes n})$-bimodules.
\end{proposition}

\begin{proof}
    Take $u_1,\ldots ,u_n\in\dot{U}(\mf{g}[t])$. By the definition of the action of $\Pi^{\otimes n}$ and the coproduct on $U(\mf{g}[t])$, a power sum symmetric function $p_{i,r}$ in the $k$th copy of $\Pi$ in $\Pi^{\otimes n}$ sends $m(u_1,\ldots ,u_n)$ to
    \begin{equation}
        m(u_1,\ldots ,u_k(\xi_i\otimes t^r),\ldots ,u_n)-m(u_1,\ldots ,(\xi_i\otimes t^r)u_{k-1},\ldots ,u_n)
    \end{equation}
    
    \noindent By Theorem~\ref{Thm:TraceCatQG} and the bubbles slides in Proposition~\ref{prop:SymFuncTPC}, this is mapped under $\varphi$ to
    \begin{equation}
        \varphi(m(u_1,\ldots,u_n))\cdot a_k(p_{i,r}).
    \end{equation}
    
    \noindent Surjectivity follows from the spanning set Lemma~\ref{lemma:SpanningSet} and surjectivity of $\rho$.
\end{proof}

\subsection{Unfurling and the trace}\label{subsec:GenericIsoWeyl}

In this subsection we apply the trace decategorification results from \cite{BHLW17} to the unfurled 2-representation on $\genTP$ of \S\ref{sec:unfurling}. We use this to determine the structure of $\TrdefTPstar$ at the generic point and apply upper semi-continuity to show that $\varphi$ descends to an isomorphism from $\WWulamk$. Since the underlying field and Cartan datum vary in this section we take care to include them in notation.

\subsubsection{The trace of $\genTP$}\label{subsec:TraceGenTP}

Recall from \S\ref{sec:unfurling} that the set $\Itild=I\times Z$ indexes simple roots for the unfurled Lie algebra $\gtild=\g^{\oplus Z}$ and there is a 2-representation of the corresponding categorified quantum group $\catQGstarKtild$ on the idempotent completion $\genTP$ of $\defTPstar\otimes_{\ATP}\KK$, where $\KK=\overline{\kk(Z)}$. By Theorem~\ref{Thm:TraceCatQG} there is an induced $\currdottild$-action on $\TrgenTP$.

We need to be a little careful with interpreting this action: the isomorphism $\rho$ is expressed diagrammatically, but the diagrammatics in $\catQGstarKtild$ doesn't match the diagrammatics of $\genTP$. In particular, for $(i,z)\in\pm\Itild$ and $r\in\NN$, $e_{(i,z)}\otimes t^r$ acts on the class of a morphism in $\TrgenTP$ by applying the functor $\E_{(i,z)}$ and acting by $\ytild\eptild_{(i,z)}\in \Rtild_1$ - the KLR algebra for $\gtild$. By Proposition~\ref{prop:KLRCompletions}, this is the same as applying $\E_i$ and acting by $(y-z)^r\varep_{(i,z)}\in R_1$. So the action of $\currdottild$ on $\TrgenTP$ is ``twisted'' according to $z\in Z$.

We wish to modify the action of $\currdottild$ on $\TrgenTP$ to remove this twist.  For $z\in  Z$, the current algebra $\currK$ of $\g$ has an automorphism $\sigma_z$ given by
\begin{equation}\label{eq:TwistedAction}
    \sigma_z(x\otimes t^r)=x\otimes (t+z)^r
\end{equation}

\noindent for $x\in \g$ and $r\geq 0$. For any $\currK$-module $M$ we can define a \emph{$z$-twisted action} on $M$ by
\begin{equation}
    u*m=\sigma_z(u)m
\end{equation}

\noindent for $u\in\currK$ and $m\in M$. Let $\sigma_Z$ denote the automorphism of $\currtild$ which restricts to $\sigma_z$ on the copy of $\g[t]$ indexed by $z$ under the identification $\gtild[t]=\g[t]^{\oplus Z}$. We define the \emph{$Z$-twisted action} on a $\currtild$-module as above.

Let $\TrZgenTP$ denote the trace of $\genTP$ under the $Z$-twisted action of $\currdottild$. Now $e_{(i,z)}\otimes t^r$ acts on the class of a morphism by applying $\E_i$ and acting by $y^r\varep_{(i,z)}$, or in diagrammatic terms, by adding a black $i$-string at the left, projecting to the generalized $z$-eigenspace of a dot, and applying $r$ dots.

Recall from Theorem~\ref{Thm:GenTPEquiv} that there is a 2-natural isomorphism
\begin{equation}
    \eta:\cyclostartild\longrightarrow \genTP
\end{equation}

\noindent where $\cyclostartild$ is the cyclotomic quotient of $\catQGstarKtild$ of weight

\begin{equation}
    \lamtild=\sum_{i\in I}\sum_{z\in Z_i} \Lambda_{(i,z)}.
\end{equation}

\noindent By Theorem~\ref{Thm:TraceCatQG}, the trace of $\cyclostartild$ is isomorphic to the local Weyl module for $\currdottild$ of weight $\lamtild$, so taking the trace of $\eta$ and twisting by $Z$ yields an isomorphim of $\currdottild$-modules:
\begin{equation}\label{eq:TrZeta}
    \Tr^{Z}(\eta):\WZtild\longrightarrow \TrZgenTP,
\end{equation}

\noindent where $\WZtild$ denotes the local Weyl module under the $Z$-twisted action.

\subsubsection{Application to $\TrdefTPstar$}

By \cite[Corollary~6]{CFK10}, there is an isomorphism of $\currdotK$-modules
\begin{equation}\label{eq:WeylGenericIso}
    \WWulamk\otimes_{\ATP}\KK\longrightarrow \bigotimes_{i\in I}\bigotimes_{z\in Z_i}W^z_{\KK}(\Lambda_i),
\end{equation}
    
\noindent where $W^z_{\KK}(\Lambda_i)$ denotes the $z$-twisted $\currK$-module structure on the local Weyl module $W_{\KK}(\Lambda_i)$. The module on the right coincides with $\WZtild$ using the identification $\gtild=\g^{\oplus Z}$. Moreover, since the trace is invariant under idempotent completion, there are $\KK$-linear isomorphisms
\begin{equation}\label{eq:GenericTP}
    \TrZgenTP\longrightarrow \TrdefTPstar\otimes_{\ATP}\KK\longrightarrow \TrdefTPstar\otimes_{\ATP}\KK.
\end{equation}

Now we can relate $\WWulamk$ and $\TrdefTPstar$ at the generic point:
    
\begin{proposition}\label{Prop:IsoGenericPoint}
    The composition
    \begin{equation}\label{eq:CompositionGeneric}
        \WWulamk\otimes_{\ATP}\KK\longrightarrow \WZtild\longrightarrow \TrZgenTP\longrightarrow \TrdefTPstar\otimes_{\ATP}\KK
    \end{equation}
    
    \noindent of the $\KK$-linear isomorphisms \eqref{eq:WeylGenericIso}, \eqref{eq:TrZeta}, and \eqref{eq:GenericTP} sends $w(u_1,\ldots ,u_n)$ to the class of the diagram \eqref{eq:ImageMainIso} for any $u_1,\ldots,u_n\in\currdotk$.
\end{proposition}

\begin{proof}
    The claim is immediate if all $u_k=1$. Assume the statement holds for $w:=w(u_1,\ldots, u_k,1,\ldots ,1)$ for some $k\in[1,n]$ and $u_1,\ldots,u_k\in\currdotk$ and take $i\in \pm I$ and $r\in \NN$. We will deduce the claim for $w^+:=(u_1,\ldots ,(e_i\otimes t^r)u_k,1,\ldots,1)$. The proposition follows by induction.
    
    By assumption, the composition \eqref{eq:CompositionGeneric} sends $w$ to the class of an endomorphism $D$ of a Stendhal pair $S$ with no black strings to the left of the red $(k+1)$-string. Let $S^+$ be the Stendhal pair obtained from $S$ as in Corollary~\ref{Cor:CommutationInGenTP}; that is, by placing a black string labelled by $\lvert i\rvert$ and oriented according to the sign of $i$ immediately to the right of this red string. Let $D^+$ be the endomorphism of $S^+$ obtained from $D$ by inserting this new string and acting on it with $r$ dots. We need to show that $w^+\mapsto [D^+]$.
    
    We obtain $w^+$ from $w\in \WWulamk$ by applying $e_i\otimes t^r$ to the first $k$ tensor factors. So in $\WZtild$,
    \begin{equation}
        w^+=\sum_{z\in Z^{(\leq k)}}(e_{(i,z)}\otimes t^r)* w,
    \end{equation}
    
    \noindent under the $Z$-twisted action, where $Z^{(\leq k)}=\bigcup_{l\leq k}Z^{(l)}$.
    
    Since $\Tr^Z(\eta)$ is an isomorphism of $\currtild$-modules, it sends $w^+$ to
    \begin{equation}\label{eq:ImageOfw'}
        \sum_{z\in Z^{(\leq k)}}(e_{(i,z)}\otimes t^r)* [D].
    \end{equation}
    
    \noindent By the discussion in \S\ref{subsec:TraceGenTP} this is the class of $r$ dots acting on the left-most string in
    \begin{equation}
        \sum_{z\in Z^{(\leq k)}}\E_{(i,z)}D.
    \end{equation}
    
    \noindent The isomorphism in Corollary~\ref{Cor:CommutationInGenTP} sends this to $D^+$ and so they are equal in the trace. The claim follows.
\end{proof}

\begin{corollary}\label{Cor:TraceFree}
    The trace $\TrdefTPstar$ is a free graded right $\ATP$-module.
\end{corollary}

\begin{proof}
    By Proposition~\ref{Prop:IsoGenericPoint},
    \begin{equation}
        \dim_{\KK} \TrdefTPstar\otimes_{\ATP}\KK= \dim_{\KK}\WWulamk\otimes_{\ATP}\KK.
    \end{equation}
    
    \noindent If $\kk$ denotes the unique simple graded $\ATP$-module then $\Wulamk\cong \WWulamk\otimes_{\ATP}\kk$ and $\TrTPstar\cong \TrdefTPstar\otimes_{\ATP}\kk$, so Corollary~\ref{cor:DimSpecialPoint} implies that
    \begin{equation}
        \dim_{\kk} \TrdefTPstar\otimes_{\ATP}\kk\leq \dim_{\kk}\WWulamk\otimes_{\ATP}\kk.
    \end{equation}
    
    By Theorem~\ref{thm:WeylFree}, $\WWulamk$ is a free graded right $\ATP$-module so the right hand terms in the above equations are equal by upper semi-continuity of dimension (see Lemma~\ref{Lem:UpperSemi}). Now the claim follows by appealing to upper semi-continuity again.
\end{proof}

Now we can prove the main theorem.

\begin{proof}[Proof of Theorem~\ref{Thm:IsoFromWeyl}]
    It suffices to construct the isomorphism from $\WWulam$. There are surjective homomorphisms of graded $(\currdot,\bA_{\ulambda})$-bimodules
    \begin{equation}\label{eq:TwoBimodMaps}
        \begin{tikzcd} 
        \WW(\ulambda)& M(\ulambda)\otimes _{\Pi^{\otimes n}}\bA_{\ulambda} \ar[l, "\psi", swap] \ar[r, "\varphi"] & \Tr(\check{\mc{X}}^{\ulambda,*})
        \end{tikzcd}
    \end{equation}

    \noindent coming from the natural projection from $M(\ulambda)$ onto $\WW(\ulambda)$ and the map in Proposition~\ref{prop:PhiBimod} respectively. By Theorem~\ref{thm:WeylFree} and Corollary~\ref{Cor:TraceFree}, all three of these $\bA_{\ulambda}$-modules are graded free, so $\Ker(\varphi)$ and $\Ker(\psi)$ are also free (they are flat and so free since $\bA_{\ulambda}$ is graded local).
    
    By Proposition~\ref{Prop:IsoGenericPoint}, these factor through an isomorphism at the generic point:
    \begin{equation}
    \begin{tikzcd} 
        & M(\ulambda)\otimes _{\Pi^{\otimes n}}\KK \ar[ld, "\psi\otimes \KK", swap] \ar[rd, "\varphi\otimes \KK"] & \\
        \WW(\ulambda)\otimes_{\bA_{\ulambda}}\KK\ar[rr] & & \Tr(\check{\mc{X}}^{\ulambda,*})\otimes_{\bA_{\ulambda}}\KK
        \end{tikzcd}
    \end{equation}
    
    \noindent So $\Ker(\varphi\otimes \KK)=\Ker(\psi\otimes \KK)$. Freeness implies that $\Ker(\varphi)=\Ker(\psi)$ and so $\varphi$ factors through an isomorphism from $\WW(\ulambda)$ as claimed.
\end{proof}

\end{document}